\documentclass[11pt, twoside, leqno]{article}

\usepackage{amssymb}
\usepackage{amsmath}
\usepackage{amsthm}
\usepackage{color}
\usepackage{mathrsfs}
\usepackage{txfonts}

\usepackage{indentfirst}

\allowdisplaybreaks

\def\red{\color{red}}

\def\rr{{\mathbb{R}}}
\def\rn{{\mathbb{R}^n}}
\def\zz{{\mathbb Z}}
\def\nn{{\mathbb N}}
\def\cp{{\mathbb P}}
\def\cs{{\mathcal S}}

\def\fz{\infty }
\def\az{\alpha}
\def\bz{\beta}
\def\dz{\delta}
\def\lz{\lambda}
\def\vaz{\varepsilon}

\def\lf{\left}
\def\r{\right}
\def\hs{\hspace{0.25cm}}
\def\ls{\lesssim}
\def\gs{\gtrsim}
\def\noz{\nonumber}
\def\wz{\widetilde}
\def\com{\complement}

\def\supp{\mathop\mathrm{\,supp\,}}
\def\esup{\mathop\mathrm{\,ess\,sup\,}}
\def\gfz{\genfrac{}{}{0pt}{}}

\def\va{\vec{a}}
\def\vp{\vec{p}}
\def\HL{M_{{\rm HL}}}
\def\vh{{H_A^{\vp}(\rn)}}
\def\vah{{H_A^{\vp,r,s}(\rn)}}
\def\vAh{{H_{A}^{p}(\rn)}}
\def\vahfz{{H_A^{\vp,\fz,s}(\rn)}}
\def\vfah{{H_{A,\,{\rm fin}}^{\vp,r,s}(\rn)}}
\def\vfahfz{{H_{A,\,{\rm fin}}^{\vp,\fz,s}(\rn)}}
\def\lq{\mathcal{L}_{\vp,\,q,\,s}^{A}(\rn)}
\def\lr{\mathcal{L}_{\vp,\,r',\,s}^{A}(\rn)}

\def\lv{{L^{\vp}(\rn)}}

\def\Qik{{Q_i^k}}
\def\lik{{\lz_i^k}}
\def\aik{{a_i^k}}
\def\xik{{x_i^k}}

\newtheorem{theorem}{Theorem}[section]
\newtheorem{lemma}[theorem]{Lemma}
\newtheorem{corollary}[theorem]{Corollary}
\newtheorem{proposition}[theorem]{Proposition}

\theoremstyle{definition}
\newtheorem{remark}[theorem]{Remark}
\newtheorem{definition}[theorem]{Definition}
\renewcommand{\appendix}{\par
   \setcounter{section}{0}%
   \setcounter{subsection}{0}%
   \setcounter{subsubsection}{0}%
   \gdef\thesection{\@Alph\c@section}%
   \gdef\thesubsection{\@Alph\c@section.\@arabic\c@subsection}%
   \gdef\theHsection{\@Alph\c@section.}%
   \gdef\theHsubsection{\@Alph\c@section.\@arabic\c@subsection}%
   \csname appendixmore\endcsname
 }

\numberwithin{equation}{section}

\begin{document}

\arraycolsep=1pt

\title{\bf\Large Real-Variable Characterizations of
New Anisotropic Mixed-Norm Hardy Spaces\footnotetext{\hspace{-0.35cm} 2010 {\it
Mathematics Subject Classification}. Primary 42B35;
Secondary 42B30, 46E30, 42B25, 42B20.
\endgraf {\it Key words and phrases.}
expansive matrix, (mixed-norm) Hardy space,
(mixed-norm) Campanato space, maximal function, (finite) atom,
Littlewood--Paley function, duality, Calder\'{o}n--Zygmund operator.
\endgraf This project is supported by the National
Natural Science Foundation of China
(Grant Nos.~11971058, 11761131002, 11671185 and 11871100).
Jun Liu is also supported by the Scientific Research Foundation
of China University of Mining and Technology (Grant No.~102519054).}}
\author{Long Huang, Jun Liu, Dachun Yang\footnote{Corresponding author/{\red
October 10
, 2019}/Final version.}\ \ and Wen Yuan}
\date{}
\maketitle

\vspace{-0.9cm}

\begin{center}
\begin{minipage}{13cm}
{\small {\bf Abstract}\quad
Let $\vec{p}\in(0,\infty)^n$ and $A$ be a general expansive
matrix on $\mathbb{R}^n$. In this article, via the non-tangential
grand maximal function, the authors first
introduce the anisotropic mixed-norm
Hardy spaces $H_A^{\vec{p}}(\mathbb{R}^n)$ associated with $A$
and then establish their radial or non-tangential
maximal function characterizations. Moreover, the authors characterize
$H_A^{\vec{p}}(\mathbb{R}^n)$, respectively, by means of atoms,
finite atoms, Lusin area functions, Littlewood--Paley $g$-functions
or $g_{\lambda}^\ast$-functions via first establishing an anisotropic
Fefferman--Stein vector-valued inequality on the
mixed-norm Lebesgue space $L^{\vec{p}}(\mathbb{R}^n)$.
In addition, the authors also obtain the duality between
$H_A^{\vec{p}}(\mathbb{R}^n)$ and the anisotropic mixed-norm Campanato spaces.
As applications, the authors establish a criterion on the
boundedness of sublinear operators from $H_A^{\vec{p}}(\mathbb{R}^n)$ into
a quasi-Banach space. Applying this criterion, the authors then
obtain the boundedness of anisotropic convolutional $\delta$-type
and non-convolutional $\beta$-order Calder\'{o}n--Zygmund
operators from $H_A^{\vec{p}}(\mathbb{R}^n)$ to itself
[or to $L^{\vec{p}}(\mathbb{R}^n)$]. As a corollary,
the boundedness of anisotropic convolutional $\delta$-type Calder\'on--Zygmund
operators on the mixed-norm Lebesgue space $L^{\vec{p}}(\mathbb{R}^n)$ with
$\vec{p}\in(1,\infty)^n$ is also presented.}
\end{minipage}
\end{center}

\vspace{0.2cm}

\section{Introduction\label{s1}}

The main purpose of this article is to introduce a new kind of
anisotropic mixed-norm Hardy spaces on $\rn$.
It is well known that the classical Hardy space $H^p(\rn)$, with $p\in(0,1]$,
is a good substitute for the Lebesgue space $L^p(\rn)$, particularly, when studying
the boundedness of maximal functions and Calder\'{o}n--Zygmund operators
(see \cite{s93}).
Nowadays, the theory of $H^p(\rn)$ has been systematically studied
and proved useful in many mathematical fields such as harmonic analysis and
partial differential equations; see, for instance, \cite{fs72,m94,s93,sw60}.

Moreover, due to the notable work of Calder\'{o}n and
Torchinsky \cite{ct75} on parabolic
Hardy spaces and also in order to meet the requirements
arising in the development of harmonic analysis and partial differential
equations, there has been an increasing interest
in extending classical function spaces from Euclidean spaces to
some more general underlying spaces; see, for instance,
\cite{mb03,cgn17,hlyy,t06,yyh13}.
In particular, as a generalization of both the isotropic
Hardy space and the parabolic Hardy space of Calder\'{o}n and Torchinsky
\cite{ct75}, in 2003, Bownik \cite{mb03} first introduced
the anisotropic Hardy space $H_A^p(\rn)$ with $p\in(0,\fz)$,
where $A$ is a general expansive matrix on $\rn$
(see Definition \ref{2d1} below).
Later on, Bownik et al. \cite{blyz08} further extended the
anisotropic Hardy space to the weighted setting.
For more progresses about the real-variable theory of function
spaces in this anisotropic setting, we refer the reader to
\cite{abr16,lby14,lyy16,lyy17hl,lyy17I,t06}.
Nowadays, the anisotropic dilation on Euclidean
space $\rn$ has proved useful not only in developing the
theory of various function spaces,
but also in some other fields such as partial differential equations
(see, for instance, \cite{jm06}) and wavelet theory
(see, for instance, \cite{mb03,zl18}).

On another hand, the mixed-norm Lebesgue space $\lv$, with
the exponent vector $\vp\in (0,\fz]^n$, is a natural generalization
of the classical Lebesgue space $L^p(\rn)$, via replacing
the constant exponent $p$ by an exponent vector $\vp$.
The study of mixed-norm Lebesgue spaces originates from Benedek and Panzone
\cite{bp61} in the early 1960's, which can be traced back to H\"{o}rmander \cite{h60}.
Later on, in 1970, Lizorkin \cite{l70} further investigated both the theory of
multipliers of Fourier integrals and estimates of convolutions
in the mixed-norm Lebesgue spaces.
In particular, based on the mixed-norm Lebesgue space,
the real-variable theory of mixed-norm
function spaces, including mixed-norm Morrey spaces,
mixed-norm Hardy spaces, mixed-norm Besov
spaces and mixed-norm Triebel--Lizorkin spaces, has rapidly been developed
in recent years; see, for instance,
\cite{cs,cgn17-2,gjn17,htw17,tn} and also the survey
article \cite{hy}.

Moreover, since the mixed-norm function spaces have finer structures
than the corresponding classical function spaces, they naturally arise
in the studies on the solutions of partial differential equations
used to model physical processes involving both in space
and time variables, such as the heat or the wave equations (particularly,
the very useful Strichartz estimates); see, for instance,
\cite{kpv93,kim08,tao06}. This is based on the fact that,
while treating some linear or nonlinear equations, functions with different
orders of integrability in different variables induce a better regularity
(of traces) of solutions; see, for instance, \cite{gs91,w02}.
Another motivation to develop the real-variable theory of
mixed-norm function spaces comes from
bilinear estimates and their vector-valued extensions
which have proved useful in partial differential equations involving functions
in $n$ dimension space variable $x$ and one dimension time variable $t$;
see, for instance, \cite{bm16,bm18-2,fk00,htw17}. In particular, in order to obtain
the smoothing properties of bilinear operators and Leibniz-type rules in
mixed-norm Lebesgue spaces, Hart et al. \cite{htw17} introduced the mixed-norm Hardy
space $H^{p,q}(\mathbb{R}^{n+1})$ with $p,\,q \in(0,\fz)$
via the Littlewood--Paley g-function.
Precisely, as was mentioned in \cite[p.\,8586]{htw17}, the space
$H^{p,q}(\mathbb{R}^{n+1})$ plays an important role in overcoming the difficulty
caused by full derivatives both in the space variable $x$ and the time variable $t$
in the mixed-norm Lebesgue spaces.
For more progresses about the applications of mixed-norm function
spaces in partial differential equations, we refer the reader to
\cite{cmr13,dk19,kim08,k07}.

In addition, very recently, Cleanthous et al.
\cite{cgn17} introduced the anisotropic mixed-norm Hardy space
${H_{\va}^{\vp}(\rn)}$ associated with an anisotropic
quasi-homogeneous norm $|\cdot|_{\va}$, where
$$\va:=(a_1,\ldots,a_n)\in [1,\fz)^n\quad
{\rm and}\quad \vp:=(p_1,\ldots, p_n)\in (0,\fz)^n,$$
via the non-tangential grand maximal function and then established its radial
or its non-tangential maximal function characterizations. To complete the
real-variable theory of this Hardy space ${H_{\va}^{\vp}(\rn)}$,
Huang et al. \cite{hlyy} established several equivalent
real-variable characterizations of ${H_{\va}^{\vp}(\rn)}$, respectively,
in terms of atoms, finite atoms, Lusin area functions, Littlewood--Paley $g$-functions or
$g_{\lambda}^\ast$-functions and also obtained the boundedness of anisotropic
Calder\'{o}n--Zygmund operators from $H_{\va}^{\vp}(\rn)$ to itself [or to $\lv$].
In \cite{hlyy18}, via the atomic and the finite atomic characterizations
from \cite{hlyy}, Huang et al. further proved that the dual space of
${H_{\va}^{\vp}(\rn)}$, with $\vp\in(0,1]^n$, is the
anisotropic mixed-norm Campanato space.
For more progresses about the real-variable theory of this sort of anisotropic
function spaces, we refer the reader to
\cite{cgn17,jms15,yam86a}.

Based on the aforementioned mixed-norm Lebesgue space $\lv$ in \cite{bp61} and also
motivated by the works of Cleanthous et al. \cite{cgn17} as well as Hart et al. \cite{htw17},
in this article, we introduce a new kind of anisotropic
mixed-norm Hardy spaces $\vh$ with $\vp\in(0,\fz)^n$ and some
general expansive matrix $A$, via borrowing some ideas from the real-variable
theories of the anisotropic Hardy space $H_A^p(\rn)$ in \cite{mb03} and
the anisotropic mixed-norm Hardy space ${H_{\va}^{\vp}(\rn)}$
in \cite{cgn17,hlyy,hlyy18}.
To be precise, the anisotropic mixed-norm Hardy space $\vh$ is first introduced via the
non-tangential grand maximal function and then characterized
by means of radial or non-tangential maximal functions.
Moreover, via establishing an anisotropic Fefferman--Stein vector-valued inequality
of the Hardy--Littlewood maximal operator on the mixed-norm
Lebesgue space, various real-variable characterizations of $\vh$, respectively,
in terms of atoms, finite atoms and the square functions, including
Lusin area functions, Littlewood--Paley $g$-functions or
$g_{\lambda}^\ast$-functions are also presented. As applications
of these real-variable characterizations,
we first give the dual space of $\vh$ and then establish a criterion on the
boundedness of sublinear operators from $\vh$ into a quasi-Banach space.
Applying this criterion, we further consider the boundedness
of anisotropic convolutional $\delta$-type and non-convolutional
$\beta$-order Calder\'{o}n--Zygmund operators from $\vh$ to itself [or to $\lv$].
As a corollary, the boundedness of anisotropic convolutional
$\dz$-type Calder\'on--Zygmund operators on the mixed-norm Lebesgue space
$\lv$ for any given $\vp\in(1,\fz)^n$ is also obtained.

To the best of our knowledge, Bownik and Wang very recently revealed some connections of
anisotropic Hardy spaces $H^p_A(\rn)$ with partial
differential equations in \cite{bw19},
namely, they first obtained a differential characterization for $H^p_A(\rn)$
and, applying this characterization, further established a parabolic differential equation
characterization of $H^p_A(\rn)$. Moreover,
Johnsen et al. \cite[Theorem 6.12]{jms15} studied
necessary conditions for the existence of a solution of a heat equation in
anisotropic mixed-norm Lizorkin--Triebel spaces, where the anisotropic
setting considered in \cite{jms15} is a special case of the anisotropy
investigated in the present article. In addition, Dong and Kim \cite{dk18}
recently established several generalized versions of the Fefferman--Stein theorem
on sharp functions in spaces of homogeneous type with weights and,
using them, Dong and Kim \cite{dk18} further established a priori weighted
mixed-norm estimates for solutions of elliptic and parabolic equations
and systems with BMO coefficients in divergence and non-divergence form,
and Dong et al. \cite{dkp19} also investigated the local mixed-norm estimate
for solutions of Stokes system.
Note that the underlying space of the present article is a special case of
spaces of homogeneous type. Thus, the generalized versions of
Fefferman--Stein theorem on sharp functions obtained in \cite{dk18}
also hold true for the anisotropic setting in this article.
However, it is still a challenging problem to find some applications
of $\vh$ in partial differential equations.

The organization of this article is as follows.

In Section \ref{s2}, we first present some notation
used in this article, including mixed-norm Lebesgue spaces and
expansive matrices as well as some known facts on
homogeneous quasi-norms from \cite{mb03}. The new
anisotropic mixed-norm Hardy space $\vh$ is also
defined in this section via the non-tangential grand maximal function.

Section \ref{s3} is devoted to characterizing the space $\vh$
by means of the radial or the non-tangential
maximal functions (see Theorem \ref{3t1} below).
For this purpose, via an auxiliary inequality (see Lemma \ref{3l2} below)
from \cite{abr16}, and first establishing the boundedness of the anisotropic Hardy--Littlewood
maximal operator [see \eqref{3e1} below] on $\lv$ with $\vp\in(1,\fz)^n$
(see Lemma \ref{3l1} below), we show that the
$\lv$ quasi-norm of the tangential maximal function $T^{N(K,L)}_\phi(f)$
can be controlled by the same norm of the non-tangential maximal function $M^{(K,L)}_\phi(f)$
for any $f\in\cs'(\rn)$ (see Lemma \ref{3l4} below), where $K\in\zz$
is the truncation level, $L\in[0,\fz)$ the decay level,
$N\in\nn\cap(\frac1{p_-},\fz)$ with $p_-:=\min\{p_1,\ldots,p_n\}$
and $\cs'(\rn)$ denotes the set of all tempered distributions on
$\rn$. By this, the monotone convergence property of increasing
sequences on $\lv$ (see Lemma \ref{3l5} below) and the obtained boundedness of the anisotropic
Hardy--Littlewood maximal operator on $\lv$ again, we then prove Theorem \ref{3t1}.
Moreover, using the obtained radial maximal function characterizations of $\vh$
and Lemma \ref{3l1}, we also show that, for any given $\vp\in(1,\fz)^n$, $\vh=\lv$
with equivalent norms (see Proposition \ref{3p1} below).

In Section \ref{s4}, by borrowing some ideas from the proofs of
\cite[p.\,38, Theorem 6.4]{mb03} and \cite[Theorem 4.8]{lwyy17}
as well as \cite[Theorem 3.16]{hlyy}, we establish the atomic
characterization of $\vh$. Indeed, we first introduce the
anisotropic mixed-norm atomic Hardy space $\vah$ in Definition \ref{4d2}
below and then prove
$\vh=\vah$
with equivalent quasi-norms
(see Theorem \ref{4t1} below). To this end,
we first establish an anisotropic Fefferman--Stein vector-valued inequality of
the Hardy--Littlewood maximal operator as in \eqref{3e1}
on the mixed-norm Lebesgue space $\lv$ (see Lemma \ref{4l3} below),
where $\vp\in(1,\fz)^n$. Moreover, by borrowing
some ideas from the proof of \cite[Lemma 3.15]{hlyy}, we show
that some estimates related to $\lv$ norms for some series of functions
can be reduced into dealing with the $L^r(\rn)$ norms of the corresponding
functions (see Lemma \ref{4l4} below), which plays a key role in the proof of
Theorem \ref{4t1} and is also of independent interest.
Then, using this key lemma, the obtained vector-valued inequality and
some arguments similar to those used in the proof of \cite[Theorem 4.8]{lwyy17},
we prove that $\vah\subset\vh$ and, moreover, the inclusion is continuous.
Conversely, by \cite[p.\,32, Lemma 5.9]{mb03} and borrowing some ideas
from the proof of \cite[Lemma 3.14]{hlyy}, we find that
$\vh\cap L^{\vp/p_-}(\rn)$ is dense in $\vh$ (see Lemma \ref{4l5} below)
with $\vp\in(0,\fz)^n$ and $p_-$ as in \eqref{2e4} below.
By this density and the anisotropic Calder\'{o}n--Zygmund decomposition
associated with non-tangential grand maximal functions from
\cite[p.\,23]{mb03} as well as some arguments similar to those
used in the proof of \cite[Theorem 4.8]{lwyy17}, we then prove that $\vh$ is
continuously embedded into $\vahfz$ and hence also into $\vah$ due to the
fact that each $(\vp,\infty,s)$-atom is also a $(\vp,r,s)$-atom for
any $r\in(1,\fz)$. This proves Theorem \ref{4t1}.

The aim of Section \ref{s5} is to establish a finite atomic characterization
of $\vh$ (see Theorem \ref{5t1} below). To be exact, we first introduce the
anisotropic mixed-norm finite atomic Hardy space $\vfah$ in Definition \ref{5d1}
below and then, via borrowing some ideas from the proofs of
\cite[Theorem 5.4]{lwyy17} and \cite[Theorem 5.9]{hlyy}, we show that,
for any given finite linear combination of $(\vp,r,s)$-atoms with
$r\in(\max\{p_+,1\},\fz)$ [or continuous $(\vp,\fz,s)$-atoms],
its quasi-norm in $\vh$ can be achieved via all its finite atomic decompositions.
This actually extends \cite[Theorem 3.1 and Remark 3.3]{msv08} and
\cite[Theorem 5.9]{hlyy} to the present setting of anisotropic mixed-norm Hardy spaces.

Section \ref{s6} is devoted to establishing the square function characterizations
of $\vh$, including characterizations via the Lusin area function,
the Littlewood--Paley $g$-function or $g_{\lambda}^\ast$-function; see, respectively,
Theorems \ref{6t1} through \ref{6t3} below. To this end, via the anisotropic Calder\'{o}n
reproducing formula from \cite[Proposition 2.14]{blyz10}
(see also Lemma \ref{6l1} below), a key inequality (see Lemma \ref{6l1'} below) and
borrowing some ideas from the proof of Theorem \ref{4t1} as
well as an argument similar to that used in the proof of \cite[Theorem 6.1]{lwyy17},
we first prove Theorem 6.1.
Then, using this, an approach initiated by Ullrich \cite{u12}, which was further developed
by Liang et al. \cite{lsuyy} and Liu et al. \cite{lyy17I}, and the anisotropic Fefferman--Stein vector-valued
inequality of the Hardy--Littlewood maximal operator on $\lv$ (see Lemma \ref{4l3} below),
we obtain the Littlewood--Paley $g$-function and $g_{\lambda}^\ast$-function
characterizations of $\vh$, respectively. In addition,
applying the obtained Littlewood--Paley $g$-function
characterizations of $\vh$, we prove that the Hardy space $\vh$, introduced
in the present article, includes the Hardy space
${H_{\va}^{\vp}(\rn)}$ of Cleanthous et al. \cite{cgn17}
as a special case in the sense of equivalent quasi-norms
(see Proposition \ref{6p1} below).

In Section \ref{s7}, using the atomic and the finite atomic
characterizations of $\vh$ obtained, respectively, in Theorems \ref{4t1} and \ref{5t1},
we prove that the dual space of $\vh$ is the anisotropic mixed-norm
Campanato space (see Theorem \ref{7t1} below). For this purpose,
we first introduce a new kind of anisotropic mixed-norm
Campanato spaces $\lq$ in Definition \ref{7d1} below, which includes the
anisotropic mixed-norm Campanato space from \cite{hlyy18}, the anisotropic Campanato
space of Bownik (see \cite[p.\,50, Definition 8.1]{mb03}) and
the space $\mathop{\mathrm{BMO}}(\rn)$ of John and Nirenberg \cite{jn61}
as well as the classical Campanato space of Campanato \cite{c64} as special cases
[see (ii) and (iii) of Remark \ref{7r1} below]. Then,
by Theorems \ref{4t1} and \ref{5t1}
as well as an argument similar to that used in the proof of \cite[Theorem 3.10]{hlyy18},
we show that the space $\lr$, with $\vp\in(0,1]^n$, $r\in(1,\fz]$ and
$1/r+1/r'=1,$
is continuously embedded into $(\vh)^*$, where $(\vh)^*$ denotes the dual space of $\vh$.
Conversely, motivated by \cite[Lemma 5.9]{zsy16} and \cite[p.\,51, Lemma 8.2]{mb03},
we first establish two useful estimates (see, respectively, Lemmas \ref{7l1} and
\ref{7l2} below), which play a key role in the proof of Theorem \ref{7t1} and
are also of independent interest. Via these two lemmas, the atomic characterization
of $\vh$ again and the Hahn--Banach theorem (see, for instance, \cite[Theorem 3.6]{ru91})
as well as some arguments similar to those used in the proof of \cite[p.\,51, Theorem 8.3]{mb03}, we then show
that $(\vh)^*$ is continuously embedded into $\lr$, which then completes the proof of
Theorem \ref{7t1}. Moreover, as a direct consequence of Theorem \ref{7t1},
we obtain an equivalent characterization of the
anisotropic mixed-norm Campanato spaces $\lq$ (see Corollary \ref{7c1} below).

Section \ref{s8} is aimed to give further applications
of the real-variable characterizations obtained above. Via the finite atomic
characterization of $\vh$ obtained in Section \ref{s5}, we first establish a
criterion on the boundedness of sublinear operators from $\vh$ into a
quasi-Banach space (see Theorem \ref{8t1} below), which further implies
the boundedness of anisotropic Calder\'{o}n--Zygmund operators from $\vh$
to itself [or to $\lv$] (see Theorems \ref{8t2} and \ref{8t3} below).
To be precise, using Theorem \ref{8t1},
we first show that, if $T$ is a sublinear operator and maps all $(\vp,r,s)$-atoms
with $r\in(1,\fz)$ [or all continuous $(\vp,\fz,s)$-atoms] into uniformly bounded
elements of some $\gamma$-quasi-Banach space $\mathcal{B}_{\gamma}$ with
$\gamma\in (0,1]$, then $T$ has a unique bounded $\mathcal{B}_{\gamma}$-sublinear
extension from $\vh$ into $\mathcal{B}_{\gamma}$ (see Corollary \ref{8c1} below),
which extends the corresponding results of Meda et al. \cite[Corollary 3.4]{msv08} and
Grafakos et al. \cite[Theorem 5.9]{gly08} as well as Ky \cite[Theorem 3.5]{ky14}
(see also \cite[Theorem 1.6.9]{ylk17}) and Huang et al.
\cite[Corollary 6.3]{hlyy} to the present setting. Then, by Theorem \ref{8t1} again
and borrowing some ideas from the proofs of \cite[Theorems 6.4, 6.5, 6.8 and 6.9]{hlyy},
we obtain the boundedness of anisotropic convolutional $\delta$-type and
non-convolutional $\bz$-order Calder\'{o}n--Zygmund operators
from $\vh$ to itself [or to $\lv$] (see Theorems \ref{8t2} and \ref{8t3} below).
In addition, as a corollary of Theorem \ref{8t2},
we obtain the boundedness of anisotropic convolutional $\dz$-type
Calder\'on--Zygmund operators on $\lv$ for any given $\vp\in(1,\fz)^n$
(see Corollary \ref{8c2} below).

Recall that, in \cite{lwyy17}, Liu et al. established various real-variable
characterizations of variable anisotropic Hardy spaces $H^{p(\cdot)}_A(\rn)$.
We point out that the integrable exponent of
the Hardy space $H^{p(\cdot)}_A(\rn)$ from \cite{lwyy17} is
a variable exponent function,
$p(\cdot):\ \rn\to(0,\fz],$
satisfying the so-called globally
log-H\"{o}lder continuous condition (see \cite[(2.5) and (2.6)]{lwyy17}), whose
associated basic function space is the variable Lebesgue space $L^{p(\cdot)}(\rn)$;
however, as was mentioned above,
the integrable exponent of the anisotropic mixed-norm Hardy space $\vh$,
investigated in the present article,
is a vector $\vp\in(0,\fz)^n$,
whose associated basic function space is the mixed-norm Lebesgue space $\lv$
which has different orders of integrability in different variables.
Obviously, $L^{p(\cdot)}(\rn)$ and $\lv$ cannot cover each other,
so do the variable anisotropic Hardy space
$H^{p(\cdot)}_A(\rn)$ of \cite{lwyy17} and the Hardy space $\vh$ of
the present article. In addition, the real-variable theory of
mixed-norm Hardy space ${H_{\va}^{\vp}(\rn)}$
associated with a vector $\va\in[1,\fz)^n$ was established
in \cite{cgn17,hlyy,hlyy18}. Observe that
the space ${H_{\va}^{\vp}(\rn)}$, introduced by
Cleanthous et al. \cite{cgn17}, is included in the Hardy space $\vh$
of the present article as a special case in the sense of
equivalent quasi-norms (see Proposition \ref{6p1} below).
Thus, it is indeed a meaningful subject to develop a real-variable theory of
anisotropic mixed-norm Hardy spaces $\vh$.
Moreover, to do so, comparing with \cite{cgn17,hlyy,hlyy18,lwyy17},
we also need to overcome some differently essential difficulties.
For instance, compared with Liu et al. \cite[Theorem 3.10]{lwyy17} on
the various maximal function characterizations of $H^{p(\cdot)}_A(\rn)$, to establish
the corresponding maximal function characterizations of $\vh$ (see Theorem \ref{3t1} below),
the main difficulty exists in the lack of the boundedness of anisotropic Hardy--Littlewood maximal operators
$\HL$ on mixed-norm Lebesgue spaces $\lv$, which is a necessary and
key tool in the proof of Theorem \ref{3t1}.
Thus, we first obtain this necessary boundedness
via the fact that the operator $\HL$ can be controlled
by the iterated maximal operator pointwisely (see Remark \ref{3r1} below)
together with a key inequality on mixed-norms
from Bagby \cite{b75} [see also \eqref{3eq4} below].
On another hand, in Section \ref{s4}, to establish the atomic characterizations of $\vh$,
we have to first establish an anisotropic Fefferman--Stein vector-valued inequality
on mixed-norm Lebesgue spaces $\lv$, which is known to be fundamental tool in developing
a real-variable theory of Hardy spaces; however, in \cite{lwyy17},
the corresponding inequality is known and can be used directly to prove the desired atomic characterization.
We obtain this desired anisotropic Fefferman--Stein vector-valued inequality
on $\lv$, via the obtained boundedness of $\HL$
on $\lv$, a duality argument of $\lv$ (see \cite[p.\,304, Theorem 2]{bp61}),
a useful result from Sawano \cite[Theorem 1.3]{s05} and a key observation [see \eqref{4e2x} below]
as well as borrowing some ideas from the proof of \cite[Theorem 2.7]{zsy16}.
In addition, note that, in the anisotropic setting of \cite{cgn17,hlyy,hlyy18},
the quasi-norm $\|\cdot\|_{\lv}$ of the characteristic function of anisotropic cubes
can be precisely calculated out (see \cite[Lemma 4.7]{hlyy}), which plays a crucial role
in establishing the real-variable characterizations of ${H_{\va}^{\vp}(\rn)}$,
while, in the present article, the corresponding quasi-norm of the characteristic
function of anisotropic cubes can not be precisely calculated out due to its more
general anisotropic structure. To overcome this difficulty, we fully use the relation between
homogeneous quasi-norms and Euclidean norms as well as the homogeneity of $\|\cdot\|_{\lv}$
to obtain some subtle estimates on the quasi-norm in $\lv$ of the characteristic
function of anisotropic cubes as substitutes.

Finally, we make some conventions on notation.
We always let
$\mathbb{N}:=\{1,2,\ldots\}$,
$\mathbb{Z}_+:=\{0\}\cup\mathbb{N}$
and $\vec0_n$ be the \emph{origin} of $\rn$.
For any multi-index
$\az:=(\az_1,\ldots,\az_n)\in(\mathbb{Z}_+)^n=:\mathbb{Z}_+^n,$
let
$|\az|:=\az_1+\cdots+\az_n$ and
$\partial^{\az}:=(\frac{\partial}{\partial x_1})^{\az_1} \cdots
(\frac{\partial}{\partial x_n})^{\az_n}.$
We denote by $C$ a \emph{positive constant}
which is independent of the main parameters,
but may vary from line to line. We also use $C_{(\az,\bz,\ldots)}$
to denote a positive constant
depending on the indicated parameters $\az,\,\bz,\ldots$.
The notation $f\ls g$ means $f\le Cg$ and, if $f\ls g\ls f$,
then we write $f\sim g$. We also use the following
convention: If $f\le Cg$ and $g=h$ or $g\le h$, we then write $f\ls g\sim h$
or $f\ls g\ls h$, \emph{rather than} $f\ls g=h$
or $f\ls g\le h$. For any $p\in[1,\fz]$, we denote by $p'$
its \emph{conjugate index}, namely, $1/p+1/p'=1$.
Moreover, if $\vec{p}:=(p_1,\ldots,p_n)\in[1,\fz]^n$, we denote by
$\vec{p}':=(p_1',\ldots,p_n')$ its \emph{conjugate index}, namely,
for any $i\in\{1,\ldots,n\}$, $1/p_i+1/p_i'=1$. In addition,
for any set $E\subset\rn$, we denote by $E^\complement$ the
set $\rn\setminus E$, by ${\mathbf 1}_E$ its \emph{characteristic function},
by $|E|$ its \emph{n-dimensional Lebesgue measure}
and by $\sharp E$ its \emph{cardinality}.
For any $\ell\in\mathbb{R}$, we denote by $\lfloor \ell\rfloor$ (resp., $\lceil\ell\rceil$)
the \emph{largest} (resp., \emph{least}) \emph{integer not} greater (resp., \emph{less})
\emph{than}  $\ell$. Throughout
this article, the \emph{symbol} $C^{\fz}(\rn)$
denotes the set of all \emph{infinitely differentiable functions} on $\rn$.

\section{Preliminaries \label{s2}}

In this section, we first recall some notions on dilations
and mixed-norm Lebesgue spaces (see, for instance, \cite{bp61,mb03}).
Then we introduce the anisotropic mixed-norm
Hardy space via the non-tangential grand maximal function.

We begin with recalling the notion of dilations
from \cite[p.\,5, Definition 2.1]{mb03}.

\begin{definition}\label{2d1}
A real $n\times n$ matrix $A$ is called an \emph{expansive matrix},
shortly, a \emph{dilation} if
$$\min_{\lz\in\sigma(A)}|\lz|>1,$$
here and thereafter, $\sigma(A)$ denotes the \emph{set of
all eigenvalues of $A$}.
\end{definition}

Let $b:=|\det A|$. Then, from \cite[p.\,6, (2.7)]{mb03}, it follows
that $b\in(1,\fz)$. By the fact that there exist an open
ellipsoid $\Delta$, with $|\Delta|=1$, and $r\in(1,\infty)$ such that
$\Delta\subset r\Delta\subset A\Delta$ (see \cite[p.\,5, Lemma 2.2]{mb03}),
we find that, for any $k\in\zz$, $B_k:=A^k\Delta$ is open,
$B_k\subset rB_k\subset B_{k+1}$ and $|B_k|=b^k$.
For any $x\in\rn$ and $k\in\mathbb{Z}$, an ellipsoid $x+B_k$
is called a \emph{dilated ball}. In what follows,
we always let $\mathfrak{B}$ be the set of all such
dilated balls, namely,
\begin{align}\label{2e1}
\mathfrak{B}:=\lf\{x+B_k:\ x\in\rn\ {\rm and}\ k\in\mathbb{Z}\r\}
\end{align}
and
\begin{align}\label{2e2}
\omega:=\inf\lf\{\ell\in\zz:\ r^\ell\ge2\r\}.
\end{align}

The following notion of homogeneous quasi-norms
is just \cite[p.\,6, Definition 2.3]{mb03}.

\begin{definition}\label{2d2}
A \emph{homogeneous quasi-norm},
associated with a dilation $A$, is a measurable mapping
$\rho:\ \rn \to [0,\infty)$ satisfying
\begin{enumerate}
\item[\rm{(i)}] if $x\neq\vec0_n$, then $\rho(x)\in(0,\fz)$;

\item[\rm{(ii)}] for any $x\in\rn$, $\rho(Ax)=b\rho(x)$;

\item[\rm{(iii)}] there exists an $H\in[1,\fz)$
such that, for any $x$, $y\in\rn$, $\rho(x+y)\le H[\rho(x)+\rho(y)]$.
\end{enumerate}
\end{definition}

In what follows, for a given dilation $A$, by \cite[p.\,6, Lemma 2.4]{mb03},
we may use, for both simplicity and convenience, the \emph{step homogeneous quasi-norm} $\rho$
defined by setting, for any $x\in\rn$,
\begin{equation}\label{2e3}
\rho(x):=\sum_{k\in\mathbb{Z}}
b^k{\mathbf 1}_{B_{k+1}\setminus B_k}(x)\hspace{0.25cm}
{\rm when}\ x\neq\vec0_n,\hspace{0.35cm} {\rm or\ else}
\hspace{0.25cm}\rho(\vec0_n):=0.
\end{equation}

Now we present the definition of mixed-norm Lebesgue spaces from \cite{bp61}.
\begin{definition}\label{2d3}
Let $\vp:=(p_1,\ldots,p_n)\in (0,\fz]^n$. The \emph{mixed-norm Lebesgue space} $\lv$ is
defined to be the set of all measurable functions $f$ such that
$$\|f\|_{\lv}:=\lf\{\int_{\mathbb R}\cdots\lf[\int_{\mathbb R}|f(x_1,\ldots,x_n)|^{p_1}\,dx_1\r]
^{\frac{p_2}{p_1}}\cdots\, dx_n\r\}^{\frac{1}{p_n}}<\fz$$
with the usual modifications made when $p_i=\fz$
for some $i\in \{1,\ldots,n\}$.
\end{definition}

\begin{remark}\label{2r1}
For any $\vp\in(0,\fz]^n$, $(\lv,\|\cdot\|_{\lv})$ is a quasi-Banach
space and, for any $\vp \in [1,\fz]^n$, $(\lv,\|\cdot\|_{\lv})$
becomes a Banach space; see \cite[p.\,304, Theorem 1.b)]{bp61}.
Obviously, when $\vp:=(\overbrace{p,\ldots,p}^{n\ \rm times})$
with some $p\in(0,\fz]$, $\lv$ coincides
with the classical Lebesgue space $L^p(\rn)$ and,
in this case, they have the same norms.
\end{remark}

For any $\vp:=(p_1,\ldots,p_n)\in (0,\fz)^n$, we always let
\begin{align}\label{2e4}
p_-:=\min\{p_1,\ldots,p_n\},\hspace{0.35cm}
p_+:=\max\{p_1,\ldots,p_n\}
\hspace{0.35cm}{\rm and}\hspace{0.35cm}
\underline{p}\in(0,\min\{p_-,1\}).
\end{align}

Recall that a $C^\infty(\rn)$
function $\varphi$ is called a \emph{Schwartz function} if,
for any $m\in\zz_+$ and multi-index $\az\in\zz_+^n$,
$$\|\varphi\|_{\alpha,m}:=
\sup_{x\in\rn}[\rho(x)]^m
|\partial^\alpha\varphi(x)|<\infty.$$
Denote by
$\cs(\rn)$ the set of all Schwartz functions, equipped
with the topology determined by
$\{\|\cdot\|_{\alpha,m}\}_{\az\in\zz_+^n,m\in\zz_+}$,
and $\cs'(\rn)$ its \emph{dual space}, equipped
with the weak-$\ast$ topology.
For any $N\in\mathbb{Z}_+$, let
$$\cs_N(\rn):=\{\varphi\in\cs(\rn):\
\|\varphi\|_{\alpha,m}\le1,\
|\alpha|\le N,\ m\le N\};$$
equivalently,
\begin{align*}
\varphi\in\cs_N(\rn)\hs
{\rm if}\ {\rm and}\ {\rm only}\ {\rm if}\hs
\|\varphi\|_{\cs_N(\rn)}:=\sup_{|\alpha|\le N}
\sup_{x\in\rn}\lf[\lf|\partial^\alpha
\varphi(x)\r|\max\lf\{1,\lf[
\rho(x)\r]^N\r\}\r]\le1.
\end{align*}
In what follows, for any
$\varphi\in\cs(\rn)$ and $k\in\mathbb{Z}$, let
$\varphi_k(\cdot):=b^{k}\varphi(A^{k}\cdot)$.

Let $\lambda_-$, $\lambda_+\in(1,\fz)$
be two \emph{numbers} such that
$$\lambda_-<\min\lf\{|\lambda|:\
\lambda\in\sigma(A)\r\}
\le\max\{|\lambda|:\
\lambda\in\sigma(A)\}<\lambda_+.$$
In addition, we should point out that, if $A$ is diagonalizable over
$\mathbb{C}$, then we may let
$\lambda_-:=\min\{|\lambda|:\
\lambda\in\sigma(A)\}$
and
$\lambda_+:=\max\{|\lambda|:\
\lambda\in\sigma(A)\}$.
Otherwise, we may choose them sufficiently close to these equalities
in accordance with what we need in our arguments.

\begin{definition}\label{2d4}
Let $\varphi\in\cs(\rn)$ and $f\in\cs'(\rn)$. The
\emph{non-tangential maximal function} $M_\varphi(f)$
with respect to $\varphi$ is defined by setting, for any $x\in\rn$,
\begin{align*}
M_\varphi(f)(x):= \sup_{y\in x+B_k,
k\in\mathbb{Z}}|f\ast\varphi_k(y)|.
\end{align*}
Moreover, for any given $N\in\mathbb{N}$, the
\emph{non-tangential grand maximal function} $M_N(f)$ of
$f\in\cs'(\rn)$ is defined by setting, for any $x\in\rn$,
\begin{equation*}
M_N(f)(x):=\sup_{\varphi\in\cs_N(\rn)}
M_\varphi(f)(x).
\end{equation*}
\end{definition}

We now introduce the anisotropic mixed-norm Hardy spaces as follows.

\begin{definition}\label{2d5}
Let $\vp\in (0,\fz)^n$ and
$N\in\mathbb{N}\cap[\lfloor(\frac1{\min\{1,p_-\}}-1)\frac{\ln b}{\ln
\lambda_-}\rfloor+2,\fz)$,
where $p_-$ is as in \eqref{2e4}.
The \emph{anisotropic mixed-norm Hardy space} $\vh$ is defined
by setting
\begin{equation*}
\vh:=\lf\{f\in\cs'(\rn):\ M_N(f)\in\lv\r\}
\end{equation*}
and, for any $f\in\vh$, let
$\|f\|_{\vh}:=\| M_N(f)\|_{\lv}$.
\end{definition}

\begin{remark}\label{2r2}
\begin{enumerate}
\item[{\rm(i)}]
The quasi-norm of $\vh$ in Definition \ref{2d5} depends on $N$, however,
by Theorem \ref{4t1} below, we know that the space $\vh$ is independent
of the choice of $N$ as long as $N$ same as in Definition \ref{2d5}.
\item[{\rm (ii)}]
To study the smoothing properties of bilinear operators and Leibniz-type
rules in mixed-norm Lebesgue spaces, Hart et al. \cite{htw17} introduced
the mixed-norm Hardy space $H^{p,q}(\mathbb{R}^{n+1})$ with $p,\,q\in(0,\fz)$
via the Littlewood--Paley $g$-function. Here we should point out that,
if $$\vp:=(\overbrace{p,\ldots,p}^{n\ \mathrm{times}},q)\quad
{\rm with}\ p,\,q\in(0,\fz),$$
and $A:=d\,{\rm I}_{(n+1)\times (n+1)}$
for some $d\in\mathbb R$ with $|d|\in(1,\fz)$,
here and thereafter, ${\rm I}_{n\times n}$ denotes
the $n\times n$ \emph{unit matrix},
then, from Theorem \ref{6t2} below, it follows that
$H_A^{\vp}(\mathbb{R}^{n+1})$, studied in this article,
and $H^{p,q}(\mathbb{R}^{n+1})$ from \cite{htw17} coincide
with equivalent quasi-norms.
\item[{\rm(iii)}]
Recall that Bownik \cite{mb03} introduced the anisotropic Hardy
space $H^p_A(\rn)$, with $p\in(0,\fz)$, via the non-tangential grand maximal
function (see \cite[p.\,17, Definition 3.11]{mb03}) and investigated its
real-variable theory.
It is easy to see that, when $\vp:=(\overbrace{p,\ldots,p}^{n\ \rm times})$
with some $p\in(0,\fz)$, the space $\vh$ just becomes the anisotropic Hardy
space $H^p_A(\rn)$ from \cite{mb03}.
\item[{\rm(iv)}]
Very recently, Cleanthous et al. \cite{cgn17} introduced the
anisotropic mixed-norm Hardy space ${H_{\va}^{\vp}(\rn)}$, with $\va\in [1,\fz)^n$
and $\vp\in (0,\fz)^n$, via the non-tangential grand maximal
function (see \cite[Definition 3.3]{cgn17}) and established its radial
or non-tangential maximal function characterizations. We should point out that,
by Proposition \ref{6p1} below, we know that, when
\begin{align}\label{2e5}
A:=\lf(\begin{array}{cccc}
2^{a_1} & 0 & \cdots & 0\\
0 & 2^{a_2} & \cdots & 0\\
\vdots & \vdots& &\vdots\\
0 & 0 & \cdots & 2^{a_n}\\
\end{array}\r)
\end{align}
with $\va:=(a_1,\ldots, a_n)\in [1,\fz)^n$ and
$\vp\in(0,\fz)^n$, the Hardy space $\vh$, introduced in this article,
and the anisotropic mixed-norm Hardy space
${H_{\va}^{\vp}(\rn)}$ from \cite{cgn17} coincide with
equivalent quasi-norms.
\end{enumerate}
\end{remark}

\section{Maximal function characterizations of $\vh$\label{s3}}

In this section, we characterize $\vh$
via the radial maximal function $M_\varphi^0$ (see Definition \ref{3d2} below)
or the non-tangential maximal function $M_\varphi$ (see Definition \ref{2d4}).
To this end, we first recall the following notions of some auxiliary maximal functions
from \cite[p.\,44]{mb03}.

\begin{definition}\label{3d1}
Let $K\in\zz$, $L\in[0,\fz)$ and $N\in\nn$. For any $\varphi\in\cs(\rn)$,
the \emph{maximal functions} $M_\varphi^{0(K,L)}(f)$,
$M_\varphi^{(K,L)}(f)$ and $T_\varphi^{N(K,L)}(f)$ of
$f\in\cs'(\rn)$ are defined, respectively, by
setting, for any $x\in\rn$,
\begin{align*}
M_\varphi^{0(K,L)}(f)(x):=\sup_{k\in\mathbb{Z},\,k\le K}
|(f\ast\varphi_k)(x)|\lf[\max\lf\{1,\rho\lf(A^{-K}x\r)\r\}\r]^
{-L}\lf(1+b^{-k-K}\r)^{-L},
\end{align*}
\begin{align*}
M_\varphi^{(K,L)}(f)(x):=\sup_{k\in\mathbb{Z},\,k\le K}
\sup_{y\in x+B_k}|(f\ast\varphi_k)(y)|\lf[\max\lf\{1,\rho
\lf(A^{-K}y\r)\r\}\r]^{-L}\lf(1+b^{-k-K}\r)^{-L}
\end{align*}
and
\begin{align*}
T_\varphi^{N(K,L)}(f)(x):=\sup_{k\in\mathbb{Z},\,k\le K}
\sup_{y\in\rn}\frac{|(f\ast\varphi_k)(y)|}{[\max\{1,\rho
(A^{-k}(x-y))\}]^{N}}\frac{(1+b^{-k-K})^{-L}}
{[\max\{1,\rho(A^{-K}y)\}]^{L}}.
\end{align*}
Moreover, the \emph{maximal functions} $M_N^{0(K,L)}(f)$
and $M_N^{(K,L)}(f)$ of $f\in\cs'(\rn)$ are defined, respectively, by
setting, for any $x\in\rn$,
\begin{align*}
M_N^{0(K,L)}(f)(x):=
\sup_{\varphi\in\cs_N(\rn)}M_\varphi^{0(K,L)}(f)(x)
\end{align*} and
\begin{align*}
M_N^{(K,L)}(f)(x):=
\sup_{\varphi\in\cs_N(\rn)}M_\varphi^{(K,L)}(f)(x).
\end{align*}
\end{definition}
The \emph{symbol} $L_{\rm loc}^1(\rn)$ denotes the set of all locally integrable functions
on $\rn$. Recall that the \emph{anisotropic Hardy--Littlewood maximal operator}
$M_{{\rm HL}}(f)$ of $f\in L^1_{{\rm loc}}(\rn)$ is defined by setting,
for any $x\in\rn$,
\begin{align}\label{3e1}
M_{{\rm HL}}(f)(x):=\sup_{k\in\mathbb{Z}}
\sup_{y\in x+B_k}\frac1{|B_k|}
\int_{y+B_k}|f(z)|\,dz=\sup_{x\in B\in\mathfrak{B}}
\frac1{|B|}\int_B|f(z)|\,dz,
\end{align}
where $\mathfrak{B}$ is as in \eqref{2e1}.
\begin{remark}\label{3r1}
For any $f\in L_{\rm loc}^1(\rn)$ and $x\in \rn$, let
\begin{align*}
M(f)(x):=\sup_{I_n\in \mathbb{I}_{x_n}}\lf\{\frac{1}{|I_n|}\int_{I_n}
\cdots\sup_{I_1\in \mathbb{I}_{x_1}}\lf[\frac{1}{|I_1|}\int_{I_1}
|f(y_1,\ldots,y_n)|\,dy_1\r]\cdots\,dy_n\r\},
\end{align*}
where, for any $k\in \{1,\ldots,n\}$, ${\mathbb I}_{x_k}$
denotes the collection of all intervals in $\mathbb R_{x_k}$ containing
$x_k$. Then there exists a positive constant $C$ such that,
for any $f\in L^1_{{\rm loc}}(\rn)$ and $x\in \rn$,
\begin{align*}
M_{{\rm HL}}(f)(x)\le CM(f)(x).
\end{align*}
\end{remark}

We first establish the following boundedness of
the anisotropic Hardy--Littlewood maximal operator $M_{\rm HL}$ on $\lv$
with any given $\vp\in(1,\fz)^n$.

\begin{lemma}\label{3l1}
Let $\vp\in (1,\fz)^n$. Then there exists a positive
constant C such that, for any $f\in L_{\rm loc}^1(\rn)$,
$$\|M_{{\rm HL}}(f)\|_{\lv}\le C\|f\|_{\lv},$$
where $M_{{\rm HL}}$ is as in \eqref{3e1}.
\end{lemma}

\begin{proof}
Let $\vp\in(1,\fz)^n$, $d_1\in\nn$ and $d_2\in \zz_+$ satisfy $d_1+d_2=n$.
For any $f\in L_{\rm loc}^1(\rn)$,
$s\in\rr^{d_1}$ and $t\in\rr^{d_2}$,
let $$f^*(s,t):=\sup_{r\in (0,\fz)}\frac{1}{|B(s,r)|}
\int_{B(s,r)}|f(y,t)|\,dy,$$
where, for any $s\in \rr^{d_1}$ and $r\in(0,\fz)$,
$B(s,r):=\{z\in\rr^{d_1}:\ |z-s|<r\}$.
If $d_2\in\nn$, then, for any given $\vec{p}_{d_2}
:=(p_1,\ldots,p_{d_2})\in(1,\fz)^{d_2}$, $f\in L_{\rm loc}^1(\rn)$
and $s\in\rr^{d_1}$, define
\begin{align*}
T_{L^{\vec{p}_{d_2}}(\rr^{d_2})}(f)(s)
:=\left\{\int_{\rr}\cdots\left[\int_{\rr}
|f(s,t_1,\ldots,t_{d_2})|^{p_1}\,dt_1\right]^{\frac{p_2}{p_1}}
\cdots \,dt_{d_2}\right\}^{\frac{1}{p_{d_2}}},
\end{align*}
where, if $d_2\equiv 0$, then let $\rr^0:=\emptyset$ and
$T_{L^{\vec{p}_{d_2}}(\rr^{d_2})}(f)(s):=|f(s)|$
for any $s\in\rr^{d_1}$. Moreover, for any given $q\in (1,\fz)$,
we have
\begin{align}\label{3eq4}
\int_{\rr^{d_1}}\lf[T_{L^{\vec{p}_{d_2}}(\rr^{d_2})}(f^*)(s)\r]^q \,ds
\ls \int_{\rr^{d_1}}\lf[T_{L^{\vec{p}_{d_2}}(\rr^{d_2})}(f)(s)\r]^q \,ds
\end{align}
(see \cite[p.\,421]{b75} or \cite[Theorem 2.12]{hy} for a detailed proof).
In addition, for any $k\in\{1,\ldots,n\}$ and $x:=(x_1,\ldots,x_n)\in \rn$, let
$$M_k(f)(x):=\sup_{I\in \mathbb{I}_{x_k}}\frac{1}{|I|}\int_I
|f(x_1,\ldots,y_k,\ldots,x_n)|\,dy_k,$$
where $\mathbb{I}_{x_k}$ is as in Remark \ref{3r1}.
Therefore, for any $x\in\rn$, we have
\begin{align*}
M(f)(x)=M_n\lf(\cdots\lf(M_1(f)\r)\cdots\r)(x).
\end{align*}
From this, Remark \ref{3r1} and \eqref{3eq4} with
$d_1=1,\,d_2=n-1$, $s=x_n$, $t=(x_1,\ldots,x_{n-1})$
and $q=p_n$, it follows that
\begin{align*}
\|M_{{\rm HL}}(f)\|_{\lv}&\ls\|M_n\lf(\cdots\lf(M_1(f)\r)\cdots\r)\|_{\lv}\\
&\ls\lf\{\int_{\rr}\lf[T_{L^{\vec{p}_{n-1}}(\rr^{n-1})}
\lf(\lf(M_{n-1}\lf(\cdots\lf(M_1(f)\r)\cdots\r)\r)^*\r)(x_n)\r]^{p_n}\, dx_n\r\}^{\frac{1}{p_n}}\noz\\
&\ls \lf\{\int_{\rr}\lf[T_{L^{\vec{p}_{n-1}}(\rr^{n-1})}
(M_{n-1}\lf(\cdots\lf(M_1(f)\r)\cdots\r))(x_n)\r]^{p_n} \,dx_n\r\}^{\frac{1}{p_n}}\noz\\
&\sim \|M_{n-1}\lf(\cdots\lf(M_1(f)\r)\cdots\r)\|_{\lv}.\noz
\end{align*}
Repeating this estimate $n-1$ times, we then complete the proof of Lemma \ref{3l1}.
\end{proof}

\begin{remark}
We point out that, when $n:=2$ and $\vp:=(p_1,\infty)$ with $p_1\in(1,\fz)$,
Lemma \ref{3l1} is not true; see \cite[Remark 4.4]{hy}.
\end{remark}

The following Lemmas \ref{3l2} and \ref{3l3} are just,
respectively, \cite[Lemma 2.3]{abr16}
and \cite[Remark 2.8(iii)]{hlyy}.

\begin{lemma}\label{3l2}
There exists a positive constant $C$ such that, for any
$K\in\zz$, $L\in[0,\fz)$, $\lz\in(0,\fz)$,
$N\in\nn\cap[\frac1{\lz},\fz)$, $\varphi\in\cs(\rn)$, $f\in\cs'(\rn)$ and $x\in\rn$,
$$\lf[T_\varphi^{N(K,L)}(f)(x)\r]^{\lz}\le CM_{{\rm HL}}
\lf(\lf[M_\varphi^{(K,L)}(f)\r]^{\lz}\r)(x),$$
where $T_\varphi^{N(K,L)}$ and $M_\varphi^{(K,L)}$ are as in Definition \ref{3d1}
and $M_{{\rm HL}}$ is as in \eqref{3e1}.
\end{lemma}

\begin{lemma}\label{3l3}
Let $\vp\in (0,\fz]^n$. Then, for any $r\in (0,\fz)$ and $f\in\lv$,
$$\lf\||f|^r\r\|_{\lv}=\|f\|_{L^{r \vp}(\rn)}^r,$$
here and thereafter, for any $\az\in\rr$, $\az\vp:=(\az p_1,\ldots,\az p_n)$.
In addition, for any $\mu\in{\mathbb C}$, $\theta\in [0,\min\{1,p_-\}]$
with $p_-$ as in \eqref{2e4} and $f,\ g\in\lv$,
$\|\mu f\|_{\lv}=|\mu|\|f\|_{\lv}$ and
\begin{align*}
\|f+g\|_{\lv}^{\theta}\le \|f\|_{\lv}^{\theta}+\|g\|_{\lv}^{\theta}.
\end{align*}
\end{lemma}

The following two results are basic facts of $\vh$, whose proof is similar to that
of \cite[p.\,21, Theorem 4.5 and p.\,18, Proposition 3.12]{mb03}; the details are omitted.

\begin{lemma}\label{4l2}
Let $\vp$ and $N$ be as in Definition \ref{2d5}.
Then $\vh\subset\cs'(\rn)$ and the inclusion is continuous.
\end{lemma}

\begin{proposition}\label{2p1}
Let $\vp$ and $N$ be as in Definition \ref{2d5}.
Then $\vh$ is complete.
\end{proposition}

By Lemmas \ref{3l1} through \ref{3l3}, we easily
obtain the following conclusion; the details are omitted.

\begin{lemma}\label{3l4}
Let $\vp\in(0,\fz)^n$. Then there exists a positive constant $C$ such that,
for any $K\in\zz$, $L\in[0,\fz)$, $N\in\nn\cap(\frac1{p_-},\fz)$, $\varphi\in\cs(\rn)$ and $f\in\cs'(\rn)$,
\begin{align*}
\lf\|T_\varphi^{N(K,L)}(f)\r\|_{\lv}
\le C\lf\|M_\varphi^{(K,L)}(f)\r\|_{\lv},
\end{align*}
where $T_\varphi^{N(K,L)}$ and $M_\varphi^{(K,L)}$ are as in Definition \ref{3d1}.
\end{lemma}

Applying the monotone convergence theorem (see \cite[p.\,62, Corollary 1.9]{ss05}) and
\cite[p.\,304, Theorem 2]{bp61}, we have the following monotone
convergence property of $\lv$ with the details omitted.

\begin{lemma}\label{3l5}
Let $\vp\in[1,\fz)^n$ and $\{g_i\}_{i\in\nn}\subset\lv$
be any sequence of non-negative functions satisfying that
$g_i$, as $i\to\fz$, increases pointwisely almost everywhere
to some $g\in\lv$. Then
$$\|g-g_i\|_{\lv}\to0\hspace{0.5cm} as\ \ i\to\fz.$$
\end{lemma}

To establish the maximal function characterizations of $\vh$, we also need
the following three technical lemmas, which are just, respectively,
\cite[p.\,45, Lemma 7.5, p.\,46, Lemma 7.6 and p.\,11, Lemma 3.2]{mb03}.

\begin{lemma}\label{3l6}
Let $\varphi\in\cs(\rn)$ and
$\int_{\rn}\varphi(x)\,dx\neq0$. Then,
for any given $N\in\mathbb{N}$ and $L\in[0,\fz)$, there
exist an $I:=N+2(n+1)+L$
and a positive constant $C_{(N,L)}$, depending on $N$ and $L$,
such that, for any $K\in\zz_+$, $f\in\cs'(\rn)$ and $x\in\rn$,
$$M_I^{0(K,L)}(f)(x)\le C_{(N,L)}T_\varphi^{N(K,L)}(f)(x),$$
where $M_I^{0(K,L)}$ and $T_\varphi^{N(K,L)}$ are as in Definition \ref{3d1}.
\end{lemma}

\begin{lemma}\label{3l7}
Let $\varphi$ be as in Lemma \ref{3l6}.
Then, for any given $\lz\in(0,\fz)$ and $K\in\zz_+$,
there exist $L\in(0,\fz)$ and a positive constant $C_{(K,\,\lz)}$,
depending on $K$ and $\lz$, such that,
for any $f\in\cs'(\rn)$ and $x\in\rn$,
\begin{align}\label{3e3}
M_\varphi^{(K,L)}(f)(x)
\le C_{(K,\,\lz)}\lf[\max\lf\{1,\rho(x)\r\}\r]^{-\lz},
\end{align}
where $M_\varphi^{(K,L)}$ is as in Definition \ref{3d1}.
\end{lemma}

\begin{lemma}\label{3l8}
There exists a positive constant $C$ such that, for any $x\in \rn$,
\begin{equation}\label{2e1'}
C^{-1}[\rho(x)]^{\ln \lambda_-/\ln b} \le \lf|x\r|
\le  C [\rho(x)]^{\ln \lambda_+/\ln b}\quad  when\quad \rho(x)\in[1,\fz),
\end{equation}
and
\begin{equation}\label{2e2'}
C^{-1} [\rho(x)]^{\ln \lambda_+/\ln b} \le \lf|x\r|
\le  C [\rho(x)]^{\ln \lambda_-/\ln b}\quad when\quad \rho(x)\in[0,1).
\end{equation}
\end{lemma}

The following notion of anisotropic radial and radial grand maximal
functions is from \cite[p.\,12, Definition 3.4]{mb03}.

\begin{definition}\label{3d2}
Let $\varphi\in\cs(\rn)$ and $f\in\cs'(\rn)$. The
\emph{radial maximal function} $M_\varphi^0(f)$
of $f$ with respect to $\varphi$ is defined by setting, for any $x\in\rn$,
\begin{equation*}
M_\varphi^0(f)(x):= \sup_{k\in\mathbb{Z}}
|f\ast\varphi_k(x)|.
\end{equation*}
Moreover, for any given $N\in\mathbb{N}$, the
\emph{radial grand maximal function} $M_N^0(f)$
of $f\in\cs'(\rn)$ is defined by setting,
for any $x\in\rn$,
\begin{equation*}
M_N^0(f)(x):=\sup_{\varphi\in\cs_N(\rn)}
M_\varphi^0(f)(x).
\end{equation*}
\end{definition}

We now state the main result of this section as follows.

\begin{theorem}\label{3t1}
Let $\vp\in(0,\fz)^n$, $N\in\nn\cap[\lfloor\frac1{p_-}\rfloor+2n+3,\fz)$
and $\varphi$ be as in Lemma \ref{3l6}. Then
the following statements are mutually equivalent:
\begin{enumerate}
\item[{\rm(i)}] $f\in\vh;$
\item[{\rm(ii)}] $f\in\cs'(\rn)$ and $M_\varphi(f)\in\lv;$
\item[{\rm(iii)}] $f\in\cs'(\rn)$ and $M_\varphi^0(f)\in\lv.$
\end{enumerate}
Moreover, there exist two positive constants $C_1$ and $C_2$, independent
of $f$, such that
\begin{align*}
\|f\|_{\vh}\le C_1\lf\|M_\varphi^0(f)
\r\|_{\lv}\le C_1\lf\|M_\varphi(f)\r\|_{\lv}
\le C_2\|f\|_{\vh}.
\end{align*}
\end{theorem}

\begin{remark}\label{3r2}
Let $\va:=(a_1,\ldots,a_n)\in [1,\fz)^n$, $a_-:=\min\{a_1,\ldots,a_n\}$,
$a_+:=\max\{a_1,\ldots,a_n\}$ and $\nu:=a_1+\cdots+a_n$.
By Remark \ref{2r2}(iv), we know that $\vh$ with $A$ as in \eqref{2e5}
becomes the Hardy space $H_{\va}^{\vp}(\rn)$
from \cite{cgn17,hlyy,hy}. Recall that, in \cite[Theorem 4.10]{hy}, Huang and Yang
established the maximal function characterizations of $H_{\va}^{\vp}(\rn)$ with $N\in\nn\cap[\lfloor\frac1{p_-}\rfloor+2\nu+3,\fz)$.
On another hand, Cleanthous et al. \cite[Theorem 3.4]{cgn17} also
obtained the maximal function characterizations
of $H_{\va}^{\vp}(\rn)$ with $N\in\mathbb{N}\cap[\lfloor\nu\frac{a_+}{a_-}
(\frac{1}{\min\{1,p_-\}}+1)+\nu+2a_+\rfloor+1,\fz),$
which is just a proper subset of
$\nn\cap[\lfloor\frac1{p_-}\rfloor+2\nu+3,\fz).$
In addition, when $\vp=(\overbrace{p,\ldots,p}^{n\ \rm times})$
with some $p\in(0,\fz)$, from Remark \ref{2r2}(iii),
it follows that Theorem \ref{3t1} is just \cite[p.\,42, Theorem 7.1]{mb03}.
\end{remark}

We now prove Theorem \ref{3t1}.

\begin{proof}[Proof of Theorem \ref{3t1}]
Clearly, (i) implies (ii) and (ii) implies (iii). Therefore, to prove Theorem \ref{3t1},
it suffices to show that (ii) implies (i) and that (iii) implies (ii).

We first prove that (ii) implies (i). For this purpose,
let $\vp\in(0,\fz)^n$, $I\in\nn\cap[\lfloor\frac1{p_-}\rfloor+2n+3,\fz)$,
$f\in\cs'(\rn)$, $\varphi$ be as in Lemma \ref{3l6} and $M_{\varphi}(f)\in\lv$.
We next show that $M_I(f)\in\lv$. Indeed,
by Lemma \ref{3l6} with $N:=\lfloor\frac1{p_-}\rfloor+1$ and $L:=0$, we conclude that,
for any given $I\in\nn\cap[\lfloor\frac1{p_-}\rfloor+2n+3,\fz)$
and any $K\in\zz_+$, $f\in\cs'(\rn)$
and $x\in\rn$, $M_I^{0(K,0)}(f)(x)\ls T_\varphi^{N(K,0)}(f)(x)$.
This, combined with Lemma \ref{3l4}, implies that, for any $K\in\mathbb{Z}_+$
and $f\in\cs'(\rn)$,
\begin{align}\label{3e4}
\lf\|M_I^{0(K,0)}(f)\r\|_{\lv}\ls
\lf\|M_\varphi^{(K,0)}(f)\r\|_{\lv}.
\end{align}
Letting $K\to\fz$ in \eqref{3e4}, by Lemma \ref{3l5}, we obtain
$$\lf\|M_I^0(f)\r\|_{\lv}\ls\lf\|M_\varphi(f)\r\|_{\lv}.$$
This, together with \cite[p.\,17, Proposition 3.10]{mb03},
implies that $M_I(f)\in\lv$ and hence (i) holds true.

Now we show that (iii) implies (ii). To this end, let $M_\varphi^0(f)\in \lv$.
By Lemma \ref{3l7} with $\lz\in(n\ln\lz_+/[p_-\ln b],\fz)$ and $K\in\zz_+$,
we know that there exists some $L\in(0,\fz)$ such that \eqref{3e3} holds true.
Therefore, for any $K\in\zz_+$, $M_\varphi^{(K,L)}(f)\in \lv$.
Indeed, when $\lz\in(n\ln\lz_+/[p_-\ln b],\fz)$, from Lemma \ref{3l3} with
$\theta=\underline{p}$, and Lemmas \ref{3l7} and \ref{3l8}, it follows that
\begin{align*}
\lf\|M_\varphi^{(K,L)}(f)\r\|_{\lv}^{\underline{p}}
&\le\lf\|M_\varphi^{(K,L)}(f){\mathbf 1}_{B_1}\r\|_{\lv}^{\underline{p}}
+\sum_{k\in\nn}\lf\|M_\varphi^{(K,L)}(f)
{\mathbf 1}_{B_{k+1}\setminus B_k}\r\|_{\lv}^{\underline{p}}\\
&\ls\lf\|{\mathbf 1}_{B_1}\r\|_{\lv}^{\underline{p}}+\sum_{k\in\nn}
b^{-\lz k\underline{p}}
\lf\|{\mathbf 1}_{B_{k+1}\setminus B_k}\r\|_{\lv}^{\underline{p}}\\
&\ls\lf\|{\mathbf 1}_{B(\vec 0_n,C)}\r\|_{\lv}^{\underline{p}}+\sum_{k\in\nn}
b^{-\lz k\underline{p}}
\lf\|{\mathbf 1}_{B(\vec 0_n,\,Cb^{k\ln\lz_+/\ln b})}\r\|_{\lv}^{\underline{p}}\\
&\ls \sum_{k\in\zz_+} b^{-\lz k\underline{p}}
b^{k\underline{p}n\ln\lz_+/(p_-\ln b)}<\fz,
\end{align*}
where, for any $r\in(0,\fz)$,
$B(\vec 0_n,r):=\{y\in\rn:\ |y|<r\}$ and $C$ is the
same positive constant as in Lemma \ref{3l8} which
is independent of $k$. Thus, $M_\varphi^{(K,L)}(f)\in \lv$.

In addition, by Lemmas \ref{3l6} and \ref{3l4}, we conclude that, for
any given $L\in(0,\fz)$, there exist some $I\in\nn$ and a positive constant $C_3$
such that, for any $K\in\zz_+$ and $f\in\cs'(\rn)$,
$$\lf\|M_I^{0(K,L)}(f)\r\|_{\lv}\le C_3
\lf\|M_\varphi^{(K,L)}(f)\r\|_{\lv}.$$
For any fixed $K\in\zz_+$, let
\begin{align*}
G_K:=\lf\{x\in\rn:\ M_I^{0(K,L)}(f)(x)\le C_4
M_\varphi^{(K,L)}(f)(x)\r\},
\end{align*}
where $C_4:=2C_3$. Then
\begin{align}\label{3e5}
\lf\|M_\varphi^{(K,L)}(f)\r\|_{\lv}\ls
\lf\|M_\varphi^{(K,L)}(f)\r\|_{L^{\vp}(G_K)},
\end{align}
due to the fact that
$$\lf\|M_\varphi^{(K,L)}(f)\r\|_{L^{\vp}(G_K
^\complement)}\le \frac1 {C_4}\lf\|M_I^{0(K,L)}(f)\r\|
_{L^{\vp}(G_K^\complement)}
\le \frac{C_3}{C_4}\lf\|M_\varphi^{(K,L)}(f)\r\|_{\lv}.$$

For any given $L\in(0,\fz)$, repeating the proof of
\cite[(4.17)]{lyy16} with $p$ therein replaced by $p_-$,
we find that, for any $r\in(0,p_-)$, $K\in\zz_+$,
$f\in\cs'(\rn)$ and $x\in G_K$,
\begin{align*}
\lf[M_\varphi^{(K,L)}(f)(x)\r]^r
\ls M_{{\rm HL}}\lf(\lf[M_\varphi^{0(K,L)}(f)\r]^r\r)(x).
\end{align*}

From this, \eqref{3e5} and Lemmas \ref{3l3} and \ref{3l1}, we further deduce that,
for any $K\in\zz_+$ and $f\in\cs'(\rn)$,
\begin{align}\label{3e7}
\lf\|M_\varphi^{(K,L)}(f)\r\|_{\lv}^r
\ls\lf\|M_\varphi^{0(K,L)}(f)\r\|_{\lv}^r.
\end{align}
Letting $K\to\fz$ in \eqref{3e7}, by the fact that
$r\in(0,\underline{p})$ and Lemma \ref{3l5},
we obtain
\begin{align*}
\lf\|M_\varphi (f)\r\|_{\lv}\ls\lf\|M_\varphi^0(f)\r\|_{\lv}.
\end{align*}
This shows that (iii) implies (ii) and
hence finishes the proof of Theorem \ref{3t1}.
\end{proof}

Applying Lemma \ref{3l1} and Theorem \ref{3t1},
we obtain the following conclusion, which plays an
important role in the proof of Theorem \ref{8t2}
below and is also of independent interest.

\begin{proposition}\label{3p1}
Let $\vp\in(1,\fz)^n$. Then $\vh=\lv$
with equivalent norms.
\end{proposition}

\begin{proof}
Let $\vp\in(1,\fz)^n$. We first show
\begin{align}\label{3e8}
\lv\subset\vh.
\end{align}
To this end, assume that $f\in\lv$ and
$\Phi$ is some fixed infinitely differentiable function on $\rn$ satisfying
$\supp \Phi\subset B_0$ and $\int_{\rn} \Phi(x)\,dx \neq 0$.
Let $M_0(f):=M_\Phi^0(f)$ with $M_\Phi^0(f)$
as in Definition \ref{3d2}. Then, for any $x\in\rn$, we have
$M_\Phi^0(f)(x)\ls \HL(f)(x).$
From this and Lemma \ref{3l1}, it follows that
$$\lf\|M_\Phi^0(f)\r\|_{\lv}\ls\|\HL(f)\|_{\lv}\ls\|f\|_{\lv},$$
which, combined with Theorem \ref{3t1},
further implies that $\|f\|_{\vh}\ls\|f\|_{\lv}$
and hence completes the proof of \eqref{3e8}.

Conversely, let $\varphi\in\cs(\rn)$ and $f\in\vh$.
Without loss of generality, we may assume that
$\int_{\rn} \varphi(x)\,dx=1$. From the assumption
$f\in\vh$ and Theorem \ref{3t1}, it follows that the sequence
$\{f\ast\varphi_{-k}\}_{k\in\nn}$ is bounded in $\lv$,
which, together with \cite[p.\,304, Theorem 1.a)]{bp61}
and \cite[Theorem 3.17]{ru91}, implies that there exists a
subsequence $\{f\ast\varphi_{-k_i}\}_{i\in\nn}$
converging weak-$\ast$ in $\lv$ and hence also in $\cs'(\rn)$.
By this and \cite[p.\,15, Lemma 3.8]{mb03}, we further
conclude that this limit is $f$. Thus, $f\in\lv$.
This implies that $\vh\subset\lv$ and hence finishes the proof of
Proposition \ref{3p1}.
\end{proof}

\section{Atomic characterizations of $\vh$\label{s4}}

In this section, we establish the atomic characterization of $\vh$.
We begin with introducing the notion of anisotropic mixed-norm
$(\vp,r,s)$-atoms as follows. In what follows, for any $q\in(0,\fz]$,
denote by $L^q(\rn)$ the \emph{space of all measurable functions} $f$ such that
$$\|f\|_{L^q(\rn)}:=\lf\{\int_{\rn}|f(x)|^q\,dx\r\}^{1/q}<\fz$$
with the usual modification made when $q=\fz$.

\begin{definition}\label{4d1}
Let $\vp\in(0,\fz)^n$, $r\in(1,\fz]$ and
\begin{align}\label{4e1}
s\in\lf[\lf\lfloor\lf(\dfrac1{p_-}-1\r)
\dfrac{\ln b}{\ln\lambda_-}\r\rfloor,\fz\r)\cap\zz_+.
\end{align}
A measurable function $a$ on $\rn$ is called an
\emph{anisotropic mixed-norm $(\vp,r,s)$-atom} if
\begin{enumerate}
\item[{\rm (i)}] $\supp a \subset B$, where
$B\in\mathfrak{B}$ and $\mathfrak{B}$ is as in \eqref{2e1};
\item[{\rm (ii)}] $\|a\|_{L^r(\rn)}\le \frac{|B|^{1/r}}{\|{\mathbf 1}_B\|_{\lv}}$;
\item[{\rm (iii)}] for any $\gamma\in\zz_+^n$ with $|\gamma|\le s$,
$\int_{\mathbb R^n}a(x)x^\gamma\,dx=0$.
\end{enumerate}
\end{definition}

In what follows, we call an \emph{anisotropic mixed-norm $(\vp,r,s)$-atom} simply by a
\emph{$(\vp,r,s)$-atom}. Now, via $(\vp,r,s)$-atoms, we introduce the following
notion of anisotropic mixed-norm atomic Hardy spaces $\vah$.

\begin{definition}\label{4d2}
Let $\vp\in(0,\fz)^n$, $r\in(1,\fz]$, $s$ be as in \eqref{4e1} and $A$ a dilation.
The \emph{anisotropic mixed-norm atomic Hardy space} $\vah$ is defined to be the
set of all $f\in\cs'(\rn)$ satisfying that there exist $\{\lz_i\}_{i\in\nn}\subset\mathbb{C}$
and a sequence of $(\vp,r,s)$-atoms, $\{a_i\}_{i\in\nn}$, supported, respectively, in
$\{B^{(i)}\}_{i\in\nn}\subset\mathfrak{B}$ such that
$f=\sum_{i\in\nn}\lz_ia_i$ in $\cs'(\rn)$.
Furthermore, for any $f\in\vah$, let
\begin{align*}
\|f\|_{\vah}:={\inf}\lf\|\lf\{\sum_{i\in\nn}
\lf[\frac{|\lz_i|{\mathbf 1}_{B^{(i)}}}{\|{\mathbf 1}_{B^{(i)}}\|_{\lv}}\r]^
{\underline{p}}\r\}^{1/\underline{p}}\r\|_{\lv},
\end{align*}
where the infimum is taken over all decompositions of $f$ as above.
\end{definition}

To establish the atomic characterization of
$\vh$, we need several technical lemmas as follows.
First, Lemma \ref{4l1} is just \cite[p.\,9, Lemma 2.7]{mb03}.

\begin{lemma}\label{4l1}
Let $\Omega\subset\rn$ be an open set and $|\Omega|<\infty$.
Then, for any $t\in\zz_+$, there exist a sequence of points,
$\{x_k\}_{k\in\mathbb{N}}\subset\Omega$, and a sequence of
integers, $\{l_k\}_{k\in\mathbb{N}}$, such that

\begin{enumerate}
\item[{\rm(i)}] $\Omega=\bigcup_{k\in\mathbb{N}}
(x_k+B_{l_k})$;

\item[{\rm(ii)}] $\lf\{x_k+B_{l_k-\omega}\r\}_{k\in\nn}$ are
pairwise disjoint with $\omega$ as in \eqref{2e2};

\item[{\rm(iii)}] for any $k\in\mathbb{N},\
(x_k+B_{l_k+t})\cap\Omega^\complement
=\emptyset$, but $(x_k+B_{l_k+t+1})\cap
\Omega^\complement\neq\emptyset$;

\item[{\rm(iv)}] for any $i$, $j\in\nn$,
$(x_i+B_{l_i+t-2\omega})\cap
(x_j+B_{l_j+t-2\omega})\neq\emptyset$ implies
$|l_i-l_j|\le\omega$;

\item[{\rm(v)}] there exists a positive constant $R$
such that, for any $j\in\nn$,
$$\sharp\lf\{i\in\nn:\ (x_i+B_
{l_i+t-2\omega})\cap(x_j+B_{l_j+t-2\omega})
\neq\emptyset\r\}\le R.$$
\end{enumerate}
\end{lemma}

By Lemma \ref{3l1}, a duality argument of $\lv$ (see \cite[p.\,304, Theorem 2]{bp61}),
a useful result from Sawano \cite[Theorem 1.3]{s05} and a key observation [see \eqref{4e2x} below]
as well as borrowing some ideas from the proof of \cite[Theorem 2.7]{zsy16},
we conclude the following anisotropic
Fefferman--Stein vector-valued inequality on the mixed-norm Lebesgue
space $\lv$.

\begin{lemma}\label{4l3}
Let $\vp\in(1,\fz)^n$ and $u\in(1,\fz]$. Then there exists a positive constant
$C$ such that, for any sequence $\{f_j\}_{j\in\nn}$ of measurable functions,
$$\lf\|\lf\{\sum_{j\in\nn}
\lf[\HL(f_j)\r]^u\r\}^{1/u}\r\|_{\lv}
\le C\lf\|\lf(\sum_{j\in\nn}|f_j|^u\r)^{1/u}\r\|_{\lv}$$
with the usual modification made when $u=\fz$,
where $\HL$ denotes the Hardy--Littlewood maximal operator as in \eqref{3e1}.
\end{lemma}

\begin{proof}
Let $r\in(1,p_-)$. Then, from Lemma \ref{3l3} and \cite[p.\,304, Theorem 2]{bp61},
it follows that there exists some non-negative measurable function
$g\in L^{(\vp/r)'}(\rn)$ with norm belonging to the range $(0,1]$ such that
\begin{align}\label{4e2}
\lf\|\lf\{\sum_{j\in\nn}\lf[\HL(f_j)\r]^u\r\}^{1/u}\r\|_{\lv}
&=\lf\|\lf\{\sum_{j\in\nn}\lf[\HL(f_j)\r]^u\r\}^{r/u}\r\|_{L^{\vp/r}(\rn)}^{1/r}\\
&\sim \lf\{\int_{\rn}\lf(\sum_{j\in\nn}\lf[\HL(f_j)(x)\r]^u\r)^{r/u}g(x)\,dx\r\}^{1/r}.\noz
\end{align}
In addition, by Lemma \ref{3l1}, we know that there exists some constant
$D\in [1,\fz)$ such that, for any $\phi\in L^{(\vp/r)'}(\rn)$,
\begin{align}\label{4e5}
\lf\|\HL(\phi)\r\|_{L^{(\vp/r)'}(\rn)}\le D\|\phi\|_{L^{(\vp/r)'}(\rn)}.
\end{align}
For this $D$ and for any $x\in \rn$, let
\begin{align*}
G(x):=\sum_{k\in\nn}\frac{1}{2^kD^k}M_{\rm HL}^k(g)(x),
\end{align*}
where $M_{\rm HL}^k$ denotes the $k$-fold iteration of the Hardy--Littlewood
maximal operator $M_{\rm HL}$. Then, by \eqref{4e2}, it
is easy to see that
\begin{align}\label{4e3}
\lf\|\lf\{\sum_{j\in\nn}\lf[\HL(f_j)\r]^u\r\}^{1/u}\r\|_{\lv}
\ls \lf\{\int_{\rn}\lf(\sum_{j\in\nn}\lf[\HL(f_j)(x)\r]^u\r)^{r/u}G(x)\,dx\r\}^{1/r}
\end{align}
and, for any $x\in\rn$,
\begin{align}\label{4e1x}
M_{\rm HL}(G)(x)\le2DG(x).
\end{align}
Using \eqref{4e1x}, we further conclude that, for any $j\in\nn$ and $x\in \rn$,
\begin{align}\label{4e2x}
\HL(f_j)(x)
&\ls \sup_{t\in(0,\fz)}\frac1{|B_{\rho}(x,t)|}
\int_{B_{\rho}(x,t)}\lf|f_j(y)\r|\frac{G(y)}{M_{\rm HL}(G)(y)}\,dy\\
&\ls \sup_{t\in(0,\fz)}\frac1{|B_{\rho}(x,t)|}
\int_{B_{\rho}(x,t)}\lf|f_j(y)\r|\frac{G(y)}
{|B_{\rho}(x,22t)|^{-1}\int_{B_{\rho}(x,22t)}G(z)\,dz}\,dy
\ls \mathcal{M}_G(f_j)(x),\noz
\end{align}
where
$$\mathcal{M}_G(f_j)(x)
:=\sup_{t\in(0,\fz)}\frac1{\int_{B_{\rho}(x,22t)}G(z)\,dz}
\int_{B_{\rho}(x,t)}\lf|f_j(y)\r|G(y)\,dy$$
and, for any $x\in\rn$ and $t\in(0,\fz)$, $B_{\rho}(x,t):=\{y\in\rn:\ \rho(x-y)<t\}$.
This, combined with \eqref{4e3}, \cite[Theorem 1.3]{s05},
the H\"{o}lder inequality for mixed-norm Lebesgue spaces
(see \cite[Remark 2.8(iv)]{hlyy}), \eqref{4e5}
and the fact that $\lf\|g\r\|_{L^{(\vp/r)'}(\rn)}\le 1$,
implies that
\begin{align*}
\lf\|\lf\{\sum_{j\in\nn}\lf[\HL(f_j)\r]^u\r\}^{1/u}\r\|_{\lv}
\ls \lf\|\lf\{\sum_{j\in\nn}\lf|f_j\r|^u\r\}^{1/u}\r\|_{\lv},
\end{align*}
which completes the proof of Lemma \ref{4l3}.
\end{proof}

From Lemma \ref{3l1} and an argument similar to that used in the proof of
\cite[Lemma 3.15]{hlyy}, we deduce the following conclusion, which plays
an important role in the proof of Theorem \ref{4t1} below and is also of
independent interest; the details are omitted.

\begin{lemma}\label{4l4}
Let $\vp\in(0,\fz)^n$, $k\in\zz$, $r\in[1,\fz]\cap(p_+,\fz]$
with $p_+$ as in \eqref{2e4} and $A$ be a dilation. Assume that
$\{\lz_i\}_{i\in\nn}\subset\mathbb{C}$, $\{B^{(i)}\}_{i\in\nn}
:=\{x_i+B_{\ell_i}\}_{i\in\nn}\subset\mathfrak{B}$
and $\{a_i\}_{i\in\nn}\subset L^r(\rn)$ satisfy that, for any $i\in\nn$,
$\supp a_i\subset x_i+A^{k}B_{\ell_i}$,
$\|a_i\|_{L^r(\rn)}
\le\frac{|B^{(i)}|^{1/r}}{\|{\mathbf 1}_{B^{(i)}}\|_{\lv}}$
and
$$\lf\|\lf\{\sum_{i\in\nn}
\lf[\frac{|\lz_i|{\mathbf 1}_{B^{(i)}}}{\|{\mathbf 1}_{B^{(i)}}\|_{\lv}}\r]^
{\underline{p}}\r\}^{1/\underline{p}}\r\|_{\lv}<\fz.$$
Then
$$\lf\|\lf[\sum_{i\in\nn}\lf|\lz_ia_i\r|^{\underline{p}}\r]
^{1/\underline{p}}\r\|_{\lv}
\le C\lf\|\lf\{\sum_{i\in\nn}
\lf[\frac{|\lz_i|{\mathbf 1}_{B^{(i)}}}{\|{\mathbf 1}_{B^{(i)}}\|_{\lv}}\r]^
{\underline{p}}\r\}^{1/\underline{p}}\r\|_{\lv},$$
where $\underline{p}$ is as in \eqref{2e4} and $C$ a positive
constant independent of $\{\lz_i\}_{i\in\nn}$, $\{B^{(i)}\}_{i\in\nn}$
and $\{a_i\}_{i\in\nn}$.
\end{lemma}

Via borrowing some ideas from the proof of \cite[Lemma 3.14]{hlyy},
we obtain the following dense subspace of $\vh$.

\begin{lemma}\label{4l5}
Let $\vp\in(0,\fz)^n$ and $N\in\nn\cap[\lfloor(\frac1{\min\{1,p_-\}}-1)
\frac{\ln b}{\ln\lambda_-}\rfloor+2,\fz)$ with $p_-$ as in \eqref{2e4}.
Then $\vh\cap L^{\vp/p_-}(\rn)$ is dense in $\vh$.
\end{lemma}

\begin{proof}
By Remark \ref{2r2}(i), without loss of generality, we may assume
\begin{align}\label{4e1'}
N\in\nn\cap\lf(\max\lf\{\frac{\ln b}{\ln \lz_-},\,\frac{\ln b}{p_-\ln \lz_-},\,
\lf\lfloor\lf(\frac1{\min\{1,p_-\}}-1\r)
\frac{\ln b}{\ln\lambda_-}\r\rfloor+2\r\},\fz\r).
\end{align}
For any $\lambda\in(0,\fz)$ and $f\in\vh$, let
$\Omega_{\lz}:=\lf\{x\in\rn:\,M_N(f)(x)>\lz\r\},$
where $M_N$ is as in Definition \ref{2d4}. Then $|\Omega_{\lz}|<\fz$ and,
applying Lemma \ref{4l1} to $\Omega_\lz$ with $t=4\omega$,
we know that there exist $\{x_i\}_{i\in\nn}\subset\Omega_\lz$ and
$\{l_i\}_{i\in\nn}\subset\zz$ such that
\begin{align}\label{4e4}
\Omega_\lz=\bigcup_{i\in\nn}(x_i+B_{l_i})
\end{align}
and, for any $i\in\nn$,
\begin{align}\label{4e6}
(x_i+B_{l_i+4\omega})\cap\Omega_\lz^\complement
=\emptyset,\hspace{0.2cm} (x_i+B_{l_i+4\omega+1})\cap
\Omega_\lz^\com\neq\emptyset\quad{\rm and}
\end{align}
\begin{align}\label{4e7}
\sharp\lf\{j\in\nn:\
(x_i+B_{l_i+2\omega})\cap(x_j+B_{l_j+
2\omega})\neq\emptyset\r\}\le R,
\end{align}
where $\omega$ is as in \eqref{2e2} and $R$ as in Lemma \ref{4l1}(v).
Moreover, by an argument similar to that used in \cite[pp.\,23-24]{mb03},
we find that there exist
distributions $g^{\lz}$ and $\{b_i^{\lz}\}_{i\in\nn}$ such that
$f=g^{\lz}+\sum_{i\in \nn} b_i^{\lz}$ in $\cs'(\rn).$
From this, \cite[p.\,32, Lemma 5.9]{mb03} with $s=N-1$, Lemma \ref{3l3},
\eqref{4e6}, \eqref{4e7} and \eqref{2e3},
we deduce that
\begin{align*}
&\lf\|M_N\lf(g^{\lz}\r)\r\|_{L^{\vp/p_-}(\rn)}\\
&\hs\ls \lz^{1-p_-}\lf\|M_N(f)\r\|_{\lv}^{p_-}
+\lf\|\lz\sum_{i\in\nn}\sum_{t\in\zz_+}\frac{|x_i+B_{l_i}|^
{\frac{N\ln\lz_-}{\ln b}}}{[\rho(\cdot-x_i)]
^{\frac{N\ln\lz_-}{\ln b}}}
{\mathbf 1}_{x_i+B_{l_i+2\omega+t+1}\setminus B_{l_i+2\omega+t}}\r\|_{L^{\vp/p_-}(\rn)}.
\end{align*}
This, together with Lemmas \ref{3l3} and \ref{4l3},
\eqref{4e1'} and \eqref{4e4}, further implies that
\begin{align*}
\lf\|M_N\lf(g^{\lz}\r)\r\|_{L^{\vp/p_-}(\rn)}
\ls \lz^{1-p_-}\lf\|M_N(f)\r\|_{\lv}^{p_-}<\fz.
\end{align*}
Therefore, $g^{\lz}\in H_A^{\vp/p_-}(\rn)$. By this and
Proposition \ref{3p1}, we further conclude that
$g^{\lz}\in L^{\vp/p_-}(\rn)$.

On another hand, by some arguments similar to those
used in Step 1 of the proof of
\cite[Theorem 4.8]{lwyy17}, we know that,
for any $\lz\in(0,\fz)$ and $x\in\rn$,
$$M_N\lf(\sum_{i\in \nn} b_i^{\lz}\r)(x)
\ls M_N(f)(x){\mathbf 1}_{\Omega_{\lz}}(x)
+\lz\sum_{i\in\nn}\lf[\HL\lf({\mathbf 1}_{x_i+B_{l_i}}\r)(x)\r]
 ^{N\frac{\ln\lz_-}{\ln b}}.$$
From this, Lemmas \ref{3l3} and \ref{4l3} again
as well as \eqref{4e1'}, it follows that
\begin{align}\label{4e11}
\lf\|f-g^{\lz}\r\|_{\vh}
\ls \lf\|M_N(f){\mathbf 1}_{\Omega_{\lz}}\r\|_{\lv}\to 0
\end{align}
as $\lz\to \fz$, which, combined with the assumption $f\in\vh$, implies that
$g^{\lz}\in\vh$ and hence completes the proof of Lemma \ref{4l5}.
\end{proof}

The main result of this section is stated as follows.

\begin{theorem}\label{4t1}
Let $\vp\in(0,\fz)^n$, $r\in(\max\{p_+,1\},\fz]$ with
$p_+$ as in \eqref{2e4}, $s$ be as in \eqref{4e1} and
$N\in\nn\cap[\lfloor(\frac1{\min\{1,p_-\}}-1)\frac{\ln b}{\ln\lambda_-}\rfloor+2,\fz)$
with $p_-$ as in \eqref{2e4}. Then $\vh=\vah$
with equivalent quasi-norms.
\end{theorem}

\begin{remark}\label{4r1}
By Proposition \ref{6p1} below, we know that Theorem \ref{4t1}
with $A$ as in \eqref{2e5} is just \cite[Theorem 3.16]{hlyy}.
Moreover, from Remark
\ref{2r2}(iii), it follows that, when
$\vp=(\overbrace{p,\ldots,p}^{n\ \rm times})$ with some $p\in(0,1]$,
Theorem \ref{4t1} was proved by Bownik in \cite[p.\,39, Theorem 6.5]{mb03}.
\end{remark}

When $\vp\in(1,\fz)^n$,
from the fact that $\vh=\lv$ with
equivalent norms (see Proposition \ref{3p1}) and Theorem \ref{4t1},
we immediately deduce the following conclusion.

\begin{corollary}
Let $\vp\in(1,\fz)^n$, $r\in(p_+,\fz]$ with
$p_+$ as in \eqref{2e4} and $s$ and $N$ be as in Theorem \ref{4t1}.
Then $\lv=\vah$ with equivalent norms.
\end{corollary}

In what follows, for any given $i\in \mathbb{Z}_+$, we use
the \emph{symbol $\mathbb{P}_i(\rn)$} to
denote the linear space of all polynomials
on $\rn$ with degree not greater than $i$.

We now prove Theorem \ref{4t1}.

\begin{proof}[Proof of Theorem \ref{4t1}]
Let $\vp\in(0,\fz)^n$, $r\in(\max\{p_+,1\},\fz]$
and $s$ be as in \eqref{4e1}.
First, we show that
\begin{align}\label{4e8}
\vah\subset\vh.
\end{align}
To this end, for any $f\in\vah$, by Definition \ref{4d2}, we find that
there exist $\{\lz_i\}_{i\in\nn}\subset\mathbb{C}$
and a sequence of $(\vp,r,s)$-atoms, $\{a_i\}_{i\in\nn}$,
supported, respectively, in
$\{B^{(i)}\}_{i\in\nn}\subset\mathfrak{B}$ such that
$f=\sum_{i\in\nn}\lz_ia_i$ in $\cs'(\rn)$ and
\begin{align*}
\|f\|_{\vah}\sim
\lf\|\lf\{\sum_{i\in\nn}
\lf[\frac{|\lz_i|{\mathbf 1}_{B^{(i)}}}{\|{\mathbf 1}_{B^{(i)}}\|_{\lv}}\r]^
{\underline{p}}\r\}^{1/\underline{p}}\r\|_{\lv}.
\end{align*}
Thus, to prove \eqref{4e8}, we only need to show that
\begin{align}\label{4e19}
\|M_N(f)\|_{\lv}\ls\lf\|\lf\{\sum_{i\in\nn}
\lf[\frac{|\lz_i|{\mathbf 1}_{B^{(i)}}}{\|{\mathbf 1}_{B^{(i)}}\|_{\lv}}\r]^
{\underline{p}}\r\}^{1/\underline{p}}\r\|_{\lv}.
\end{align}
To this end, observe that, for any $i\in\nn$,
there exist $l_i\in\zz$ and $x_i\in\rn$
such that $x_i+B_{l_i}=B^{(i)}$.
From an argument similar to that used in the proof of
\cite[(3.27)]{lyy16}, it follows that,
for any $i\in\nn$ and $x\in (x_i+B_{l_i+\omega})^\com$,
$$M_N(a_i)(x)
\ls\frac1{\|{\mathbf 1}_{B^{(i)}}\|_{\lv}}
\lf[\HL\lf({\mathbf 1}_{B^{(i)}}\r)(x)\r]^\gamma,$$
where $\gamma:=(\frac{\ln b}{\ln \lambda_-}+s+1)
\frac{\ln \lambda_-}{\ln b}.$
By this, we easily know that, for any
$N\in\mathbb{N}\cap[\lfloor(\frac1{\min\{1,p_-\}}-1)
\frac{\ln b}{\ln\lambda_-}\rfloor+2,\fz)$ and $x\in\rn$,
\begin{align}\label{3x7}
M_N(f)(x)&\le\sum_{i\in\nn}|\lz_i|M_N(a_i)(x){\mathbf 1}_{x_i+B_{l_i+\omega}}(x)
+\sum_{i\in\nn}|\lz_i|M_N(a_i)(x){\mathbf 1}_{\lf(x_i+B_{l_i+\omega}\r)^\com}(x)\\
&\ls\lf\{\sum_{i\in\nn}\lf[|\lz_i|M_N(a_i)(x){\mathbf 1}_{x_i+B_{l_i+\omega}}(x)\r]
^{\underline{p}}\r\}^{1/\underline{p}}
+\sum_{i\in\nn}\frac{|\lz_i|}{\|{\mathbf 1}_{B^{(i)}}\|_{\lv}}
\lf[\HL\lf({\mathbf 1}_{B^{(i)}}\r)(x)\r]^\gamma\noz\\
&=:{\rm I}+{\rm II}.\noz
\end{align}

For {\rm I}, by the boundedness of $M_N$ on $L^q(\rn)$
with $q\in(1,\fz]$ (see \cite[Remark 2.10]{lyy16}) and
the fact that $r\in(\max\{p_+,1\},\fz]$, we conclude that
\begin{align*}
\lf\|M_N\lf(a_i\r){\mathbf 1}_{x_i+B_{l_i+\omega}}\r\|_{L^r(\rn)}
\ls \lf\|a_i\r\|_{L^r(\rn)}
\ls\frac{|B^{(i)}|^{1/r}}{\|{\mathbf 1}_{B^{(i)}}\|_{\lv}},
\end{align*}
which, together with Lemma \ref{4l4}, further implies that
\begin{align}\label{3x6}
\|{\rm I}\|_{\lv}\ls\|f\|_{\vah}.
\end{align}
For the term ${\rm II}$, applying Lemma \ref{3l3},
the fact that $\gamma>\frac1{p_-}$ and
Lemma \ref{4l3}, we find that
\begin{align*}
\|{\rm II}\|_{\lv}
\sim\lf\|\lf\{\sum_{i\in\nn}\frac{|\lz_i|}
{\|{\mathbf 1}_{B^{(i)}}\|_{\lv}}
\lf[\HL\lf({\mathbf 1}_{B^{(i)}}\r)\r]
^\gamma\r\}^{1/\gamma}\r\|^{\gamma}_{L^{\gamma\vp}(\rn)}
\ls\|f\|_{\vah}.
\end{align*}
This, combined with \eqref{3x7} and \eqref{3x6},
further implies that \eqref{4e19} holds true
and hence finishes the proof of \eqref{4e8}.

We now prove that $\vh\subset\vah$. For this purpose, it suffices
to show that
\begin{align}\label{4e9}
\vh\subset H_A^{\vp,\fz,s}(\rn),
\end{align}
due to the fact that each $(\vp,\infty,s)$-atom is also a $(\vp,r,s)$-atom
and hence $H_A^{\vp,\fz,s}(\rn)\subset\vah$.

Next we prove \eqref{4e9} by two steps.

\emph{Step 1)} In this step, we show that,
for any $f\in\vh\cap L^{\vp/{p_-}}(\rn)$,
\begin{align}\label{4e10}
\|f\|_{H_A^{\vp,\fz,s}(\rn)}\ls\|f\|_{\vh}.
\end{align}

To this end, for any $k\in\zz$,
$N\in\nn\cap[\lfloor(\frac1{\min\{1,p_-\}}-1)\frac{\ln b}{\ln
\lambda_-}\rfloor+2,\fz)$ and $f\in\vh\cap L^{\vp/{p_-}}(\rn)$,
let
$
\Omega_k:=\{x\in\rn:\ M_N(f)(x)>2^k\}.
$
Then $|\Omega_{k}|<\fz$ and, applying Lemma \ref{4l1} to $\Omega_k$ with $t=6\omega$,
it is easy to see that there exist
$\{x_i^k\}_{i\in\nn}\subset\Omega_k$ and $\{\ell_i^k\}_{i\in\nn}\subset\zz$
such that
\begin{align}\label{4e16}
\Omega_k=\bigcup_{i\in\nn}(x_i^k+B_{\ell_i^k});
\end{align}
\begin{align*}
(x_i^k+B_{\ell_i^k-\omega})\cap(x_j^k+B_{\ell_j^k-\omega})
=\emptyset\hspace{0.2cm} {\rm for\ any}\  i,\ j\in\nn\ {\rm with}\ i\neq j;
\end{align*}
\begin{align}\label{4e18}
(x_i^k+B_{\ell_i^k+6\omega})\cap\Omega_k^\complement
=\emptyset,\hspace{0.2cm} (x_i^k+B_{\ell_i^k+6\omega+1})\cap
\Omega_k^\complement\neq\emptyset\ \ \ {\rm for\ any}\
i\in\mathbb{N};
\end{align}
\begin{align*}
(x_i^k+B_{\ell_i^k+4\omega})\cap(x_j^k+B_{\ell_j^k+
4\omega})\neq\emptyset\hspace{0.25cm} {\rm implies}
\hspace{0.25cm} |\ell_i^k-\ell_j^k|\le\omega;
\end{align*}
\begin{align}\label{4e17}
\sharp\lf\{j\in\nn:\
(x_i^k+B_{\ell_i^k+4\omega})\cap(x_j^k+B_{\ell_j^k+
4\omega})\neq\emptyset\r\}\le R\hspace{0.2cm} {\rm for\ any}\  i\in\nn,
\end{align}
where $\omega$ is as in \eqref{2e2} and $R$ as in Lemma \ref{4l1}(v).

Let $\xi\in C^{\fz}(\rn)$ satisfy that $\supp\xi\subset B_\omega$,
$0\le\xi\le1$ and $\xi\equiv1$ on $B_0$. For
any $i\in\nn$, $k\in\zz$ and $x\in\rn$, let
$\xi_i^k(x):=\xi(A^{-\ell_i^k}(x-x_i^k))$
and
\begin{align*}
\eta_i^k(x):=
\frac {\xi_i^k(x)}{\sum_{j\in\nn}\xi_j^k(x)}.
\end{align*}
Then one may easily verify that, for any given $k\in\zz$,
$\{\eta_i^k\}_{i\in\mathbb{N}}$
forms a smooth partition of unity of $\Omega_k$,
namely, for any $i\in\nn$, $\eta_i^k\in C^{\fz}(\rn)$,
$\supp \eta_i^k\subset x_i^k+B_{\ell_i^k+\omega}$,
$0\le\eta_i^k\le1$, $\eta_i^k\equiv1$
on $x_i^k+B_{\ell_i^k-\omega}$ and
${\mathbf 1}_{\Omega_k}=\sum_{i\in\mathbb{N}}\eta_i^k.$

For any $i\in\nn$, $k\in\zz$ and $P\in\mathbb{P}_s(\rn)$,
define the norm in the space $\mathbb{P}_s(\rn)$ by setting
\begin{align}\label{3x8}
\|P\|_{i,k}:=
\lf[\frac1{\int_\rn\eta_i^k(x)\,dx}\int_\rn
|P(x)|^2\eta_i^k(x)\,dx\r]^{1/2},
\end{align}
which makes a finite dimensional Hilbert space
$(\mathbb{P}_s(\rn),\|\cdot\|_{i,k})$. For any $i\in\nn$
and $k\in\zz$, via
$$H\mapsto\frac1{\int_\rn\eta_i^k(x)\,dx}\lf\langle
f,H\eta_i^k\r\rangle,\hspace{0.2cm} H\in\mathbb{P}_s(\rn),$$
the function $f$ induces a linear bounded functional on
$\mathbb{P}_s(\rn)$. Then, applying the Riesz lemma, we find
that there exists a unique polynomial $P_i^k\in\mathbb{P}_s(\rn)$
such that, for any $q\in\mathbb{P}_s(\rn)$,
\begin{align*}
\frac1{\int_\rn\eta_i^k(x)\,dx}
\lf\langle f,q\eta_i^k\r\rangle
&=\frac1{\int_\rn\eta_i^k(x)\,dx}
\lf\langle P_i^{k},q\eta_i^k\r\rangle
=\frac 1{\int_\rn\eta_i^k(x)\,dx}
\int_\rn P_i^{k}(x)q(x)\eta_i^k(x)\,dx.
\end{align*}
For any $i\in\mathbb{N}$ and $k\in\mathbb{Z}$, let
$b_i^{k}:=[f-P_i^{k}]\eta_i^k.$
Notice that, for any $k\in\zz$ and $x\in\rn$,
$\sum_{i\in\nn}
{\mathbf 1}_{x_i^k+B_{\ell_i^k+4\omega}}(x)\le R$
and
$\supp b_i^k\subset x_i^k+B_{\ell_i^k+4\omega}$,
it follows that $\{\sum_{i=1}^w b_i^k\}_{w\in\nn}$
converges in $\cs'(\rn)$. Thus, for any $k\in\zz$,
we define a distribution
\begin{align}\label{4e21}
g_{k}:=f-\sum_{i\in\nn}b_i^{k}
=f-\sum_{i\in\nn}\lf[f-P_i^{k}\r]\eta_i^k
=f{\mathbf 1}_{\Omega_k^\com}+\sum_{i\in\nn}P_i^{k}\eta_i^k.
\end{align}
Then, similarly to \eqref{4e11}, we conclude that
\begin{align}\label{4e22}
\lf\|f-g_k\r\|_{\vh}=\lf\|\sum_{i\in\nn}b_i^k\r\|_{\vh}
\ls\lf\|M_N(f){\mathbf 1}_{\Omega_k}\r\|_{\lv}\to0
\end{align}
as $k\to \fz$.

On another hand, by \eqref{4e21}, we know that
$\|g_{k}\|_{L^{\infty}(\rn)}\ls2^k$ and
$\|g_{k}\|_{L^{\infty}(\rn)}\to0$
as $k\to-\infty$ (see also \cite[p.\,1679]{lyy16}).
By this, \eqref{4e22} and Lemma \ref{4l2},
we find that
$f=\sum_{k\in\zz}[g_{k+1}-g_k]$ in $\cs'(\rn)$.
Furthermore, from an argument similar to that used in \cite[p.\,38]{mb03}
(see also \cite[pp.\,1680-1681]{lyy16}), we deduce that
\begin{align*}
f=\sum_{k\in\zz}\lf[g_{k+1}-g_k\r]
=\sum_{k\in\zz}\sum_{i\in\nn}\lf[b_i^{k}-\sum_{j\in\nn}
\lf(b_j^{k+1}\eta_i^k-P_{i,j}^{k+1}\eta_j^{k+1}\r)\r]
=:\sum_{k\in\zz}\sum_{i\in\nn}h_i^{k}
\hspace{0.5cm} {\rm in}\hspace{0.2cm} \cs'(\rn),
\end{align*}
where, for any $i$, $j\in\nn$ and $k\in\zz$,
$P_{i,j}^{k+1}$ is the orthogonal projection of
$[f-P_j^{k+1}]\eta_i^k$ on $\mathbb{P}_{s}(\rn)$
with respect to the norm defined as in \eqref{3x8}
and $h_i^k$ is a multiple of a
$(\vp,\fz,s)$-atom satisfying
\begin{align}\label{4e12}
\int_\rn h_i^{k}(x)q(x)\,dx=0\hspace{0.4cm}
{\rm for\ any}\ q\in\mathbb{P}_{s}(\rn),
\end{align}
\begin{align}\label{4e14}
\supp h_i^{k}\subset (x_i^k+B_{\ell_i^k+4\omega})
\qquad {\rm and}\qquad
\lf\|h_i^{k}\r\|_{L^{\infty}(\rn)}\le\wz{C}2^k
\end{align}
with $\wz{C}$ being a positive
constant independent of $k$ and $i$.

For any $k\in\zz$ and $i\in\nn$, let
\begin{align}\label{4e15}
\lz_i^k:=\wz{C}2^k\lf\|{\mathbf 1}_{\xik+B_{\ell_i^k+4\omega}}\r\|_{\lv}
\qquad {\rm and}\qquad\aik:=\lf[\lik\r]^{-1}h_i^k,
\end{align}
where $\wz{C}$ is as in \eqref{4e14}. Then, by \eqref{4e12}
and \eqref{4e14}, we know that,
for any $k\in\zz$ and $i\in\nn$, $\aik$ is a $(\vp,\fz,s)$-atom.
Moreover, we have
$f=\sum_{k\in\zz}\sum_{i\in\nn}\lik\aik$ in $\cs'(\rn)$.
In addition, by \eqref{4e15} and \eqref{4e16} through \eqref{4e17},
we conclude that
\begin{align*}
\lf\|\lf\{\sum_{k\in\zz}\sum_{i\in\nn}
\lf[\frac{|\lik|{\mathbf 1}_{\xik+B_{\ell_i^k+4\omega}}}
{\|{\mathbf 1}_{\xik+B_{\ell_i^k+4\omega}}\|_{\lv}}\r]^
{\underline{p}}\r\}^{1/\underline{p}}\r\|_{\lv}
\ls\|f\|_{\vh},
\end{align*}
which implies that \eqref{4e10} holds true.

\emph{Step 2)} In this step, we prove that, for any $f\in\vh$,
\eqref{4e10} also holds true.

For this purpose, let $f\in\vh$. Then, by Lemma \ref{4l5},
we find that there exists a sequence
$\{f_j\}_{j\in\nn}\subset\vh\cap L^{\vp/{p_-}}(\rn)$
such that $f=\sum_{j\in\nn}f_j$ in $\vh$ and, for any $j\in\nn$,
$$\lf\|f_j\r\|_{\vh}\le2^{2-j}\|f\|_{\vh}.$$
Moreover, for any $j\in\nn$, by the conclusion obtained in Step 1,
we conclude that there exist $\{\lz_i^{j,k}\}_{k\in\zz,i\in\nn}
\subset \mathbb{C}$ and a sequence $\{a_i^{j,k}\}_{k\in\zz,i\in\nn}$
of $(\vp,\fz,s)$-atoms such that
$f_j=\sum_{k\in\zz}\sum_{i\in\nn}\lz_i^{j,k}a_i^{j,k}$
in $\cs'(\rn)$. Therefore,
$f=\sum_{j\in\nn}\sum_{k\in\zz}\sum_{i\in\nn}\lz_i^{j,k}a_i^{j,k}$
in $\cs'(\rn)$
and
$$\|f\|_{H_A^{\vp,\fz,s}(\rn)}
\le\lf[\sum_{j\in\nn}\lf\|f_j\r\|_{\vh}^{\underline{p}}\r]^{1/\underline{p}}
\ls\|f\|_{\vh}.$$
This finishes the proof of \eqref{4e10} for any $f\in\vh$ and
hence of Theorem \ref{4t1}.
\end{proof}

From the proof of Theorem \ref{4t1}, we deduce the following
conclusion, which plays an important role in the proof of Theorem
\ref{8t2} below; the details are omitted.

\begin{proposition}\label{4p1}
Let $\vp,\,r,\,s$ and $N$ be as in Theorem \ref{4t1},
$q\in(1,\fz]$ and $f\in\vh\cap L^q(\rn)$.
Then there exist a sequence of $(\vp,r,s)$-atoms, $\{\aik\}_{i\in\nn,k\in\zz}$,
and $\{\lik\}_{i\in\nn,k\in\zz}\subset \mathbb{C}$ such that
$f=\sum_{k\in\zz}\sum_{i\in\nn}\lik\aik$,
where the series converge in $L^q(\rn)$ and almost everywhere.
\end{proposition}

\section{Finite atomic characterizations of $\vh$\label{s5}}

In this section, we obtain the finite atomic characterizations of $\vh$.
To be precise, for any given finite linear combination of $(\vp,r,s)$-atoms
with $r\in(\max\{p_+,1\},\fz)$ [or continuous $(\vp,\fz,s)$-atoms], we show that
its quasi-norm in $\vh$ can be achieved via all its finite combinations of atoms
of the same type. We begin with introducing the notion of anisotropic
mixed-norm finite atomic Hardy spaces $\vfah$ as follows.

\begin{definition}\label{5d1}
Let $\vp\in(0,\fz)^n$, $r\in(1,\fz]$, $s$ be
as in \eqref{4e1} and $A$ a dilation.
The \emph{anisotropic mixed-norm finite atomic Hardy space}
$\vfah$ is defined to be the set of all
$f\in\cs'(\rn)$ satisfying that there exist $I\in\nn$,
$\{\lz_i\}_{i\in[1,I]\cap\nn}\subset\mathbb{C}$ and
a finite sequence of $(\vp,r,s)$-atoms,
$\{a_i\}_{i\in[1,I]\cap\nn}$, supported, respectively, in
$\{B^{(i)}\}_{i\in[1,I]\cap\nn}\subset\mathfrak{B}$
such that $f=\sum_{i=1}^I\lambda_ia_i$ in $\cs'(\rn)$.
Moreover, for any $f\in\vfah$, let
\begin{align*}
\|f\|_{\vfah}:=
{\inf}\lf\|\lf\{\sum_{i=1}^{I}
\lf[\frac{|\lz_i|{\mathbf 1}_{B^{(i)}}}{\|{\mathbf 1}_{B^{(i)}}\|_{\lv}}\r]^
{\underline{p}}\r\}^{1/\underline{p}}\r\|_{\lv},
\end{align*}
where $\underline{p}$ is as in \eqref{2e4} and the
infimum is taken over all decompositions of $f$ as above.
\end{definition}

The following conclusion is from Theorem \ref{4t1} and its proof,
which is needed in the proof of Theorem \ref{5t1} below.

\begin{lemma}\label{5l1}
Let $\vp,\,r$ and $s$ be as in Definition \ref{5d1} and $\omega$ as in \eqref{2e2}. Then,
for any $f\in \vh\cap L^r(\rn)$, there exist
$\{\lz_i^k\}_{k\in\zz,\,i\in\nn}\subset\mathbb{C}$, dilated balls
$\{x_i^k+B_{\ell_i^k}\}_{k\in\zz,\,i\in\nn}\subset\mathfrak{B}$ and
$(\vp,\fz,s)$-atoms $\{\aik\}_{k\in\zz,\,i\in\nn}$ such that
$f=\sum_{k\in\zz}\sum_{i\in\nn}\lz_i^ka_i^k$ in $\cs'(\rn)$,
\begin{align*}
\supp a_i^k\subset \xik+B_{\ell_i^k+4\omega},\hspace{0.2cm}
\Omega_k=\bigcup_{j\in\mathbb{N}}(x_j^k+B_{\ell_j^k+4\omega})\hspace{0.2cm}
for\ any\ k\in\zz\ and\ i\in\nn,
\end{align*}
here $\Omega_k:=\{x\in\rn:\ M_N(f)(x)>2^k\}$ with $N$ as in Definition \ref{2d5},
\begin{align*}
(x_i^k+B_{\ell_i^k-\omega})\cap(x_j^k+B_{\ell_j^k-\omega})
=\emptyset\hspace{0.3cm} for\ any\ k\in\zz\ and\ i,\,j\in\nn\ with\ i\neq j,
\end{align*}
and
\begin{align*}
\sharp\lf\{j\in\mathbb{N}:\
(x_i^k+B_{\ell_i^k+4\omega})\cap(x_j^k+B_{\ell_j^k+
4\omega})\neq\emptyset\r\}\le R\hspace{0.2cm} for\ any\
k\in\zz\  and\  i\in\nn,
\end{align*}
with $R$ being a
positive constant independent of $k,\,i$ and $f$.
Moreover, there exists a positive constant $C$, independent of $f$, such that,
for any $k\in\zz$, $i\in\nn$ and for almost every $x\in\rn$,
$|\lz_i^ka_i^k(x)|\le C2^k$
and
\begin{align}\label{5e5}
\lf\|\lf\{\sum_{k\in\zz}\sum_{i\in\nn}
\lf[\frac{|\lik|{\mathbf 1}_{\xik+B_{\ell_i^k+4\omega}}}
{\|{\mathbf 1}_{\xik+B_{\ell_i^k+4\omega}}\|_{\lv}}\r]^
{\underline{p}}\r\}^{1/\underline{p}}\r\|_{\lv}
\le C\|f\|_{\vh}
\end{align}
with $\underline{p}$ as in \eqref{2e4}.
\end{lemma}

In what follows, denote by $C(\rn)$
the \emph{set of all continuous functions}.
Then we obtain the following finite atomic characterizations
of $\vh$, which actually extends \cite[Theorem 3.1 and Remark 3.3]{msv08}
and \cite[Theorem 5.9]{hlyy} to the present
setting of anisotropic mixed-norm Hardy spaces.

\begin{theorem}\label{5t1}
Let $\vp\in(0,\fz)^n$ and $s$ be as in \eqref{4e1}.
\begin{enumerate}
\item[{\rm (i)}]
If $r\in(\max\{p_+,1\},\fz)$ with $p_+$ as in
\eqref{2e4}, then $\|\cdot\|_{\vfah}$
and $\|\cdot\|_{\vh}$ are equivalent quasi-norms on $\vfah$;
\item[{\rm (ii)}]
$\|\cdot\|_{\vfahfz}$
and $\|\cdot\|_{\vh}$ are equivalent quasi-norms on
$\vfahfz\cap C(\rn)$.
\end{enumerate}
\end{theorem}

\begin{remark}\label{5r1}
By Proposition \ref{6p1} below, we find that, when $A$ is as in
\eqref{2e5}, Theorem \ref{5t1} is just \cite[Theorem 5.9]{hlyy}.
In addition, recall that Bownik et al. in \cite[Theorem 6.2]{blyz08}
established the finite atomic characterization of the weighted
anisotropic Hardy space $H_w^p(\rn;A)$ with $w$
being an anisotropic $\mathbb{A}_{\fz}$ Muckenhoupt weight (see \cite[Definition 2.5]{blyz08}).
As was mentioned in \cite[p.\,3077]{blyz08}, if $w=1$,
then $H_w^p(\rn;A)$ becomes the anisotropic Hardy space $\vAh$ of Bownik \cite{mb03}.
By this and Remark \ref{2r2}(iii), we know that,
when $\vp:=(\overbrace{p,\ldots,p}^{n\ \rm times})$
with $p\in(0,1]$, Theorem \ref{5t1} is just \cite[Theorem 6.2]{blyz08}
with the weight $w=1$.
\end{remark}

By Proposition \ref{3p1} and Theorem \ref{5t1}, we have the following
conclusion with the details omitted.

\begin{corollary}
Let $\vp\in(1,\fz)^n$ and $s$ be as in Theorem \ref{5t1}.
\begin{enumerate}
\item[{\rm (i)}]
If $r\in(p_+,\fz)$ with $p_+$ as in
\eqref{2e4}, then $\|\cdot\|_{\vfah}$ and $\|\cdot\|_{\lv}$
are equivalent quasi-norms on $\vfah$;
\item[{\rm (ii)}]
$\|\cdot\|_{\vfahfz}$
and $\|\cdot\|_{\lv}$ are equivalent quasi-norms on
$\vfahfz\cap C(\rn)$.
\end{enumerate}
\end{corollary}

Now we show Theorem \ref{5t1}.

\begin{proof}[Proof of Theorem \ref{5t1}]
Assume that $\vp\in(0,\fz)^n$, $r\in(\max\{p_+,1\},\fz]$
with $p_+$ as in \eqref{2e4} and $s$ is as in \eqref{4e1}.
Then, from Theorem \ref{4t1}, it follows that
$\vfah\subset \vh$ and, for any $f\in \vfah$,
$\|f\|_{\vh}\ls\|f\|_{\vfah}$.
Thus, to prove Theorem \ref{5t1}, it suffices to show that,
for any $f\in \vfah$ when $r\in(\max\{p_+,1\},\fz)$,
and for any $f\in [\vfahfz\cap C(\rn)]$ when $r=\fz$,
$$\|f\|_{\vfah}\ls\|f\|_{\vh}.$$
We next prove this by three steps.

\emph{Step 1)}
Let $r\in(\max\{p_+,1\},\fz]$ and $f\in \vfah$. Then,
without loss of generality, we may assume that $\|f\|_{\vh}=1$.
From the fact that $f$ has compact support, we deduce that there
exists some $k_0\in\zz$ such that $\supp f\subset B_{k_0}$.
For any $k\in\zz$, let
$\Omega_k:=\{x\in\rn:\ M_N(f)(x)>2^k\},$
here and thereafter in this section, we always let
$N:=\lfloor(1/\min\{1,p_-\}-1)\ln b/\ln\lambda_-\rfloor+2$
with $p_-$ as in \eqref{2e4}.

Notice that
$f\in \vh\cap L^{\widetilde{r}}(\rn)$, where
$\widetilde{r}:=r$ when $r\in(\max\{p_+,1\},\fz)$ and
$\widetilde{r}:=2$ when $r=\fz$. Then,
by Lemma \ref{5l1}, we conclude that there exist $\{\lz_i^k\}_{k\in\zz,\,i\in\nn}
\subset\mathbb{C}$ and a sequence $\{a_i^k\}_{k\in\zz,\,i\in\nn}$
of $(\vp,\fz,s)$-atoms such that
\begin{align}\label{5e9}
f=\sum_{k\in\zz}\sum_{i\in\nn}\lz_i^k a_i^k\quad{\rm in}\quad \cs'(\rn).
\end{align}
From this and an argument similar to that used in the
proof of \cite[(5.13)]{lyy16},
it follows that there exists a positive constant $C_5$ such that,
for any $x\in(B_{k_0+4\omega})^\com$,
\begin{align}\label{5x1}
M_N(f)(x)\le C_5\lf\|{\mathbf 1}_{B_{k_0}}\r\|_{\lv}^{-1}.
\end{align}
Let
\begin{align}\label{5e1}
\wz{k}:=\sup\lf\{k\in\zz:\ 2^k<C_5\lf\|{\mathbf 1}_{B_{k_0}}\r\|_{\lv}^{-1}\r\}
\end{align}
with $C_5$ as in \eqref{5x1}. Now we rewrite \eqref{5e9} as
\begin{align*}
f=\sum_{k=-\fz}^{\wz{k}}\sum_{i\in\nn}\lz_i^k a_i^k+
\sum_{k=\wz{k}+1}^\fz\sum_{i\in\nn}\lz_i^k a_i^k=:h+\ell
\quad{\rm in}\quad \cs'(\rn).
\end{align*}

By this and an argument similar to that used in Step 1 of
the proof of \cite[Theorem 5.9]{hlyy} (see also \cite{lyy16}),
we further conclude that there exists a positive constant $C_6$,
independent of $f$, such that $h/C_6$ is a $(\vp,\fz,s)$-atom
and also a $(\vp,r,s)$-atom for any $\vp\in (0,\fz)^n$,
$r\in(\max\{p_+,1\},\fz]$ and $s$ as in \eqref{3e1}.

\emph{Step 2)} In this step, we show (i). To this end,
for any $K\in(\widetilde{k},\fz)\cap\zz$ and
$k\in[\widetilde{k}+1,K]\cap\zz$, where $\wz{k}$
is as in \eqref{5e1}, let
\begin{align*}
I_{(K,k)}:=\lf\{i\in\nn:\ |i|+|k|\le K\r\}
\hspace{0.3cm} {\rm and}\hspace{0.3cm} \ell_{(K)}:=\sum_{k=\widetilde{k}+1}^{K}
\sum_{i\in I_{(K,k)}}\lz_i^ka_i^k.
\end{align*}
On another hand, for any $r\in(\max\{p_+,1\},\fz)$,
from an argument similar to that used in Step 2 of
the proof of \cite[Theorem 5.9]{hlyy} (see also \cite{lwyy17,lyy16}),
it follows that $\ell_{(K)}$
converges to $\ell$ in $L^r(\rn)$ as $K\to\fz$.
This further implies that, for any given
$\eta\in(0,1)$, there exists a $K\in[\widetilde{k}+1,\fz)\cap\zz$ large enough,
depending on $\eta$, such that $(\ell-\ell_{(K)})/\eta$ is a
$(\vp,r,s)$-atom and hence $f=h+\ell_{(K)}+[\ell-\ell_{(K)}]$
is a finite linear combination of $(\vp,r,s)$-atoms.
From this, Step 1) and \eqref{5e5}, it follows that
$$\|f\|_{\vfah}
\ls C_6+
\lf\|\lf\{\sum_{k=\wz{k}+1}^{K}\sum_{i\in I_{(K,k)}}
\lf[\frac{|\lik|{\mathbf 1}_{\xik+B_{\ell_i^k+4\omega}}}
{\|{\mathbf 1}_{\xik+B_{\ell_i^k+4\omega}}\|_{\lv}}\r]^
{\underline{p}}\r\}^{1/\underline{p}}\r\|_{\lv}
+\eta\ls1,$$
which completes the proof of (i).

\emph{Step 3)}
This step is aimed to prove (ii). To this end, let
$f\in \vfahfz\cap C(\rn)$. From Step 1) above, we deduce
that $f=h+\ell$ in $\cs'(\rn)$, where $h/C_6$ is a $(\vp,\fz,s)$-atom
and $$\ell:=\sum_{k=\wz{k}+1}^\fz\sum_{i\in\nn}\lz_i^k a_i^k.$$
By this and an argument similar to that used in Step 2 of
the proof of \cite[Theorem 5.7]{lyy16} (see also \cite{hlyy,lwyy17}),
we know that, for any given $\epsilon\in(0,\fz)$, $\ell$ can be split
into two parts $\ell_1^{\epsilon}$ and $\ell_2^{\epsilon}$,
where $\ell_1^\epsilon$ is a finite linear combination of continuous
$(\vp,\fz,s)$-atoms and $\ell_2^\epsilon$ satisfies that
$\|\ell_2^\epsilon\|_{\vfahfz}\le \epsilon.$
Therefore,
$$\|f\|_{\vfahfz}
\ls\|h\|_{\vfahfz}
+\lf\|\ell_1^\epsilon\r\|_{\vfahfz}
+\lf\|\ell_2^\epsilon\r\|_{\vfahfz}
\ls1.$$
This finishes the proof of (ii) and hence of Theorem \ref{5t1}.
\end{proof}

\section{Littlewood--Paley function characterizations of $\vh$\label{s6}}

The aim of this section is to characterize $\vh$, respectively,
by means of Lusin area functions, Littlewood--Paley
$g$-functions and Littlewood--Paley $g_\lambda^\ast$-functions.
To this end, we first recall the following Calder\'{o}n reproducing
formula given in \cite[Proposition 2.14]{blyz10}.
In what follows, $C_c^{\fz}(\rn)$ denotes the set of all
\emph{infinitely differentiable functions with compact supports} on $\rn$ and, for any $\varphi\in\cs(\rn)$, $\widehat{\varphi}$ denotes
its \emph{Fourier transform}, namely, for any $\xi\in\rn$,
$\widehat \varphi(\xi) := \int_{\rn} \varphi(x) e^{-2\pi\imath x \cdot \xi} \, dx,$
where $\imath:=\sqrt{-1}$ and, for any $x=(x_1,\ldots,x_n)$,
$\xi=(\xi_1,\ldots,\xi_n)\in\rn$, $x\cdot \xi := \sum_{i=1}^n x_i \xi_i$.

Recall also that $f\in\cs'(\rn)$ is said to
\emph{vanish weakly at infinity} if, for any $\psi\in\cs(\rn)$,
$f\ast\psi_{j}\to0$ in $\cs'(\rn)$ as $j\to-\fz$.
Denote by $\cs'_0(\rn)$ the set of all $f\in\cs'(\rn)$
vanishing weakly at infinity.

\begin{lemma}\label{6l1}
Let $s\in\mathbb{Z_+}$ and $A:=(a_{i,j})_{1\le i,j\le n}$ be a dilation.
For any $\phi\in C_c^{\fz}(\rn)$ satisfying $\supp\phi\subset B_0$,
$\int_{\rn}x^\gamma\phi(x)\,dx=0$ for any $\gamma\in\zz_+^n$ with
$|\gamma|\le s$, and $|\widehat{\phi}(\xi)|\ge C$
for any $\xi\in\{x\in\rn:\ (2\|A\|)^{-1}\le\rho(x)\le 1\}$,
where $C\in(0,\fz)$ is a constant and $\|A\|:=(\sum_{i,j=1}^n|a_{i,j}|^2)^{1/2}$,
there exists a $\psi\in\cs(\rn)$ such that
\begin{enumerate}
\item[{\rm(i)}] $\supp \widehat{\psi}$
is compact and away from the origin;
\item[{\rm(ii)}] for any $\xi\in\rn\setminus\{\vec{0}_n\}$,
$\sum_{k\in\mathbb{Z}}
\widehat{\psi}((A^\ast)^k\xi)\widehat{\phi}((A^\ast)^k\xi)=1$,
where $A^\ast$ denotes the adjoint matrix of $A$.
\end{enumerate}

Moreover, for any $f\in\cs'_0(\rn),\,f=
\sum_{k\in\mathbb{Z}}f\ast\psi_k\ast\phi_k$ in $\cs'(\rn)$.
\end{lemma}

Assume that $\varphi\in\cs(\rn)$ satisfies the same assumptions
as $\phi$ in Lemma \ref{6l1} with $s$ as in \eqref{4e1}.
Recall that, for any $f\in\cs'(\rn)$ and
$\lambda\in(0,\fz)$, the \emph{anisotropic Lusin area function} $S(f)$,
the \emph{anisotropic Littlewood--Paley} $g$-\emph{function} $g(f)$ and
the \emph{anisotropic Littlewood--Paley} $g_\lambda^\ast$-\emph{function} $g_\lambda^\ast(f)$
are defined, respectively, by setting, for any $x\in\rn$,
\begin{align*}
S(f)(x):=\lf[\sum_{k\in\mathbb{Z}}b^{-k}\int_{x+B_k}
\lf|f\ast\varphi_{-k}(y)\r|^2\,dy\r]^{1/2},
\quad
g(f)(x):=\lf[\sum_{k\in\mathbb{Z}}
\lf|f\ast\varphi_{k}(x)\r|^2\r]^{1/2}
\end{align*}
and
\begin{align*}
g_\lambda^\ast(f)(x):=
\lf\{\sum_{k\in\mathbb{Z}}b^{-k}\int_{\rn}
\lf[\frac{b^k}{b^k+\rho(x-y)}\r]^\lambda
\lf|f\ast\varphi_{-k}(y)\r|^2\,dy\r\}^{1/2}
\end{align*}
(see \cite{lfy15}).

The main results of this section are the
succeeding three theorems.

\begin{theorem}\label{6t1}
Let $\vp\in(0,\fz)^n$.
Then $f\in\vh$ if and only if
$f\in\cs'_0(\rn)$ and $S(f)\in\lv$. Moreover,
there exists a constant $C\in[1,\fz)$ such that,
for any $f\in\vh$,
$C^{-1}\|S(f)\|_{\lv}\le\|f\|_{\vh}\le C\|S(f)\|_{\lv}.$
\end{theorem}

\begin{theorem}\label{6t2}
Let $\vp\in(0,\fz)^n$.
Then $f\in\vh$ if and only if
$f\in\cs'_0(\rn)$ and $g(f)\in\lv$. Moreover,
there exists a constant $C\in[1,\fz)$ such that,
for any $f\in\vh$,
$C^{-1}\|g(f)\|_{\lv}\le\|f\|_{\vh}\le C\|g(f)\|_{\lv}.$
\end{theorem}

\begin{theorem}\label{6t3}
Let $\vp\in(0,\fz)^n$ and
$\lambda\in(1+\frac{2}{{\min\{p_-,2\}}}, \fz)$
with $p_-$ as in \eqref{2e4}.
Then $f\in\vh$ if and only if $f\in\cs'_0(\rn)$ and
$g_\lambda^{\ast}(f)\in\lv$. Moreover,
there exists a constant $C\in[1,\fz)$ such that,
for any $f\in\vh$,
$C^{-1}\lf\|g_\lz^\ast(f)\r\|_{\lv}\le\|f\|_{\vh}
\le C\lf\|g_\lz^\ast(f)\r\|_{\lv}.$
\end{theorem}

\begin{remark}\label{6r1}
\begin{enumerate}
\item[{\rm (i)}]
We should point out that
the range of $\lz$ in Theorem \ref{6t3} does not coincide with the best known one,
namely, $\lz\in(2/p,\fz)$ with $p\in(0,1]$, of the $g_\lz^\ast$-function characterization
of the classical Hardy space $H^p(\rn)$ and it is still unclear whether or not the
$g_\lz^\ast$-function, when
$\lz\in(\frac{2}{\min\{p_-,2\}},1+\frac{2}{\min\{p_-,2\}}]$,
can characterize $\vh$ because the method used in the proof of
Theorem \ref{6t3} does not work in this case.
\item[{\rm (ii)}]
Recall that, via the Lusin area function,
the Littlewood--Paley $g$-function or $g_\lambda^\ast$-function,
Li et al. in \cite[Theorems 2.8, 3.1 and 3.9]{lfy15} characterized the
anisotropic Musielak--Orlicz Hardy space $H_A^\varphi(\rn)$ with
$\varphi:\ \rn\times[0,\fz)\to[0,\fz)$ being an anisotropic growth
function (see \cite[Definition 2.3]{lfy15}). As was mentioned in
\cite[p.\,285]{lfy15}, if, for any given $p\in(0,1]$
and for any $x\in\rn$ and $t\in(0,\fz)$,
\begin{align}\label{6e4}
\varphi(x,t):=t^p,
\end{align}
then $H_A^\varphi(\rn)=\vAh$, where $\vAh$ denotes the anisotropic
Hardy space of Bownik \cite{mb03}. From this and Remark \ref{2r2}(iii),
we deduce that, when $\vp:=(\overbrace{p,\ldots,p}^{n\ \rm times})$
with $p\in(0,1]$, Theorems \ref{6t1}, \ref{6t2} and \ref{6t3}
are just \cite[Theorems 2.8, 3.1 and 3.9]{lfy15},
respectively, with $\varphi$ as in \eqref{6e4}.
\end{enumerate}
\end{remark}

The following proposition establishes the relation between $\vh$
and ${H_{\va}^{\vp}(\rn)}$, where ${H_{\va}^{\vp}(\rn)}$ denotes the
anisotropic mixed-norm Hardy space introduced by Cleanthous et al. in
\cite[Definition 3.3]{cgn17}.

\begin{proposition}\label{6p1}
Let $\vp\in(0,\fz)^n$, $\va\in[1,\fz)^n$ and $A$ be as in \eqref{2e5}.
Then $\vh$ and the anisotropic mixed-norm Hardy space ${H_{\va}^{\vp}(\rn)}$
coincide with equivalent quasi-norms.
\end{proposition}

\begin{proof}
For any $\vp\in(0,\fz)^n$, $\va:=(a_1,\ldots,a_n)\in[1,\fz)^n$
and $A$ as in \eqref{2e5}, by Theorem \ref{6t2},
we conclude that
$f\in\vh$ if and only if $f\in\cs'_0(\rn)$ and
$$\lf[\sum_{k\in\mathbb{Z}}
\lf|2^{k\nu}f\ast\varphi(2^{k\va}\cdot)\r|^2\r]^{1/2}
=\lf[\sum_{k\in\mathbb{Z}}
\lf||\det A|^{k}f\ast\varphi(A^{k}\cdot)\r|^2\r]^{1/2}
=g(f)\in \lv,$$
where $\nu:=a_1+\cdots +a_n$,
$2^{k\va} x:=(2^{k a_1}x_1,\ldots,2^{k a_n}x_n)$
for any $k\in\zz$ and $x=(x_1,\ldots,x_n)\in \rn$, and
$\varphi$ satisfies the same assumptions
as $\phi$ in Lemma \ref{6l1} with $s$ as in \eqref{4e1}.
This, combined with \cite[Theorem 4.2]{hlyy},
further implies that $f\in\vh$ if and only if $f\in H_{\va}^{\vp}(\rn)$.
Thus, in this case, $\vh=H_{\va}^{\vp}(\rn)$ with equivalent quasi-norms.
This finishes the proof of Proposition \ref{6p1}.
\end{proof}

\begin{remark}\label{6r2}
Very recently, Huang et al. \cite{hlyy} characterized
$H_{\va}^{\vp}(\rn)$ with $\va\in[1,\fz)^n$ and $\vp\in(0,\fz)^n$
via the Lusin area function,
the Littlewood--Paley $g$-function or $g_\lambda^\ast$-function.
We should point out that, by Proposition \ref{6p1},
when $A$ is as in \eqref{2e5},
Theorems \ref{6t1} through \ref{6t3}
are just, respectively, \cite[Theorems 4.1 through 4.3]{hlyy}.
\end{remark}

To prove Theorem \ref{6t1}, we need several technical lemmas.
The following lemma is just \cite[Lemma 2.3]{blyz10}, which
originates from \cite[Theorem 11]{c90}.

\begin{lemma}\label{6l2}
Let $A$ be a dilation. Then there exists a set
$$\mathcal{Q}:=\lf\{Q_\alpha^k\subset\rn:\ k\in\mathbb{Z},
\,\alpha\in E_k\r\}$$
of open subsets, where $E_k$ is an index set, such that
\begin{enumerate}
\item[{\rm (i)}]for each $k\in\zz$,
$|\rn\setminus\bigcup_{\alpha}Q_\alpha^k|=0$
and, when $\alpha\neq\beta$,
$Q_\alpha^k\cap Q_\beta^k=\emptyset$;
\item[{\rm(ii)}] for any $\alpha,\,\beta,\,k,\,\ell$ with $\ell\ge k$,
either $Q_\alpha^k\cap Q_\beta^\ell=\emptyset$ or
$Q_\alpha^\ell\subset Q_\beta^k$;
\item[{\rm(iii)}] for each $(\ell,\beta)$ and each $k<\ell$,
there exists a unique $\alpha$ such that
$Q_\beta^\ell\subset Q_\alpha^k$;
\item[{\rm(iv)}] there exist some $v\in\zz\setminus\zz_+$
and $u\in\nn$ such that, for any $Q_\alpha^k$
with $k\in\mathbb{Z}$ and $\alpha\in E_k$,
there exists $x_{Q_\alpha^k}\in Q_\alpha^k$
such that, for any $x\in Q_\alpha^k$,
$x_{Q_\alpha^k}+B_{vk-u}
\subset Q_\alpha^k\subset x+B_{vk+u}.$
\end{enumerate}
\end{lemma}

Henceforth, we call
$\mathcal{Q}:=
\{Q_\alpha^k\}_{k\in\mathbb{Z},\,\alpha\in E_k}$
from Lemma \ref{6l2} \emph{dyadic cubes} and
$k$ the \emph{level}, denoted by $\ell(Q_\alpha^k)$,
of the dyadic cube $Q_\alpha^k$
for any $k\in\mathbb{Z}$ and $\alpha\in E_k$.

\begin{remark}\label{6r3}
In Definition \ref{4d1},
if we replace dilated balls $\mathfrak{B}$ by
dyadic cubes, then, by Lemma \ref{6l2}, we know that
the corresponding anisotropic mixed-norm atomic Hardy space
and the original one (see Definition \ref{4d2}) coincide with
equivalent quasi-norms.
\end{remark}

The following lemma is necessary in the proof of Theorem \ref{6t1},
whose proof is similar to that of \cite[Lemma 4.7]{hy};
the details are omitted.

\begin{lemma}\label{6l2'}
Let $\vp\in(0,\fz)^n$, $m\in\zz$, $\vaz\in(0,\fz)$,
$\kappa\in(0,\fz)$, $r\in[1,\fz]\cap(p_+,\fz]$
with $p_+$ as in \eqref{2e4} and $A$ be a dilation. Assume that
$\{\lz_i\}_{i\in\nn}\subset\mathbb{C}$, $\{B^{(i)}\}_{i\in\nn}
:=\{x_i+B_{\ell_i}\}_{i\in\nn}\subset\mathfrak{B}$
and $\{a_i^{\vaz,\kappa}\}_{i\in\nn}\subset L^r(\rn)$ satisfy that, for any $i\in\nn$,
$\supp a_i^{\vaz,\kappa}\subset x_i+A^{m}B_{\ell_i}$,
$\|a_i^{\vaz,\kappa}\|_{L^r(\rn)}
\le\frac{|B^{(i)}|^{1/r}}{\|{\mathbf 1}_{B^{(i)}}\|_{\lv}}$
and
$$\lf\|\lf\{\sum_{i\in\nn}
\lf[\frac{|\lz_i|{\mathbf 1}_{B^{(i)}}}{\|{\mathbf 1}_{B^{(i)}}\|_{\lv}}\r]^
{\underline{p}}\r\}^{1/\underline{p}}\r\|_{\lv}<\fz.$$
Then
$$\lf\|\liminf_{\vaz\to 0^+}\lf[\sum_{i\in\nn}\lf|\lz_ia_i^{\vaz,\kappa}\r|
^{\underline{p}}\r]^{1/\underline{p}}\r\|_{\lv}
\le C\lf\|\lf\{\sum_{i\in\nn}
\lf[\frac{|\lz_i|{\mathbf 1}_{B^{(i)}}}{\|{\mathbf 1}_{B^{(i)}}\|_{\lv}}\r]^
{\underline{p}}\r\}^{1/\underline{p}}\r\|_{\lv},$$
where $\underline{p}$ is as in \eqref{2e4} and $C$ a positive
constant independent of $\{\lz_i\}_{i\in\nn}$, $\{B^{(i)}\}_{i\in\nn}$
$\{a_i^{\vaz,\kappa}\}_{i\in\nn}$, $\vaz$ and $\kappa$.
\end{lemma}

We also need the following conclusion, whose proof is similar
to that of \cite[Lemma 4.8]{hlyy} (see also \cite[Lemma 6.5]{yyyz16});
the details are omitted.

\begin{lemma}\label{6l3}
Let $\vp\in(0,\fz)^n$. Then $\vh\subset\cs'_0(\rn)$.
\end{lemma}

In addition, we have the following lemma, which plays
a key role in the proof of Theorem \ref{6t1} and is also of independent interest.

\begin{lemma}\label{6l1'}
Assume that $E\subset\rn$, $F\in\mathcal{Q}$ with $\mathcal{Q}$ as in Lemma \ref{6l2},
$E\subset F$ and there exists a constant $c_0\in(0,1]$ such that $|E|\ge c_0|F|$.
Then, for any $\vp\in(0,\fz)^n$, there exists a positive constant $C$,
independent of $E$ and $F$, such that
$\frac{\|{\mathbf 1}_F\|_{\lv}}{\|{\mathbf 1}_E\|_{\lv}}\le C.$
\end{lemma}

\begin{proof}
By the fact that $F\in\mathcal{Q}$ and Lemma \ref{6l2}(iv),
we conclude that there exists a dilated ball $B_F\in \mathfrak{B}$
such that, for any $x\in F$, $F\subset x+B_F$
and $|F|\sim|B_F|$. This, together with \eqref{3e1},
implies that, for any $x\in F$,
\begin{align*}
\HL\lf({\mathbf 1}_E\r)(x)&=\sup_{x\in B\in\mathfrak{B}}
\frac1{|B|}\int_{B\cap E}{\mathbf 1}_E(y)\,dy
\ge \frac1{|B_F|}\int_{(x+B_F)\cap E}{\mathbf 1}_E(y)\,dy\\
&\gs \frac1{|F|}\int_{F\cap E}{\mathbf 1}_E(y)\,dy
\gs\frac{|E|}{|F|}\gs c_0.
\end{align*}
From this and Lemma \ref{3l1}, it follows that, for any $\vp\in(0,\fz)^n$
with $p_-$ as in \eqref{2e4},
$$c_0\|{\mathbf 1}_F\|_{L^{2\vp/p_-}(\rn)}\ls \lf\|\HL\lf({\mathbf 1}_E\r)\r\|_{L^{2\vp/p_-}(\rn)}
\ls \|{\mathbf 1}_E\|_{L^{2\vp/p_-}(\rn)}.$$
Thus,
$$\frac{\|{\mathbf 1}_F\|_{\lv}}{\|{\mathbf 1}_E\|_{\lv}}
=\frac{\|{\mathbf 1}_F\|_{L^{2\vp/p_-}(\rn)}^{2/p_-}}
{\|{\mathbf 1}_E\|_{L^{2\vp/p_-}(\rn)}^{2/p_-}}\ls c_0^{-2/p_-},$$
which completes the proof of Lemma \ref{6l1'}.
\end{proof}

To show Theorem \ref{6t1}, we first prove that the $\lv$
quasi-norms of the anisotropic Lusin
area function $S(f)$ are independent of the choices of $\phi$
and $\psi$ as in Lemma \ref{6l1}. For this purpose, we denote by
$S_\phi(f)$ and $S_\psi(f)$ the anisotropic Lusin area functions
defined, respectively, by $\phi$ and $\psi$.

\begin{proposition}\label{6p1'}
Let $\vp\in(0,\fz)^n$, $\phi$ and $\psi$ be as in Lemma \ref{6l1}
with $s$ as in \eqref{4e1}.
Then there exists a positive constant $C$ such that, for any $f\in\cs'_0(\rn)$,
$C^{-1}\|S_\phi(f)\|_{\lv}\le\|S_\psi(f)\|_{\lv}\le C\|S_\phi(f)\|_{\lv}.$
\end{proposition}

To show Proposition \ref{6p1'}, the following lemma is necessary,
whose proof is similar to that of \cite[Lemma 4.6]{hy};
the details are omitted.
In what follows, for any $t\in\mathbb{R}$,
we denote by $\lceil t\rceil$ the \emph{least integer not less
than $t$}.

\begin{lemma}\label{6l3'}
Let $s$ be as in \eqref{4e1}, $v$ and $u$ be as in Lemma \ref{6l2}(iv)
and $r\in(\frac{\ln b}{\ln b+(s+1)\ln \lz_-},1]$. Then
there exists a positive constant $C$ such that, for any
$k,$ $i\in\zz$, $\{c_{Q}\}_{Q\in\mathcal{Q}}\subset [0,\fz)$
with $\mathcal{Q}$ as in Lemma \ref{6l2} and $x\in\rn$,
\begin{align*}
&\sum_{\gfz{\ell(Q)=\lceil\frac{k-u}{v}\rceil\,if\,k\ge u}
{\ell(Q)=\lfloor\frac{k-u}{v}\rfloor\,if\,k< u}}|Q|
\frac{b^{(k\vee i)(s+1) \frac{\ln \lz_-}{\ln b}}}{\lf[b^{(k\vee i)}
+\rho(x-z_{Q})\r]^{(s+1)\frac{\ln \lz_-}{\ln b}+1}}c_{Q}\\
&\hs\hs \le C b^{-[k-(k\vee i)](1/r-1)}\lf\{M_{\rm HL}
\lf(\sum_{\gfz{\ell(Q)=\lceil\frac{k-u}{v}\rceil\,if\,k\ge u}
{\ell(Q)=\lfloor\frac{k-u}{v}\rfloor\,if\,k< u}}
\lf[c_{Q}\r]^r\mathbf{1}_{Q}\r)(x)\r\}^{1/r},
\end{align*}
where $\ell(Q)$ denotes the level of $Q\in\mathcal{Q}$, $z_{Q}\in Q$ and,
for any $k$, $i\in\zz$, $k\vee i:=\max\{k,i\}$.
\end{lemma}

Now we show Proposition \ref{6p1'}.

\begin{proof}[Proof of Proposition \ref{6p1'}]
By symmetry, to prove this proposition,
it suffices to show that, for any $f\in\cs'_0(\rn)$,
\begin{align}\label{6e4'}
\|S_\phi(f)\|_{\lv}\ls\|S_\psi(f)\|_{\lv}.
\end{align}
To this end, for any $i\in\zz$, $x\in\rn$ and $y\in x+B_i$, let
$E_{\phi_i}(f)(y):=f\ast \phi_{-i}(y)$
and, for any $k\in\zz$ and $z\in\rn$,
$$J_{\psi_k}(f)(z):=\lf[b^{-k}\int_{z+B_k}
\lf|f\ast\psi_{-k}(y)\r|^2\,dy\r]^{1/2}.$$
Moreover, for any $Q\in \mathcal{Q}$ with $\mathcal{Q}$
as in Lemma \ref{6l2}, let
$\ell(Q)$ denote the level of $Q$, and $v$ and
$u$ be the constants as in Lemma \ref{6l2}(iv).
Then, from Lemma
\ref{6l1} and the Lebesgue dominated convergence theorem,
it follows that, for any $i\in\zz$, $x\in\rn$ and $y\in x+B_i$,
\begin{align}\label{6e5'}
E_{\phi_i}(f)(y)&=\sum_{k\in\zz}f\ast\psi_{-k}\ast\phi_{-k}\ast\phi_{-i}(y)
=\sum_{k\in\zz}\int_{\rn}f\ast\psi_{-k}(z)\phi_{-k}\ast\phi_{-i}(y-z)\,dz\\
&=\sum_{k\in\zz}\sum_{\gfz{\ell(Q)=\lceil\frac{k-u}{v}\rceil\,if\,k\ge u}
{\ell(Q)=\lfloor\frac{k-u}{v}\rfloor\,if\,k< u}}
\int_{Q}f\ast\psi_{-k}(z)\phi_{-k}\ast\phi_{-i}(y-z)\,dz\noz
\end{align}
in $\cs'(\rn)$. For any $\varphi\in\cs(\rn)$ and $k\in\zz$,
let $\wz{\varphi}_k:=\varphi_{-k}$. By \cite[Lemma 5.4]{blyz10},
we know that, for any $k$, $i\in\zz$ and $x\in\rn$,
$$\lf|\wz{\phi}_k\ast\wz{\phi}_i(x)\r|
\ls b^{-(s+1)|k-i|\frac{\ln \lz_-}{\ln b}}
\frac{b^{(k\vee i)(s+1)\frac{\ln \lz_-}{\ln b}}}{[b^{(k\vee i)}
+\rho(x)]^{(s+1)\frac{\ln \lz_-}{\ln b}+1}}.$$
From this and the fact that $y\in x+B_i$, it follows that
there exists $z_Q\in Q$ such that, for any $k,\ i\in \zz$ and $z\in Q$,
\begin{align}\label{6e6'}
\lf|\wz{\phi}_k\ast\wz{\phi}_i(y-z)\r|
\ls b^{-(s+1)|k-i|\frac{\ln \lz_-}{\ln b}}
\frac{b^{(k\vee i)(s+1)\frac{\ln \lz_-}{\ln b}}}{[b^{(k\vee i)}
+\rho(x-z_{Q})]^{(s+1)\frac{\ln \lz_-}{\ln b}+1}}.
\end{align}
On another hand, note that, for any $k\in\zz$, when $k\ge u$,
$\ell(Q)=\lceil\frac{k-u}{v}\rceil$ and when $k<u$,
$\ell(Q)=\lfloor\frac{k-u}{v}\rfloor$.
Therefore, $B_{v\ell(Q)+u}\subset B_k$.
By this, the H\"{o}lder inequality and Lemma \ref{6l2}(iv),
we find that
$$\lf|\frac1{|Q|}\int_{Q}f\ast\psi_{-k}(y)\,dy\r|
\ls \inf_{z\in Q} J_{\psi_k}(f)(z).$$
From this, \eqref{6e5'}, \eqref{6e6'} and Lemma \ref{6l3'},
we deduce that, for any given $r\in(\frac{\ln b}{\ln b+(s+1)\ln \lz_-},1]$
and any $i\in\zz$, $x\in\rn$ and $y\in x+B_i$,
\begin{align}\label{6e7'}
\lf|E_{\phi_i}(f)(y)\r|
&\ls \sum_{k\in\zz}b^{-(s+1)|k-i|\frac{\ln \lz_-}{\ln b}}
b^{-[k-(k\vee i)](1/r-1)}\\
&\quad\times\lf\{M_{\rm HL}
\lf(\sum_{\gfz{\ell(Q)=\lceil\frac{k-u}{v}\rceil\,if\,k\ge u}
{\ell(Q)=\lfloor\frac{k-u}{v}\rfloor\,if\,k< u}}\inf_{z\in Q}\lf[
J_{\psi_k}(f)(z)\r]^r\mathbf{1}_{Q}\r)(x)\r\}^{1/r}.\noz
\end{align}
Recall that $s$ is as in \eqref{4e1}, we can choose
$r\in(\frac{\ln b}{\ln b+(s+1)\ln \lz_-},\min\{p_-,1\})$
with $p_-$ as in \eqref{2e4}. Thus, from \eqref{6e7'},
we deduce that, for any $x\in\rn$,
\begin{align*}
\lf[S_{\phi}(f)(x)\r]^2
&\ls \sum_{i\in\zz}\Bigg[\sum_{k\in\zz}b^{-(s+1)|k-i|\frac{\ln \lz_-}{\ln b}}
b^{-[k-(k\vee i)](1/r-1)}\\
&\quad\lf.\times\lf\{M_{\rm HL}
\lf(\sum_{\gfz{\ell(Q)=\lceil\frac{k-u}{v}\rceil\,if\,k\ge u}
{\ell(Q)=\lfloor\frac{k-u}{v}\rfloor\,if\,k< u}}\inf_{z\in Q}\lf[
J_{\psi_k}(f)(z)\r]^r\mathbf{1}_{Q}\r)(x)\r\}^{1/r}\r]^2,
\end{align*}
which, combined with the H\"{o}lder inequality
and the fact that $r>\frac{\ln b}{\ln b+(s+1)\ln \lz_-}$,
implies that
\begin{align*}
\lf[S_{\phi}(f)(x)\r]^2
\ls \sum_{k\in\zz}\lf\{M_{\rm HL}
\lf(\lf[J_{\psi_k}(f)\r]^r\r)(x)\r\}^{2/r}.
\end{align*}
Then, by the fact that $r<p_-$
and Lemma \ref{4l3}, we find that
\begin{align*}
\lf\|S_{\phi}(f)\r\|_{\lv}
\ls \lf\|\lf(\sum_{k\in\zz}\lf[J_{\psi_k}(f)\r]^2\r)^{1/2}\r\|_{\lv}
\sim \lf\|S_{\psi}(f)\r\|_{\lv},
\end{align*}
which implies \eqref{6e4'} holds true and
hence completes the proof of Proposition \ref{6p1'}.
\end{proof}

We now prove Theorem \ref{6t1}.

\begin{proof}[Proof of Theorem \ref{6t1}]
Let $\vp\in(0,\fz)^n$, $r\in(\max\{p_+,1\},\fz]$
and $s$ be as in \eqref{4e1}. We first show the necessity of this theorem.
To this end, let $f\in\vh$. Then, from Lemma \ref{6l3},
it follows that $f\in\cs'_0(\rn)$.
On another hand, by Theorem \ref{4t1}, we find that
there exist $\{\lz_i\}_{i\in\nn}\subset\mathbb{C}$
and a sequence of $(\vp,r,s)$-atoms, $\{a_i\}_{i\in\nn}$,
supported, respectively, in
$\{B^{(i)}\}_{i\in\nn}\subset\mathfrak{B}$ such that
$f=\sum_{i\in\nn}\lz_ia_i$ in $\cs'(\rn)$
and
\begin{align*}
\|f\|_{\vh}\sim
\lf\|\lf\{\sum_{i\in\nn}
\lf[\frac{|\lz_i|{\mathbf 1}_{B^{(i)}}}{\|{\mathbf 1}_{B^{(i)}}\|_{\lv}}\r]^
{\underline{p}}\r\}^{1/\underline{p}}\r\|_{\lv}.
\end{align*}
Obviously, for any $i\in\nn$, there exist $l_i\in\zz$ and $x_i\in\rn$
such that $x_i+B_{l_i}=B^{(i)}$.
Moreover, it was proved, in \cite[(6.5)]{lwyy17}
(see also \cite[(5.10)]{lyy17hl}), that, for any $x\in\rn$,
\begin{align*}
S(f)(x)&\le
\sum_{i\in\nn}|\lz_i|S(a_i)(x){\mathbf 1}_{x_i+B_{l_i+q}}(x)
+\sum_{i\in\nn}|\lz_i|S(a_i)(x){\mathbf 1}_{(x_i+B_{l_i+q})^\com}(x)\\
&\ls\lf\{\sum_{i\in\nn}\lf[|\lz_i|S(a_i)(x){\mathbf 1}_{x_i+B_{l_i+q}}(x)\r]
^{\underline{p}}\r\}^{1/\underline{p}}
+\sum_{i\in\nn}\frac{|\lz_i|}{\|{\mathbf 1}_{B^{(i)}}\|_{\lv}}
\lf[\HL({\mathbf 1}_{B^{(i)}})(x)\r]^{\gamma},\noz
\end{align*}
where $q:=u-v+2\omega$ with $u$ and $v$ as in Lemma \ref{6l2},
$\underline{p}$ is as in \eqref{2e4},
$\gamma:=(\frac{\ln b}{\ln \lambda_-}+s+1)
\frac{\ln \lambda_-}{\ln b}$
and $\HL$ denotes the Hardy--Littlewood maximal operator
as in \eqref{3e1}.
From this, the boundedness of $S$ on $L^t(\rn)$ with
$t\in(1,\fz)$ (see \cite[Theorem 3.2]{blyz10}) and
an argument similar to that used in the proof of Theorem \ref{4t1},
we deduce that
$\|S(f)\|_{\lv}\ls\|f\|_{\vh},$
which completes the proof of the necessity of Theorem \ref{6t1}.

Next we prove the sufficiency of this theorem. For this purpose,
let $\psi$ be as in Lemma \ref{6l1}, $f\in\cs'_0(\rn)$ and $S(f)\in\lv$.
Then, by Proposition \ref{6p1'},
we know that $S_{\psi}(f)\in\lv$. Thus, to show the sufficiency of this theorem,
it suffices to show that $f\in\vh$ and
\begin{align}\label{6e6}
\|f\|_{\vh}\ls\|S_{\psi}(f)\|_{\lv}.
\end{align}
To do this, for any $k\in\mathbb{Z}$, let
$\Theta_k:=\{x\in\rn:\ S_{\psi}(f)(x)>2^k\}$ and
$$\mathcal{Q}_k:=\lf\{Q\in\mathcal{Q}:\
|Q\cap\Theta_k|>\frac{|Q|}2\quad{\rm and}\quad
|Q\cap\Theta_{k+1}|\le\frac{|Q|}2\r\}.$$
Obviously, for any $Q\in\mathcal{Q}$,
there exists a unique $k\in\zz$
such that $Q\in\mathcal{Q}_k$.
Let $\{Q_i^k\}_i$ be the set of all \emph{maximal dyadic cubes}
in $\mathcal{Q}_k$,
namely, there exists no $Q\in\mathcal{Q}_k$
such that $Q_i^k\subsetneqq Q$ for any $i$.

Moreover, by Lemmas \ref{6l1}, \ref{6l2} and \ref{6l2'}
and a proof similar to that of the sufficiency of
\cite[Theorem 4.6]{hy} (see also the sufficiency of
\cite[Theorem 4.1]{hlyy} and \cite[Theorem 5.2]{lyy17hl}), we have
$f=\sum_{k\in\zz}\sum_i
\lambda_i^k a_i^k$ in $\cs'(\rn)$,
where, for any $k\in\zz$ and $i$,
$\lik\sim2^k\|{\mathbf 1}_{B_i^k}\|_\lv$ with the
positive equivalence constants independent of $k$ and $i$, and
$a_i^k$ is a $(\vp,r,s)$-atom satisfying,
for any $r\in (\max\{p_+,1\},\fz)$ and
$\gamma\in\zz_+^n$ with $|\gamma|\le s$,
\begin{align*}
\supp a_i^k
\subset B_i^k:=x_{Q_i^k}+B_{v[\ell(Q_i^k)-1]+u+3\omega}\
{\rm with}\ v\ {\rm and}\ u\ {\rm as\ in\ Lemma\ \ref{6l2}(iv)},
\end{align*}
\begin{align*}
\lf\|a_i^k\r\|_{L^r(\rn)}
\le\lf\|{\mathbf 1}_{B_i^k}\r\|_\lv^{-1}\lf|B_i^k\r|^{1/q}
\quad{\rm and}\quad
\int_{\rn}a_i^k(x)x^\gamma\,dx=0.
\end{align*}
By this, Theorem \ref{4t1},
the disjointness of $\{Q_i^k\}_{k\in\mathbb{Z},\,i}$,
Lemma \ref{6l2}(iv), the fact that $|\Qik\cap\Theta_k|\ge\frac{|\Qik|}{2}$
and Lemma \ref{6l1'}, we further conclude that \eqref{6e6} holds true.
This finishes the proof of the sufficiency
and hence of Theorem \ref{6t1}.
\end{proof}

Recall that, for any given dilation $A$,
$\varphi\in\cs(\rn)$, $t\in(0,\fz)$,
$j\in\zz$ and for any $f\in\cs'(\rn)$,
the\emph{ anisotropic Peetre maximal function}
$(\varphi_j^*f)_t$ is defined by setting,
for any $x\in\rn$,
\begin{align*}
\lf(\varphi_j^*f\r)_t(x)
:=\esup_{y\in\rn}\frac{|(\varphi_j\ast f)(x+y)|}
{[1+b^j\rho(y)]^t}
\end{align*}
and the \emph{discrete $g$-function associated with $(\varphi_j^*f)_t$}
is defined by setting, for any $x\in\rn$,
\begin{align*}
g_{t,\ast}(f)(x)
:=\lf\{\sum_{j\in\zz}\lf[\lf(\varphi_j^*f\r)_t(x)\r]^2\r\}^{1/2},
\end{align*}
where, for any $j\in\zz$, $\varphi_j(\cdot):=b^j\varphi(A^j\cdot)$.

The following estimate is just \cite[Lemma 6.9]{lwyy17},
which plays a key role in the proof of Theorem \ref{6t2}
and originates from \cite{lyy17I,u12}.

\begin{lemma}\label{6l4}
Let $\varphi\in\cs(\rn)$ satisfy the same assumptions
as $\phi$ in Lemma \ref{6l1} with $s$ as in \eqref{4e1}.
Then, for any given $N_0\in\nn$ and $r\in(0,\fz)$, there exists a positive
constant $C_{(N_0,r)}$, which depends on $N_0$ and $r$, such that,
for any $t\in(0,N_0)$, $\ell\in\zz$, $f\in\cs'(\rn)$ and $x\in\rn$,
\begin{align*}
\lf[\lf(\varphi^*_\ell f\r)_t(x)\r]^r
\le C_{(N_0,r)}\sum_{k\in\zz_+}b^{-kN_0r}b^{k+\ell}
\int_\rn\frac{|(\varphi_{k+\ell}\ast f)(y)|^r}
{[1+b^\ell\rho(x-y)]^{tr}}\,dy.
\end{align*}
\end{lemma}

We now prove Theorem \ref{6t2}.

\begin{proof}[Proof of Theorem \ref{6t2}]
Let $f\in\vh$. Then, by Lemma \ref{6l3}, we find that
$f\in\cs'_0(\rn)$. Moreover, repeating the proof
of the necessity of Theorem \ref{6t1} with some slight
modifications, we easily know that $g(f)\in \lv$ and
$\lf\|g(f)\r\|_{\lv}\ls\|f\|_{\vh}$.
Therefore, to show Theorem \ref{6t2},
by Theorem \ref{6t1}, it suffices to prove that, for any $f\in\cs'_0(\rn)$
with $g(f)\in\lv$,
\begin{align}\label{6e8}
\|S(f)\|_{\lv}\ls\|g(f)\|_{\lv}.
\end{align}
Indeed, the fact that, for any $f\in\cs'_0(\rn)$, $t\in(0,\fz)$ and
for almost every $x\in\rn$,
$S(f)(x)\ls g_{t,\ast}(f)(x)$, implies that, to show \eqref{6e8}, we only need to prove that,
for any $f\in\cs'_0(\rn)$ and some $t\in(\frac1{\min\{p_-,2\}},\fz)$,
\begin{align}\label{6e9}
\lf\|g_{t,*}(f)\r\|_{\lv}\ls\lf\|g(f)\r\|_{\lv}.
\end{align}

Next we show \eqref{6e9}. Indeed, let $\varphi\in\cs(\rn)$
satisfy the same assumptions
as $\phi$ in Lemma \ref{6l1} with $s$ as in \eqref{4e1}.
Notice that $t\in(\frac1{\min\{p_-,2\}},\fz)$.
Thus, there exists an $r_0\in\lf(0,{\min\{p_-,2\}}\r)$
such that $t\in(\frac1{r_0},\fz)$. Fix $N_0\in(\frac1{r_0},\fz)$.
From this, Lemma \ref{6l4} and the Minkowski inequality,
we deduce that, for any $x\in\rn$,
\begin{align*}
g_{t,*}(f)(x)
\ls\lf\{\sum_{j\in\zz_+}b^{-j(N_0r_0-1)}\lf[\sum_{k\in\zz}b^{{2k}/r_0}
\lf\{\int_\rn\frac{|(\varphi_{j+k}\ast f)(y)|^{r_0}}
{[1+b^k\rho(x-y)]^{tr_0}}\,dy\r\}^{2/r_0}\r]^{r_0/2}\r\}^{1/r_0}.
\end{align*}
This, combined with Lemma \ref{3l3}, implies that
\begin{align*}
&\lf\|g_{t,*}(f)\r\|_{\lv}^{r_0\underline{p}}\\
&\hs\ls\sum_{j\in\zz_+}b^{-j(N_0r_0-1)\underline{p}}
\lf\|\lf\{\sum_{k\in\zz}b^{2k/r_0}
\lf[\sum_{i\in\zz_+}b^{-itr_0}\int_{E_{k,i}(\cdot)}
\lf|(\varphi_{j+k}\ast f)(y)\r|^{r_0}\,dy\r]^{2/r_0}\r\}^{r_0/2}\r\|
_{L^{\vp/r_0}(\rn)}^{\underline{p}},
\end{align*}
where $E_{k,i}(\cdot)$ denotes the set
$\{z\in\rn:\ \rho(\cdot-z)<b^{-k}\}$ when $i=0$,
or the set $\{z\in\rn:\ b^{i-k-1}\le\rho(\cdot-z)<b^{i-k}\}$ when $i\in\nn$.
Therefore, by the Minkowski inequality again and Lemma \ref{4l3},
we find that \eqref{6e9} holds true and hence
complete the proof of Theorem \ref{6t2}.
\end{proof}

Applying Theorems \ref{6t1} and \ref{6t2}, we next show Theorem \ref{6t3}.

\begin{proof}[Proof of Theorem \ref{6t3}]
To prove this theorem, it suffices to show the necessity.
Indeed, the sufficiency of Theorem \ref{6t3} follows from Theorem \ref{6t1}
and the fact that, for any $f\in\cs'(\rn)$ and $x\in\rn$,
$S(f)(x)\ls g_\lz^{\ast}(f)(x)$.

To prove the necessity, let $f\in\vh$. Then, by Lemma \ref{6l3},
we know that $f\in\cs'_0(\rn)$. From the assumption that
$\lambda\in(1+\frac{2}{{\min\{p_-,2\}}}, \fz)$,
it follows that there exists some
$t\in(\frac1{{\min\{p_-,2\}}},\fz)$ such that $\lz\in(1+2t,\fz)$.
Assume that $\varphi\in\cs(\rn)$ satisfies the same assumptions
as $\phi$ in Lemma \ref{6l1} with $s$ as in \eqref{4e1}. Then, for any $x\in\rn$,
we have $g_\lambda^\ast(f)(x)\ls g_{t,*}(f)(x),$
which, together with \eqref{6e9} and Theorem \ref{6t2},
implies that $g_\lambda^\ast(f)\in\lv$ and
$\|g_\lambda^\ast(f)\|_{\lv}
\ls\|f\|_{\vh}.$
This finishes the proof of Theorem \ref{6t3}.
\end{proof}

\section{Dual spaces\label{s7}}

In this section, we give the dual space of $\vh$.
More precisely, as an application of the atomic and finite atomic
characterizations of $\vh$ obtained, respectively, in Theorems \ref{4t1} and \ref{5t1},
we prove that the dual space of $\vh$ is the anisotropic mixed-norm
Campanato space $\mathcal{L}^{A}_{\vec{p},\,r,\,s}(\mathbb{R}^n)$ with
$r\in[1,\fz)$ and $s$ as in \eqref{4e1}.

For this purpose, we first introduce the anisotropic mixed-norm
Campanato space $\lq$.

\begin{definition}\label{7d1}
Let $A$ be a given dilation, $\vp\in(0,\fz]^n,\,q\in[1,\fz]$ and $s\in\zz_+$.
The \emph{anisotropic mixed-norm Campanato space} $\lq$ is defined to be
the set of all measurable functions $f$ such that, when $q\in[1,\fz)$,
$$\|f\|_{\lq}:=\sup_{B\in\mathfrak{B}}\inf_{P\in\cp_s(\rn)}
\frac{|B|}{\|{\mathbf 1}_B\|_{\lv}}\lf[\frac1{|B|}
\int_B\lf|f(x)-P(x)\r|^q\,dx\r]^{1/q}<\fz$$
and
$$\|f\|_{\mathcal{L}^{A}_{\vec{p},\,\fz,\,s}(\rn)}
:=\sup_{B\in\mathfrak{B}}\inf_{P\in\cp_s(\rn)}
\frac{|B|}{\|{\mathbf 1}_B\|_{\lv}}\lf\|f-P\r\|_{L^{\fz}(B)}<\fz,$$
where $\mathfrak{B}$ is as in \eqref{2e1}.
\end{definition}

\begin{remark}\label{7r1}
\begin{enumerate}
\item[{\rm (i)}]
Obviously, $\|\cdot\|_{\lq}$ is a seminorm and
$\cp_s(\rn)\subset \lq$. Indeed, $\|f\|_{\lq}=0$ if and only if $f\in\cp_s(\rn)$.
Therefore, if we identify $f_1$ with $f_2$ when $f_1-f_2\in\cp_s(\rn)$, then
$\lq$ becomes a Banach space. Throughout this article, we always identify
$f\in\lq$ with $\{f+P:\, P\in\cp_s(\rn)\}$.
\item[{\rm (ii)}]
Notice that, when $\vp:=(\overbrace{p,\ldots,p}^{n\ \rm times})$
with some $p\in(0,1]$, then, for any $B\in\mathfrak{B}$,
$\|{\mathbf 1}_B\|_{\lv}=|B|^{1/p}$. We should point out that,
in this case, the Campanato space $\lq$,
in Definition \ref{7d1}, is just the anisotropic Campanato space
$C^{1/p-1}_{q,s}(\rn)$ introduced by Bownik in
\cite[p.\,50, Definition 8.1]{mb03}.
\item[{\rm (iii)}]
Let $A$ be as in \eqref{2e5}. Then the
space $\lq$ is just $\mathcal{L}_{\vp,\,q,\,s}^{\va}(\rn)$ introduced
by Huang et al. in \cite[Definition 3.1]{hlyy18},
which includes the classical isotropic
Campanato space $L_{\frac1p-1,\,q,\,s}(\rn)$
of Campanato \cite{c64} and the
space $\mathop{\mathrm{BMO}}(\rn)$ of John and Nirenberg \cite{jn61}
as special cases (see \cite[Remark 3.2(ii)]{hlyy18}).
\end{enumerate}
\end{remark}

To give the dual space of $\vh$, we also need the following
technical lemma with the details omitted.

\begin{lemma}\label{7l1}
Let $\vp\in (0,1]^n$. Then, for any $\{\lz_i\}_{i\in\nn}\subset\mathbb{C}$
and $\{B^{(i)}\}_{i\in\nn}\subset\mathfrak{B}$,
$$\sum_{i\in\nn}|\lz_i|\le \lf\|\lf\{\sum_{i\in\nn}
\lf[\frac{|\lz_i|{\mathbf 1}_{B^{(i)}}}{\|{\mathbf 1}_{B^{(i)}}\|_{\lv}}\r]^
{\underline{p}}\r\}^{1/{\underline{p}}}\r\|_{\lv},$$
where $\underline{p}$ is as in \eqref{2e4}.
\end{lemma}

\begin{lemma}\label{7l2}
Let $\vp\in (0,1]^n$, $r\in(1,\fz]$ and $s$ be as in \eqref{4e1}.
Then, for any continuous linear functional $L$ on $\vh=\vah$,
\begin{align}\label{7x2}
\lf\|L\r\|_{(\vah)^*}:=\sup\lf\{|L(f)|:\ \|f\|_{\vah}\le 1\r\}
=\sup\lf\{|L(a)|:\ a\ is\ any\ (\vp,r,s){\text-}atom\r\},
\end{align}
here and thereafter, $(\vah)^*$ denotes the dual space of $\vah$.
\end{lemma}

\begin{proof}
For any $(\vp,r,s)$-atom $a$, one may easily show that
$\|a\|_{\vah}\le 1$.
Therefore,
\begin{align}\label{7x1}
\sup\lf\{|L(a)|:\,a\ {\rm is\ any}\ (\vp,r,s){\text-}{\rm atom}\r\}
\le \sup\lf\{|L(f)|:\,\|f\|_{\vah}\le 1\r\}.
\end{align}

Conversely, let $f\in\vh$ and $\|f\|_{\vah}\le 1$.
Then, for any $\varepsilon\in(0,\fz)$, from an argument similar to that used
in the proof of Theorem \ref{4t1},
it follows that there exist $\{\lz_i\}_{i\in\nn}\subset\mathbb{C}$
and a sequence of $(\vp,r,s)$-atoms, $\{a_i\}_{i\in\nn}$,
supported, respectively, in
$\{B^{(i)}\}_{i\in\nn}\subset\mathfrak{B}$ such that
$$
f=\sum_{i\in\nn}\lz_ia_i
~~{\rm in}~~\vh
\quad {\rm and}\quad
\lf\|\lf\{\sum_{i\in\nn}
\lf[\frac{|\lz_i|{\mathbf 1}_{B^{(i)}}}{\|{\mathbf 1}_{B^{(i)}}\|_{\lv}}\r]^
{\underline{p}}\r\}^{1/{\underline{p}}}\r\|_{\lv}\le 1+\varepsilon.
$$
By this, the continuity of $L$ and Lemma \ref{7l1},
we further conclude that
\begin{align*}
|L(f)|&\le \sum_{i\in\nn}|\lz_i||L(a_i)|
\le\lf[\sum_{i\in\nn}|\lz_i|\r]
\sup\lf\{|L(a)|:\,a\ {\rm is\ any}\ (\vp,r,s){\text-}{\rm atom}\r\}\\
&\le(1+\varepsilon)\sup\lf\{|L(a)|:\,a\ {\rm is\ any}\ (\vp,r,s){\text-}{\rm atom}\r\}.
\end{align*}
This, together with the arbitrariness of $\varepsilon\in(0,\fz)$
and \eqref{7x1}, implies that \eqref{7x2} holds  true
and hence finishes the proof of Lemma \ref{7l2}.
\end{proof}

The following conclusion is the main result of this section,
which gives the dual space of $\vh$.

\begin{theorem}\label{7t1}
Let $\vp,\,r$ and $s$ be as in Lemma \ref{7l2}.
Then the dual space of $\vh$, denoted by $(\vh)^*$,
is $\lr$ in the following sense:
\begin{enumerate}
\item[{\rm (i)}] Let $f\in\lr$. Then the linear functional
\begin{align}\label{7v1}
L_f:\,g\longmapsto L_f(g):=\int_{\rn}f(x)g(x)\,dx,
\end{align}
initially defined for any $g\in\vfah$, has a bounded extension to $\vh$.

\item[{\rm (ii)}] Conversely, any continuous linear functional on $\vh$
arises as in \eqref{7v1} with a unique $f\in\lr$.
\end{enumerate}
Moreover, $\|f\|_{\lr}\sim \|L_f\|_{(\vh)^*}$,
where the positive equivalence constants are independent of $f$.
\end{theorem}

\begin{remark}\label{7r2}
\begin{enumerate}
\item[{\rm (i)}]
When $\vp$ is as in Remark \ref{7r1}(ii), by Remarks
\ref{2r2}(ii) and \ref{7r1}(ii), we find that Theorem \ref{7t1}
goes back to \cite[p.\,51, Theorem 8.3]{mb03}.
\item[{\rm (ii)}]
By Proposition \ref{6p1} and Remark \ref{7r1}(iii),
we know that, when $A$ is as in \eqref{2e5},
the spaces $\vh$ and $\lr$ are just, respectively,
${H_{\va}^{\vp}(\rn)}$ and
$\mathcal{L}_{\vp,\,r',\,s}^{\va}(\rn)$.
Thus, in this case,
Theorem \ref{7t1} is just \cite[Theorem 3.10]{hlyy18}, which includes
the conclusion that $(H^p(\rn))^*=L_{\frac1p-1,\,r',\,s}(\rn)$, with $p\in(0,1]$,
of Taibleson and Weiss \cite{tw80} and the distinguished duality result
of Fefferman and Stein \cite{fs72}, namely,
$(H^1(\rn))^*={\mathop{\mathrm{BMO}}}(\rn)$, as special cases
(see also \cite[Remark 3.11(i)]{hlyy18}).
\item[{\rm (iii)}]
When $\vp\in(1,\fz)^n$, from Proposition \ref{3p1}, it follows that
$\vh=\lv$ with equivalent norms, which, combined with
\cite[p.\,304, Theorem 1.a)]{bp61}, further implies that,
for any $\vp\in(1,\fz)^n$, $L^{\vp'}(\rn)$ is the dual space of $\vh$.
However, when
$\vp:=(p_1,\ldots,p_n)\in(0,\fz)^n$ with $p_{i_0}\in(0,1]$ and $p_{j_0}\in(1,\fz)$
for some $i_0,j_0\in\{1,\ldots,n\}$, the dual space of $\vh$ is still unclear so far.
\end{enumerate}
\end{remark}

The following equivalence of the spaces $\lq$ is an immediate
consequence of Theorem \ref{7t1} with the details omitted.

\begin{corollary}\label{7c1}
Let $\vp$ and $s$ be as in Theorem \ref{7t1} and $q\in[1,\fz)$.
Then $\lq=\mathcal{L}^{A}_{\vec{p},\,1,\,s_0}(\rn)$ with equivalent quasi-norms,
where $s_0:=\lfloor(\frac1{p_-}-1)\frac{\ln b}{\ln\lambda_-}\rfloor$ with $p_-$
as in \eqref{2e4}.
\end{corollary}

Now we prove Theorem \ref{7t1}.

\begin{proof}[Proof of Theorem \ref{7t1}]
Using Theorem \ref{4t1}, Lemma \ref{7l1}, Theorem \ref{5t1},
and repeating the proof of (i) of \cite[Theorem 3.10]{hlyy18}
with some slight modifications, we know that
$\lr\subset (\vh)^*$ and (i) holds true.

We next prove $(\vh)^*\subset\lr$. To this end, by Theorem \ref{4t1},
it suffices to show $(\vah)^* \subset \lr.$
To prove this, for any $B\in\mathfrak{B}$,
where $\mathfrak{B}$ is as in \eqref{2e1}, let
$\pi_B:\ L^1(B)\longrightarrow \cp_s(\rn)$
be the natural projection satisfying,
for any $g\in L^1(B)$ and $q\in\cp_s(\rn)$,
\begin{align*}
\int_B\pi_B(g)(x)q(x)\,dx=\int_Bg(x)q(x)\,dx.
\end{align*}
For any $r\in(1,\fz]$ and $B\in\mathfrak{B}$, let
$$L^r_0(B):=\lf\{g\in L^r(B):\ \pi_B(g)=0\ {\rm and}\ g\
{\rm is\ not\ zero\ almost\ everywhere}\r\},$$
where we identify $L^r(B)$ with all the
$L^r(\rn)$ functions vanishing outside $B$.
Thus, for any $g\in L^r_0(B)$,
$a:=\frac{|B|^{1/r}}{\|{\mathbf 1}_B\|_{\lv}}\|g\|_{L^r(B)}^{-1}g$
is a $(\vp,r,s)$-atom. From this and Lemma \ref{4t1},
we deduce that, for any $L\in(\vh)^*=(\vah)^*$ and $g\in L^r_0(B)$,
\begin{align}\label{7e2}
|L(g)|\le \frac{\|{\mathbf 1}_B\|_{\lv}}{|B|^{1/r}}\|L\|_{(\vah)^*}\|g\|_{L^r(B)},
\end{align}
which implies that $L$ is a bounded linear functional on $L^r_0(B)$.
By the Hahn-Banach theorem (see, for instance, \cite[Theorem 3.6]{ru91})
and an argument similar to that used in the proof of
\cite[Theorem 3.10]{hlyy18} with the anisotropic ball
therein replaced by the dilated ball as in \eqref{2e1},
we further conclude that, for any $r\in(1,\fz]$ and $g\in L^{r}_0(B)$,
there exists a unique $\eta\in L^{r'}(B)/{\cp_s(B)}$ such that
\begin{align*}
L(g)=\int_B g(x)\eta(x)\,dx.
\end{align*}

On another hand, for any $k\in\nn$ and $g\in L^r_0(B_k)$ with $r\in(1,\fz]$,
let $f_k\in L^{r'}(B_k)/{\cp_s(B_k)}$
be the unique element such that
$L(g)=\int_{B_{k}} g(x)f_k(x)\,dx.$
Then we easily know that, for any $i,\,k\in\nn$ with $i<k$, $f_k|_{B_i}=f_i$.
By this and the fact that, for any $g\in\vfah$, there exists some $k_0\in\nn$
such that $g\in L^r_0(B_{k_0})$, we find that, for any $g\in\vfah$,
\begin{align}\label{7e3}
L(g)=\int_{\rn}g(x)f(x)\,dx,
\end{align}
where $f(x)=f_{k_0}(x)$ for any $x\in B_{k_0}$ with $k_0\in\nn$.

Now we only need to prove that $f\in\lr$. Indeed,
from \eqref{7e3} and \eqref{7e2},
it follows that, for any $r\in(1,\fz]$ and $B\in\mathfrak{B}$,
\begin{align}\label{7e1'}
\|f\|_{(L^{r}_0(B))^*}
\le \frac{\|{\mathbf 1}_B\|_{\lv}}{|B|^{1/r}}\|L\|_{(\vah)^*}.
\end{align}
Moreover, by \cite[p.\,52, (8.12)]{mb03}, we have
$\|f\|_{(L^{r}_0(B))^*}=\inf_{P\in\cp_s(\rn)}\|f-P\|_{L^{r'}(B)}$.
This, combined with Definition \ref{7d1} and \eqref{7e1'}, further implies that,
for any $r\in(1,\fz]$,
\begin{align*}
\|f\|_{\lr}
=\sup_{B\in\mathfrak{B}}\frac{|B|^{1/r}}{\|{\mathbf 1}_B\|_{\lv}}\|f\|_{(L^{r}_0(B))^*}
\le \|L\|_{(\vah)^*}<\fz,
\end{align*}
which completes the proof of (ii) and hence of Theorem \ref{7t1}.
\end{proof}

\section{Applications to boundedness of operators \label{s8}}

Let $\gamma\in(0,1]$ and $\mathcal{B}_{\gamma}$ be
a $\gamma$-quasi-Banach space. In this section, as applications,
we first establish a criterion on the
boundedness of $\mathcal{B}_{\gamma}$-sublinear operators from
$\vh$ into a quasi-Banach space.
Then, applying this criterion, we further obtain
the boundedness of anisotropic Calder\'on--Zygmund operators
from $\vh$ to itself [or to $\lv$].

Recall that a complete vector space $\mathcal{B}$, equipped with a quasi-norm
$\|\cdot\|_{\mathcal{B}}$, is called a \emph{quasi-Banach space} if
\begin{enumerate}
\item[{\rm (i)}] $\|\varphi\|_{\mathcal{B}}=0$
if and only if $\varphi$ is the zero element of $\mathcal{B}$;
\item[{\rm (ii)}] there exists a positive constant $C\in[1,\fz)$ such that, for any
$\varphi,\,\phi\in\mathcal{B}$,
$\|\varphi+\phi\|_{\mathcal{B}}
\le C(\|\varphi\|_{\mathcal{B}}+\|\phi\|_{\mathcal{B}}).$
\end{enumerate}

On another hand, for any given $\gamma\in(0,1]$,
a \emph{$\gamma$-quasi-Banach space} ${\mathcal{B}_{\gamma}}$
is a quasi-Banach space equipped with a quasi-norm
$\|\cdot\|_{\mathcal{B}_{\gamma}}$ satisfying that
there exists a constant $C\in[1,\fz)$ such that,
for any $K\in \nn$ and $\{\varphi_i\}_{i=1}^{K}\subset\mathcal{B}_{\gamma}$,
$\|\sum_{i=1}^K \varphi_i\|_{\mathcal{B}_{\gamma}}^{\gamma}\le
C \sum_{i=1}^K \|\varphi_i\|_{\mathcal{B}_{\gamma}}^{\gamma}$
(see \cite{ky14, ylk17, zy08}).
Let $\mathcal{B}_{\gamma}$ be a $\gamma$-quasi-Banach space with
$\gamma\in(0,1]$ and $\mathcal{Y}$ a linear space. An operator
$T$ from $\mathcal{Y}$ to $\mathcal{B}_{\gamma}$ is said to be
$\mathcal{B}_{\gamma}$-\emph{sublinear} if
there exists a positive constant $\wz C$ such that, for any $K\in \nn$,
$\{\mu_{i}\}_{i=1}^K\subset \mathbb{C}$ and $\{\varphi_{i}\}_{i=1}^K\subset\mathcal{Y}$,
\begin{align*}
\lf\|T\lf(\sum_{i=1}^K \mu_i \varphi_{i}\r)\r\|_{\mathcal{B}_{\gamma}}^{\gamma}\le
\wz C\sum_{i=1}^K |\mu_i|^{\gamma}\lf\|T(\varphi_{i})\r\|_{\mathcal{B}_{\gamma}}^{\gamma}
\end{align*}
and, for any $\varphi,\,\phi\in \mathcal{Y}$,
$\|T(\varphi)-T(\phi)\|_{\mathcal{B}_{\gamma}}
\le \wz C\|T(\varphi-\phi)\|_{\mathcal{B}_{\gamma}}$
(see \cite{ky14, ylk17, zy08}).

\begin{remark}\label{8r1}
By Lemma \ref{3l3}, we conclude that, for any $\vp\in(0,\fz)^n$,
both $(\lv,\|\cdot\|_{\lv})$ and $(\vh,\|\cdot\|_{\vh})$ are the
$\underline{p}$-quasi-Banach spaces with $\underline{p}$ as in \eqref{2e4}.
\end{remark}

Applying the finite atomic characterizations of $\vh$ obtained
in Section \ref{s5} (see Theorem \ref{5t1}),
we immediately obtain the following
criterion on the boundedness of sublinear operators from $\vh$
into a quasi-Banach space $\mathcal{B}_{\gamma}$,
whose proof is similar to that of
\cite[Theorem 6.2]{hlyy}; the details are omitted.

\begin{theorem}\label{8t1}
Let $\vp\in (0,\fz)^n$, $r\in(\max\{p_+,1\},\fz]$
with $p_+$ as in \eqref{2e4}, $\gamma\in (0,1]$, $s$ be
as in \eqref{4e1} and $\mathcal{B}_{\gamma}$ a $\gamma$-quasi-Banach space.
If either of the following two statements holds true:
\begin{enumerate}
\item[{\rm (i)}] $r\in(\max\{p_+,1\},\fz)$ and
$T:\ \vfah\to\mathcal{B}_{\gamma}$
is a $\mathcal{B}_{\gamma}$-sublinear operator satisfying that
there exists a positive constant $C$ such that,
for any $f\in \vfah$,
\begin{equation*}
\|T(f)\|_{\mathcal{B}_{\gamma}}\le C\|f\|_{\vfah};
\end{equation*}
\item[{\rm (ii)}]
$T:\ \vfahfz\cap C(\rn)\to\mathcal{B}_{\gamma}$
is a $\mathcal{B}_{\gamma}$-sublinear operator satisfying that
there exists a positive constant $C$ such that,
for any $f\in \vfahfz\cap C(\rn)$,
$\|T(f)\|_{\mathcal{B}_{\gamma}}\le C\|f\|_{\vfahfz},$
\end{enumerate}
then $T$ uniquely extends to a bounded $\mathcal{B}_{\gamma}$-sublinear
operator from $\vh$ into $\mathcal{B}_{\gamma}$.
\end{theorem}

\begin{remark}\label{8r2}
By Remark \ref{2r2}(iv) and Proposition \ref{6p1},
we easily know that Theorem \ref{8t1} with $A$ as
in \eqref{2e5} is just \cite[Theorem 6.2]{hlyy}.
\end{remark}

The following Corollary \ref{8c1}, as a consequence of Theorem \ref{8t1}, extends the
corresponding results of Meda et al. \cite[Corollary 3.4]{msv08},
Grafakos et al. \cite[Theorem 5.9]{gly08} and
Ky \cite[Theorem 3.5]{ky14} (see also \cite[Theorem 1.6.9]{ylk17})
as well as Huang et al. \cite[Corollary 6.3]{hlyy} to the present
setting; the details are omitted.

\begin{corollary}\label{8c1}
Let $\vp\in(0,1]^n$, $r\in(1,\fz]$ and $\gamma$, $s$ and $\mathcal{B}_{\gamma}$ be as
in Theorem \ref{8t1}. If either of the following two statements holds true:
\begin{enumerate}
\item[{\rm (i)}] $r\in(1,\fz)$ and $T$ is a $\mathcal{B}_{\gamma}$-sublinear
operator from $\vfah$ to $\mathcal{B}_{\gamma}$ satisfying
\begin{equation*}
\sup\lf\{\|T(a)\|_{\mathcal{B}_{\gamma}}:\
a\ is\ any\ (\vp,r,s){\text-}atom\r\}<\fz;
\end{equation*}
\item[{\rm(ii)}] $T$ is a $\mathcal{B}_{\gamma}$-sublinear
operator defined on all continuous $(\vp,\fz,s)$-atoms satisfying
$$\sup\lf\{\|T(a)\|_{\mathcal{B}_{\gamma}}:\
a\ is\ any\ continuous\ (\vp,\fz,s){\text-}atom\r\}<\fz,$$
\end{enumerate}
then $T$ uniquely extends to a bounded $\mathcal{B}_{\gamma}$-sublinear
operator from $\vh$ into $\mathcal{B}_{\gamma}$.
\end{corollary}

Next, using the obtained boundedness criterion, Theorem \ref{8t1},
we establish the boundedness of anisotropic Calder\'{o}n--Zygmund operators
from $\vh$ to itself [or to $\lv$]. We begin with recalling the
notion of anisotropic convolutional
$\delta$-type Calder\'{o}n--Zygmund operators from \cite{lyy16}.

\begin{definition}\label{8d1}
For any $\delta\in (0,\frac{\ln \lz_+}{\ln b}]$,
a linear operator $T$ is called an \emph{anisotropic convolutional
$\delta$-type Calder\'{o}n--Zygmund operator} if $T$
is bounded on $L^2(\rn)$ with kernel
$k\in \cs'(\rn)$ coinciding with a locally integrable
function on $\rn\setminus\{\vec{0}_n\}$ and satisfying that
there exists a positive constant $C$ such that,
for any $x,\,y\in \rn$ with $\rho(x)>b^{2\omega}\rho(y)$,
\begin{align}\label{8e5}
|k(x-y)-k(x)|\le C\frac{[\rho(y)]^{\delta}}{[\rho(x)]^{1+\delta}}
\end{align}
and, for any $f\in L^2(\rn)$, $T(f)(x):={\rm p.\,v.}\,k\ast f(x)$.
\end{definition}

We then introduce the anisotropic non-convolutional
$\beta$-order Calder\'{o}n--Zygmund operators as follows.
In what follows, for any $\Omega\subset \rn\times\rn$
and $r\in\zz_+$, denote by $C^r(\Omega)$
the set of all functions on $\Omega$ whose derivatives with order
not greater than $r$ are continuous.

\begin{definition}\label{8d2}
Let $\beta\in (0,\fz)$. A linear operator
$T$ is called an \emph{anisotropic non-convolutional
$\beta$-order Calder\'{o}n--Zygmund operator}
if $T$ is bounded on $L^2(\rn)$ and its kernel
$$\mathcal{K}:\ \Theta:=\{(x,y)\in \rn\times\rn:\ x\neq y\}\to \mathbb{C}$$
satisfies that $\mathcal{K}\in C^{\lceil \bz\rceil-1}(\Theta)$ and
there exists a positive constant $C$ such that, for any $\az\in\zz_+^n$ with
$|\alpha|= \lceil \bz\rceil-1$ and $x,\,y,\,z\in \rn$
with $\rho(x-y)>b^{2\omega}\rho(y-z)\neq 0$,
\begin{align}\label{8e4}
&\lf|\lf[\partial^{\az}\mathcal{K}(x,\cdot)\r](y)
-\lf[\partial^{\az}\mathcal{K}(x,\cdot)\r](z)\r|\\
&\hs\le C\frac{[\rho(y-z)]^{\frac{\ln \lz_+}{\ln b}\bz}}
{[\rho(x-y)]^{1+\frac{\ln \lz_+}{\ln b}\bz}}
\min\lf\{[\rho(y-z)]^{-\frac{\ln \lz_+}{\ln b}(\lceil \bz\rceil-1)},\,
[\rho(y-z)]^{-\frac{\ln \lz_-}{\ln b}(\lceil \bz\rceil-1)}\r\}\noz
\end{align}
and, for any $f\in L^2(\rn)$ with compact support and $x\notin\supp f$,
$$T(f)(x)=\int_{\supp f}\mathcal{K}(x,y)f(y)\,dy.$$
\end{definition}

Recall that, for any $m\in\nn$, an operator $T$ is
said to have the \emph{vanishing moments
up to order $m$} if, for any $f\in L^2(\rn)$ with compact support and
satisfying that, for any $\alpha\in\zz_+^n$ with $|\alpha|\le m$,
$\int_{\rn}x^{\alpha}f(x)\,dx=0$, it holds true that $\int_{\rn}x^{\alpha}T(f)(x)\,dx=0$.

One of the main results of this section is stated as follows.

\begin{theorem}\label{8t2}
Let $\delta\in(0,\frac{\ln \lz_+}{\ln b}]$, $\vp\in (0,\fz)^n$
and $p_-\in(\frac1{1+\delta},\fz)$
with $p_-$ as in $\eqref{2e4}$. Assume that $T$ is an anisotropic convolutional
$\dz$-type Calder\'on--Zygmund operator as in Definition \ref{8d1}. Then there
exists a positive constant $C$ such that, for any $f\in \vh$,
\begin{enumerate}
\item[{\rm (i)}]
$\|T(f)\|_{\lv}\le C\|f\|_{\vh}$;
\item[{\rm (ii)}]
$\|T(f)\|_{\vh}\le C\|f\|_{\vh}$.
\end{enumerate}
\end{theorem}

\begin{remark}\label{8r3}
\begin{enumerate}
\item[{\rm (i)}]
If $A:=d\,{\rm I}_{n\times n}$ for some $d\in\mathbb R$ with $|d|\in(1,\fz)$,
then $\frac{\ln\lambda_+}{\ln b}=\frac1n$ and $\vh$ becomes the
isotropic mixed-norm Hardy spaces $H^{\vp}(\rn)$.
By Theorem \ref{8t2}, we know that, in this case, if
$\delta\in(0,1]$, $\vp\in(0,\fz)^n$ with $p_-\in(\frac n{n+\delta},\fz)$
and $T$ is the classical isotropic Calder\'on--Zygmund operator,
namely, $T$ satisfying
all the conditions of Definition \ref{8d1} with \eqref{8e5} replaced by
\begin{align*}
|k(x-y)-k(x)|\le C \frac{|y|^\dz}{|x|^{n+\dz}},
\hspace{0.3cm} \forall\,|x|> 2|y|\neq0,
\end{align*}
where $C$ is a positive constant independent of
$x$ and $y$, then $T$ is bounded from $H^{\vp}(\rn)$ to itself [or to $\lv$].
We point out that, even in this case, the results obtained in Theorem \ref{8t2}
are also new. Moreover, when $\vp:=(\overbrace{p,\ldots,p}^{n\ \rm times})$
with $p\in (\frac n{n+\delta},\fz)$, by Theorem \ref{8t2}, we know that
$T$ is bounded from $H^p(\rn)$ to itself [or to $L^p(\rn)$],
which is a well-known result (see, for instance, \cite{s93}).

\item[{\rm (ii)}]
Let $\va:=(a_1,\ldots,a_n)\in[1,\fz)^n$. If $A$ is as in \eqref{2e5},
then $\frac{\ln\lambda_+}{\ln b}=\frac{a_+}{\nu}$ with
$a_+:=\max\{a_1,\ldots,a_n\}$ and $\nu:=a_1+\cdots +a_n$
and, by Proposition \ref{6p1}, $\vh$ becomes the Hardy space $H_{\va}^{\vp}(\rn)$
from \cite{cgn17,hlyy}. In this case, Theorem \ref{8t2} implies that,
if $\delta\in(0,a_+]$, $\vp\in(0,\fz)^n$ with
$p_-\in(\frac\nu{\nu+\delta},\fz)$ and
$T$ is an anisotropic convolutional $\dz$-type
Calder\'on--Zygmund operator as in Definition \ref{8d1}
with \eqref{8e5} replaced by
$$|k(x-y)-k(x)|\le C\frac{|y|_{\va}^{\delta}}{|x|_{\va}^{\nu+\delta}}
\quad{\rm when}\quad |x|_{\va}>2|y|_{\va}\neq0,$$
where $|\cdot|_{\va}$ denotes the anisotropic quasi-homogeneous norm as in
\cite[Definition 2.1]{hlyy} and $C$ a positive constant independent of $x$ and $y$,
then $T$ is bounded from $H_{\va}^{\vp}(\rn)$ to itself [or to $\lv$],
which is just the conclusions obtained in \cite[Theorems 6.4 and 6.5]{hlyy}.
\end{enumerate}
\end{remark}

In addition, as a direct corollary of
Theorem \ref{8t2} and Proposition \ref{3p1}, we have the
following boundedness of anisotropic convolutional $\dz$-type
Calder\'on--Zygmund operators on the mixed-norm Lebesgue space $\lv$ with
$\vp\in(1,\fz)^n$.

\begin{corollary}\label{8c2}
Let $\vp\in(1,\fz)^n$, $\delta\in(0,\frac{\ln \lz_+}{\ln b}]$
and $T$ be an anisotropic convolutional $\dz$-type
Calder\'on--Zygmund operator as in Definition \ref{8d1}.
Then $T$ is bounded on $\lv$.
\end{corollary}

\begin{remark}
We should point out that,
when $\vp:=(\overbrace{p,\ldots,p}^{n\ \rm times})$ with $p\in (1,\fz)$,
the mixed-norm Lebesgue space $\lv$ becomes the classical
Lebesgue space $L^p(\rn)$. Then, by Corollary \ref{8c2},
we know that the anisotropic convolutional $\dz$-type
Calder\'on--Zygmund operator $T$ as in Definition \ref{8d1}
is bounded on $L^p(\rn)$ for any given $p\in(1,\fz)$,
which is also obtained in \cite[p.\,60]{mb03}.
\end{remark}

Now we prove Theorem \ref{8t2}.

\begin{proof}[Proof of Theorem \ref{8t2}]
By similarity, we only show (ii) by two steps.

\emph{Step 1)} In this step, we show that (ii) holds true for any
$\vp\in(0,2)^n$ with $p_-\in(\frac1{1+\delta},2)$.
To this end, let $s$ be as in \eqref{4e1},
$f\in H^{\vp,2,s}_{A,\rm {fin}}(\rn)$  and $T$ be an anisotropic
convolutional $\dz$-type Calder\'on--Zygmund operator as in Definition \ref{8d1}.
Then, by Theorem \ref{5t1}(i), without loss of generality, we may assume that $\|f\|_{\vh}=1$.
Thus, to prove (ii), by Remark \ref{8r1} and Theorem \ref{8t1}(i), it suffices to show
\begin{align}\label{8e6}
\|T(f)\|_{\vh}\ls 1.
\end{align}
Notice that $f\in \vh\cap L^2(\rn)$.
By Proposition \ref{4p1},
we find that there exist a sequence of $(\vp,2,s)$-atoms,
$\{a_{i}\}_{i\in\mathbb{N}}$, supported, respectively, in
$\{x_i+B_{l_i}\}_{i\in\nn}\subset \mathfrak{B}$ and
$\{\lz_{i}\}_{i\in\mathbb{N}}\subset\mathbb{C}$ such that
\begin{align}\label{8e7}
f=\sum_{i\in\mathbb{N}}\lambda_{i}a_{i}\quad{\rm in}\quad L^2(\rn)
\end{align}
and
\begin{align*}
\lf\|\lf\{\sum_{i\in\nn}
\lf[\frac{|\lz_{i}|{\mathbf 1}_{x_i+B_{l_i}}}
{\|{\mathbf 1}_{x_i+B_{l_i}}\|_{\lv}}\r]^
{\underline{p}}\r\}^{1/\underline{p}}\r\|_{\lv}
\lesssim\|f\|_{\vh}\ls 1
\end{align*}
with $\underline{p}$ as in \eqref{2e4}.
By \eqref{8e7} and the boundedness of $T$ on $L^2(\rn)$,
we conclude that, for any $f\in H^{\vp,2,s}_{\va,\rm {fin}}(\rn)$,
$T(f)=\sum_{i\in\mathbb{N}}\lambda_{i}T(a_{i})$ in $ L^2(\rn)$
and hence in $\cs'(\rn)$. From this and Theorem \ref{3t1}, it follows that
\begin{align}\label{8e9}
\|T(f)\|_{\vh}&\sim\|M_0(T(f))\|_{\lv}\\
&\ls \lf\|\sum_{i\in\mathbb{N}}|\lambda_{i}|M_0(T(a_{i}))
{\mathbf 1}_{x_i+B_{l_i+\omega}}\r\|_{\lv}+\lf\|
\sum_{i\in\mathbb{N}}|\lambda_{i}|M_0(T(a_{i}))
{\mathbf 1}_{(x_i+B_{l_i+\omega})^\com}\r\|_{\lv}\noz\\
&=:\textrm{I}_1+\textrm{I}_2,\noz
\end{align}
here and thereafter, for any $f\in\cs'(\rn)$, $M_0(f):=M_\Phi^0(f)$ with $M_\Phi^0(f)$
as in Definition \ref{3d2}, where
$\Phi$ is some fixed $C^\fz(\rn)$ function satisfying
$\supp \Phi\subset B_0$ and $\int_{\rn} \Phi(x)\,dx \neq 0$.

For the term ${\rm I}_1$, from Lemma \ref{3l1} and
the fact that $T$ is bounded on $L^2(\rn)$,
we deduce that, for any $i\in\nn$,
\begin{align*}
\lf\|M_0\lf(T(a_i)\r){\mathbf 1}_{x_i+B_{l_i+\omega}}\r\|_{L^2(\rn)}
&\ls \lf\|M_{\rm {HL}}(T(a_i)){\mathbf 1}_{x_i+B_{l_i+\omega}}\r\|_{L^2(\rn)}
\ls\frac{|x_i+B_{l_i}|^{1/2}}{\|{\mathbf 1}_{x_i+B_{l_i}}\|_{\lv}},
\end{align*}
where $M_{\rm {HL}}$ is as in \eqref{3e1}. By this,
the fact that $p_+\in(0,2)$
and Lemma \ref{4l4}, we further conclude that
\begin{align}\label{8e10}
{\rm I}_1
\ls \lf\|\lf\{\sum_{i\in\mathbb{N}}\lf[|\lambda_{i}|M_0(T(a_{i}))
{\mathbf 1}_{x_i+B_{l_i+\omega}}\r]^{\underline{p}}\r\}^{1/\underline{p}}\r\|_{\lv}
\ls\lf\|\lf\{\sum_{i\in\nn}
\lf[\frac{|\lz_{i}|{\mathbf 1}_{x_i+B_{l_i}}}
{\|{\mathbf 1}_{x_i+B_{l_i}}\|_{\lv}}\r]^
{\underline{p}}\r\}^{1/\underline{p}}\r\|_{\lv}
\ls 1.
\end{align}

To deal with ${\rm I}_2$, for any $j\in\zz$,
let $k^{(j)}:=k\ast\Phi_j$, where $k$ is the kernel of $T$.
Then $k^{(j)}$ satisfies the same conditions as $k$.
Indeed, for any $j\in\zz$ and $f\in L^2(\rn)$,
by the boundedness of $T$ on $L^2(\rn)$ and the Minkowski
inequality of integrals, we know that, for any $j\in\zz$,
\begin{align*}
\lf\|k^{(j)}*f\r\|_{L^2(\rn)}&=\lf\|k*\Phi_j*f\r\|_{L^2(\rn)}
=\lf\|k*(\Phi_j*f)\r\|_{L^2(\rn)}
\ls \lf\|\Phi_j*f\r\|_{L^2(\rn)}\ls\|f\|_{L^2(\rn)}.
\end{align*}
Moreover, via an argument similar to that
used in the proof of \cite[p.\,117, Lemma]{s93},
we find that, for any $x,\,y\in \rn$ with
$\rho(x)>b^{2\omega}\rho(y)$,
$$\lf|k^{(j)}(x-y)-k^{(j)}(x)\r|\ls \frac{[\rho(y)]^{\delta}}{[\rho(x)]^{1+\delta}}.$$
Therefore, from the vanishing moment of $a_i$ and the H\"{o}lder
inequality, it follows that, for any $x\in(x_i+B_{l_i+\omega})^\com$,
$$M_0(T(a_{i}))(x)
\ls\frac1{\|{\mathbf 1}_{x_i+B_{l_i}}\|_{\lv}}
\lf[M_{\rm HL}\lf({\mathbf 1}_{x_i+B_{l_i}}\r)(x)\r]
^{1+\delta}.$$
By this, Lemma \ref{3l3}, the fact that $p_-\in(\frac1{1+\delta},2)$
and Lemma \ref{4l3}, we conclude that
\begin{align*}
\textrm{I}_2
\ls \lf\|\sum_{i\in\mathbb{N}}\frac{|\lambda_{i}|}
{\|{\mathbf 1}_{x_i+B_{l_i}}\|_{\lv}}
\lf[M_{\rm HL}\lf({\mathbf 1}_{x_i+B_{l_i}}\r)\r]^{1+\delta}\r\|_{\lv}
\ls \lf\|\lf\{\sum_{i\in\nn}
\lf[\frac{|\lz_{i}|{\mathbf 1}_{x_i+B_{l_i}}}
{\|{\mathbf 1}_{x_i+B_{l_i}}\|_{\lv}}\r]^
{\underline{p}}\r\}^{1/\underline{p}}\r\|_{\lv}
\ls 1,
\end{align*}
which, together with \eqref{8e9} and \eqref{8e10}, further implies that
\eqref{8e6} holds true and hence completes the proof of Step 1).

\emph{Step 2)} In this step, we prove that (ii) holds true for any
$\vp\in(0,\fz)^n$ and $p_-\in(\frac1{1+\delta},\fz)$.
From the conclusion obtained in Step 1), we deduce that, for any given $\vp\in(0,2)^n$
with $p_-\in(\frac1{1+\delta},2)$ and for any $f\in\vh$, $\|T(f)\|_{\vh}\ls\|f\|_{\vh}$,
which, combined with Proposition \ref{3p1}, further implies that,
for any given $p\in(1,2)$ and for any $f\in L^p(\rn)$,
\begin{align}\label{8e20}
\|T(f)\|_{L^p(\rn)}\ls\|f\|_{L^p(\rn)}.
\end{align}
On another hand, note that the adjoint operator of $T$, denoted by $T^*$,
has the kernel $k^*(\cdot)=\overline{k(-\cdot)}$ which also satisfies Definition
\ref{8d1}. By this, the H\"{o}lder inequality and \eqref{8e20},
we conclude that, for any given $p\in(2,\fz)$ and for any $f\in L^p(\rn)$,
\begin{align*}
\|T(f)\|_{L^p(\rn)}&=\sup_{\|g\|_{L^{p'}(\rn)}\le 1}\lf|\int_{\rn}T(f)(x)g(x)\,dx\r|
=\sup_{\|g\|_{L^{p'}(\rn)}\le 1}\lf|\langle T(f),\,g\rangle\r|
=\sup_{\|g\|_{L^{p'}(\rn)}\le 1}\lf|\langle f,\,T^*(g)\rangle\r|\\
&\le \sup_{\|g\|_{L^{p'}(\rn)}\le 1}\lf\| f\r\|_{L^p(\rn)}
\lf\| T^*(g)\r\|_{L^{p'}(\rn)}
\ls\sup_{\|g\|_{L^{p'}(\rn)}\le 1}\lf\| f\r\|_{L^p(\rn)}
\lf\| g\r\|_{L^{p'}(\rn)}
\ls \lf\| f\r\|_{L^p(\rn)},
\end{align*}
which, together with \eqref{8e20},
implies that, for any given $p\in(1,\fz)$ and for any $f\in L^p(\rn)$,
\begin{align*}
\|T(f)\|_{L^p(\rn)}\ls\|f\|_{L^p(\rn)}.
\end{align*}
Using this and repeating the proof of Step 1) with some slight
modifications, we further find that, for any given $\vp\in(0,\fz)^n$ with
$p_-\in(\frac1{1+\dz},\fz)$, (ii) holds true.
This finishes the proof of Step 2) and hence of Theorem \ref{8t2}.
\end{proof}

Another main result of this section is stated as follows.

\begin{theorem}\label{8t3}

Let $\bz\in(0,\fz)$ and $\vp\in(0,2)^n$ with
$p_-\in(\frac{\ln b}{\ln b+\bz\ln\lz_-},\frac{\ln b}
{\ln b+(\lceil \bz\rceil-1) \ln \lz_-}]$,
where $p_-$ is as in $\eqref{2e4}$.
\begin{enumerate}
\item[{\rm (i)}]
If $T$ is an anisotropic non-convolutional $\beta$-order
Calder\'{o}n--Zygmund operator, then there exists a positive constant $C$
such that, for any $f\in \vh$,
$$\|T(f)\|_{\lv}\le C\|f\|_{\vh}.$$
\item[{\rm (ii)}]
If $T$ is an anisotropic non-convolutional
$\beta$-order Calder\'{o}n--Zygmund operator having the
vanishing moments up to order $\lceil\bz\rceil-1$,
then there exists a positive constant $C$ such that, for any $f\in \vh$,
$$\|T(f)\|_{\vh}\le C\|f\|_{\vh}.$$
\end{enumerate}
\end{theorem}

\begin{remark}\label{8r4}
\begin{enumerate}
\item[{\rm (i)}]
Notice that, differently from Theorem \ref{8t2}, in Theorem \ref{8t3},
we have an additional restriction on the range of $\vp$, namely, $\vp\in(0,2)^n$,
which is caused by the fact that, for any given $r\in(1,\fz)\setminus\{2\}$,
we do not know whether or not the boundedness of the anisotropic non-convolutional
$\beta$-order Calder\'{o}n--Zygmund operator $T$ on $L^r(\rn)$ as in Definition \ref{8d2} holds true.
If $T$ is bounded on $L^r(\rn)$ for any given $r\in(1,\fz)$, then, by an argument
similar to that used in Step 1) of the proof of Theorem \ref{8t2},
we can improve this restriction into the same restriction $\vp\in(0,\fz)^n$
as in Theorem \ref{8t2}.

\item[{\rm (ii)}]
Let $\bz:=\delta\in(0,1)$.
Then, in Definition \ref{8d2}, $\az=(\overbrace{0,\ldots,0}^{n\ \mathrm{times}})$
and the operator $T$ becomes an anisotropic non-convolutional
$\delta$-type Calder\'{o}n--Zygmund operator. Thus, from Theorem \ref{8t3},
it follows that, for any given $\delta\in(0,1)$ and $\vp\in(0,2)^n$ with
$p_-\in(\frac{\ln b}{\ln b+\delta\ln\lz_-},1]$,
the anisotropic non-convolutional $\delta$-type Calder\'{o}n--Zygmund
operator is bounded from $\vh$ to itself [or to $\lv$].

\item[{\rm (iii)}]
Let $\va:=(a_1,\ldots,a_n)\in[1,\fz)^n$. If $A$ is as in \eqref{2e5},
then, by Proposition \ref{6p1}, we find that $\vh$ becomes the Hardy space
$H_{\va}^{\vp}(\rn)$ from \cite{cgn17,hlyy}, $\frac{\ln\lambda_+}{\ln b}=\frac{a_+}{\nu}$ and
$\frac{\ln\lambda_-}{\ln b}=\frac{a_-}{\nu}$ with
$a_+:=\max\{a_1,\ldots,a_n\}$, $a_-:=\min\{a_1,\ldots,a_n\}$ and $\nu:=a_1+\cdots +a_n$,
and $T$ becomes an anisotropic non-convolutional
$\beta$-order Calder\'on--Zygmund operator as in Definition \ref{8d2}
with \eqref{8e4} replaced by that, for any $\az\in\zz_+^n$ with
$|\alpha|=\lceil \bz\rceil-1$ and $x,\,y,\,z\in \rn$ with
$|x-y|_{\va}>2|y-z|_{\va}\neq0$,
$$
\lf|\lf[\partial^{\az}\mathcal{K}(x,\cdot)\r](y)
-\lf[\partial^{\az}\mathcal{K}(x,\cdot)\r](z)\r|\le
C\frac{|y-z|_{\va}^{\bz a_+}}{|x-y|_{\va}^{\nu+\bz a_+}}
\min\lf\{|y-z|_{\va}^{-(\lceil \bz\rceil-1)a_+}
,\,|y-z|_{\va}^{-(\lceil \bz\rceil-1)a_-}\r\},
$$
where $|\cdot|_{\va}$ is as in Remark \ref{8r3}(ii)
and $C$ a positive constant independent of $x$, $y$ and $z$.
In this case, by Theorem \ref{8t3}, we know that, for any
given $\va\in[1,\fz)^n$,
$\bz\in(0,\fz)$ and $\vp\in(0,2)^n$ with
$p_-\in(\frac{\nu}{\nu+\bz a_-},\frac{\nu}{\nu+(\lceil \bz\rceil-1) a_-}]$,
$T$ is bounded from $H_{\va}^{\vp}(\rn)$ to itself [or to $\lv$],
which includes \cite[Theorems 6.8 and 6.9]{hlyy} as a special case.

\item[{\rm (iv)}]
We should also point out that the boundedness of anisotropic non-convolutional
$\bz$-order Calder\'{o}n--Zygmund operators on the mixed-norm Lebesgue
space $\lv$ with any given $\vp\in(1,\fz)^n$ is still unknown so far.
\end{enumerate}
\end{remark}

We now prove Theorem \ref{8t3}.

\begin{proof}[Proof of Theorem \ref{8t3}]
By similarity, we only prove (ii).
Let $\{\lz_i\}_{i\in\nn}$ and $\{a_i\}_{i\in\nn}$
be the same as in the proof of Theorem \ref{8t2}.
From an argument similar to that used in Step 1) of the proof of Theorem \ref{8t2},
we deduce that, to prove (ii), we only need to show
\begin{align}\label{8e11}
\lf\|\sum_{i\in\mathbb{N}}|\lambda_{i}|M_0(T(a_{i}))\r\|_{\lv}\ls 1,
\end{align}
where $M_0$ is as in the proof of Theorem \ref{8t2}.
To this end, first, it is easy to see that
\begin{align}\label{8e19}
\lf\|\sum_{i\in\mathbb{N}}|\lambda_{i}|M_0(T(a_{i}))\r\|_{\lv}
&\ls \lf\|\sum_{i\in\mathbb{N}}|\lambda_{i}|M_0(T(a_{i}))
{\mathbf 1}_{x_i+B_{l_i+4\omega}}\r\|_{\lv}\\
&\hs+\lf\|\sum_{i\in\mathbb{N}}|\lambda_{i}|M_0(T(a_{i}))
{\mathbf 1}_{(x_i+B_{l_i+4\omega})^{\com}}\r\|_{\lv}
=:\Sigma_1+\Sigma_2,\noz
\end{align}
where, for any $i\in\nn$, $x_i+B_{l_i}$
is the same as in the proof of Theorem \ref{8t2}
and $\omega$ as in \eqref{2e2}.

Similarly to \eqref{8e10}, we have $\Sigma_1\ls 1$.
To deal with $\Sigma_2$,
by the vanishing moments of $T$ and the fact that
$\lceil \bz\rceil-1\le(\frac1{p_-}-1)\frac{\ln b}{\ln\lambda_-}$,
which implies $\lceil \bz\rceil-1\le s$, we find that, for any
$i\in\nn,\,k\in\zz$ and $x\in (x_i+B_{l_i+4\omega})^\com$,
\begin{align}\label{8e12}
&\lf|\Phi_k*T(a_i)(x)\r|\\
&\hs\le b^k\lf[\int_{\{y\in\rn:\ \rho(y-x_i)<b^{l_i+2\omega}\}}+
\int_{\{y\in\rn:\ b^{l_i+2\omega}\le\rho(y-x_i)<b^{-2\omega}\rho(x-x_i)\}}+
\int_{\{y\in\rn:\ \rho(y-x_i)\ge b^{-2\omega}\rho(x-x_i)\}}\r]\noz\\
&\hs\hs\times \lf|\Phi\lf(A^k(x-y)\r)-\sum_{|\az|\le \lceil \bz\rceil-1}
\frac{\partial^{\az}\Phi(A^k(x-x_i))}{\az!}\lf(A^k(y-x_i)\r)^{\az}\r|\lf|T(a_i)(y)\r|\,dy\noz\\
&\hs=:\textrm{J}_1+\textrm{J}_2+\textrm{J}_3,\noz
\end{align}
where $\Phi$ is as in the proof of Theorem \ref{8t2}.

For $\textrm{J}_1$, from Taylor's remainder theorem,
\eqref{2e1'} and \eqref{2e2'},
we deduce that, for any $i\in\nn$, $m\in\zz_+$, $k\in\zz$,
$x\in (x_i+B_{l_i+4\omega})^\com$ and $y\in\rn$ with $\rho(y-x_i)<b^{l_i+2\omega}$,
there exists $\eta_1(y)\in x_i+B_{l_i+2\omega}$ such that
\begin{align}\label{8e13}
\textrm{J}_1
&\ls b^k\int_{\{y\in\rn:\ \rho(y-x_i)<b^{l_i+2\omega}\}}
\frac{1}{[\rho(A^k(x-x_i))]^m}\\
&\hs\times\max\lf\{\lf[\rho(A^k(y-x_i))\r]^{\lceil \bz\rceil
\frac{\ln\lambda_+}{\ln b}},\,
\lf[\rho(A^k(y-x_i))\r]^{\lceil \bz\rceil\frac{\ln\lambda_-}{\ln b}}\r\}
\lf|T(a_i)(y)\r|\,dy\noz,
\end{align}
where the implicit positive constant depends on $m$.
When $\rho(A^k(x-x_i))\ge 1$, let
\begin{align*}
m:=\lf\{
\begin{array}{cl}
\vspace{0.1cm}
&\lf\lfloor1+\lceil \bz\rceil\dfrac{\ln\lambda_+}{\ln b}\r\rfloor+1
\hspace{0.5cm} {\rm when}\hspace{0.5cm} \rho(A^k(y-x_i))\in [1,\fz),\\
&\lf\lfloor1+\lceil \bz\rceil\dfrac{\ln\lambda_-}{\ln b}\r\rfloor+1
\hspace{0.5cm}{\rm when}\hspace{0.5cm} \rho(A^k(y-x_i))\in [0,1)
\end{array}\r.
\end{align*}
in \eqref{8e13}. By this, Definition \ref{2d2}(ii), the H\"{o}lder
inequality, the fact that $T$ is bounded on $L^2(\rn)$ and the size
condition of $a_i$, we know that,
for any $i\in\nn$ and $x\in (x_i+B_{l_i+4\omega})^\com$,
\begin{align}\label{8e14}
\textrm{J}_1
\ls&\int_{\{y\in\rn:\ \rho(y-x_i)<b^{l_i+2\omega}\}}
\max\lf\{\frac{[\rho(y-x_i)]^{\lceil \bz\rceil\frac{\ln\lambda_+}{\ln b}}}
{[\rho(x-x_i)]^{1+\lceil \bz\rceil\frac{\ln\lambda_+}{\ln b}}},\,
\frac{[\rho(y-x_i)]^{\lceil \bz\rceil\frac{\ln\lambda_-}{\ln b}}}
{[\rho(x-x_i)]^{1+\lceil \bz\rceil\frac{\ln\lambda_-}{\ln b}}}\r\}\\
&\times\lf|T(a_i)(y)\r|\,dy\noz\\
\ls&\frac1{\|{\mathbf 1}_{x_i+B_{l_i}}\|_{\lv}}
\lf[M_{\rm HL}\lf({\mathbf 1}_{x_i+B_{l_i}}\r)(x)\r]
^{1+\bz\frac{\ln\lambda_-}{\ln b}}\noz.
\end{align}
When $\rho(A^k(x-x_i))< 1$, let $m:=\lfloor1+\lceil \bz\rceil
\frac{\ln\lambda_-}{\ln b}\rfloor$
in \eqref{8e13}. Then we easily find that \eqref{8e14} also holds true.

To deal with $\textrm{J}_2$, for any $i\in\nn$ and
$x\in (x_i+B_{l_i+4\omega})^\com$, let
$$E_{i}^x:=\lf\{y\in\rn:\ b^{l_i+2\omega}
\le\rho(y-x_i)<b^{-2\omega}\rho(x-x_i)\r\}.$$
Then, similarly to \eqref{8e14}, from the vanishing moments of $a_i$,
the fact that $\lceil \bz\rceil-1\le s$ and Taylor's remainder theorem,
it follows that, for any $i\in\nn$ and $z\in x_i+B_{l_i}$,
there exists $\eta_2(z)\in x_i+B_{l_i}$
such that, for any $x\in (x_i+B_{l_i+4\omega})^\com$,
\begin{align*}
\textrm{J}_2
&\ls \int_{E_{i}^x}
\frac{[\rho(y-x_i)]^{\lceil \bz\rceil\frac{\ln\lambda_-}{\ln b}}}
{[\rho(x-x_i)]^{1+\lceil \bz\rceil\frac{\ln\lambda_-}{\ln b}}}\\
&\hs\hs\times \int_{x_i+B_{l_i}}|a_i(z)|\lf|\sum_{|\az|=\lceil \bz\rceil-1}
\frac{[\partial^{\az}\mathcal{K}(y,\cdot)](x_i)
-[\partial^{\az}\mathcal{K}(y,\cdot)](\eta_2(z))}{\az!}(z-x_i)^{\az}\r|\,dz\,dy.
\end{align*}
By this, \eqref{8e4}, \eqref{2e1'} and \eqref{2e2'},
we find that, for any $x\in (x_i+B_{l_i+4\omega})^\com$,
\begin{align}\label{8.1}
\textrm{J}_2
\ls \int_{E_{i}^x}
\frac{[\rho(y-x_i)]^{\lceil \bz\rceil\frac{\ln\lambda_-}{\ln b}}}
{[\rho(x-x_i)]^{1+\lceil \bz\rceil\frac{\ln\lambda_-}{\ln b}}}
\int_{x_i+B_{l_i}}|a_i(z)|
\frac{|x_i+B_{l_i}|^{\bz\frac{\ln\lambda_-}{\ln b}}}
{[\rho(y-x_i)]^{1+\bz\frac{\ln\lambda_-}{\ln b}}}\,dz\,dy.
\end{align}
This, combined with the H\"{o}lder
inequality and the size condition of $a_i$, further implies that
\begin{align*}
\textrm{J}_2
\ls \frac1{\|{\mathbf 1}_{x_i+B_{l_i}}\|_{\lv}}
\lf[\frac{|x_i+B_{l_i}|}
{\rho(x-x_i)}\r]^{1+\bz\frac{\ln\lambda_-}{\ln b}}.
\end{align*}
Thus, for any $x\in (x_i+B_{l_i+4\omega})^\com$, we have
\begin{align}\label{8e15}
\textrm{J}_2
\ls\frac1{\|{\mathbf 1}_{x_i+B_{l_i}}\|_{\lv}}
\lf[M_{\rm HL}\lf({\mathbf 1}_{x_i+B_{l_i}}\r)(x)\r]
^{1+\bz\frac{\ln\lambda_-}{\ln b}}.
\end{align}

For $\textrm{J}_3$, by the vanishing moments of $a_i$,
the fact that $\lceil \bz\rceil-1\le s$ and
Taylor's remainder theorem, we conclude that,
for any $i\in\nn$ and $z\in x_i+B_{l_i}$,
there exists $\eta_3(z)\in x_i+B_{l_i}$
such that, for any $k\in\zz$ and
$x\in (x_i+B_{l_i+4\omega})^\com$,
\begin{align*}
\textrm{J}_3
&\ls b^k\int_{\{y\in\rn:\ \rho(y-x_i)\ge b^{-2\omega}\rho(x-x_i)\}}
\lf|\Phi\lf(A^k(x-y)\r)-\sum_{|\az|\le \lceil \bz\rceil-1}
\frac{\partial^{\az}\Phi(A^k(x-x_i))}{\az!}\lf(A^k(x_i-y)\r)^{\az}\r|\\
&\hs\times \int_{x_i+B_{l_i}}|a_i(z)|\lf|\sum_{|\gamma|= \lceil \bz\rceil-1}
\frac{[\partial^{\gamma}\mathcal{K}(y,\cdot)](x_i)-[\partial^{\gamma}\mathcal{K}(y,\cdot)]
(\eta_3(z))}{\gamma!}(z-x_i)^{\gamma}\r|\,dz\,dy.
\end{align*}
From this, \eqref{8e4}, \eqref{2e1'}, \eqref{2e2'} and
the H\"{o}lder inequality, we deduce that,
for any $m\in\zz_+$, $k\in\zz$ and $x\in (x_i+B_{l_i+4\omega})^\com$,
\begin{align}\label{8e16}
\textrm{J}_3
&\ls \int_{\{y\in\rn:\ \rho(y-x_i)\ge b^{-2\omega}\rho(x-x_i)\}}
\lf|\Phi_k(x-y)\r|\int_{x_i+B_{l_i}}|a_i(z)|
\frac{|x_i+B_{l_i}|^{\bz\frac{\ln\lambda_+}{\ln b}}}
{[\rho(y-x_i)]^{1+\bz\frac{\ln\lambda_+}{\ln b}}}\,dz\,dy\\
&\hs\hs+\int_{\{y\in\rn:\ \rho(y-x_i)\ge b^{-2\omega}\rho(x-x_i)\}}
\lf|b^k\sum_{|\az|\le \lceil \bz\rceil-1}
\frac{\partial^{\az}\Phi(A^k(x-x_i))}{\az!}\lf(A^k(x_i-y)\r)^{\az}\r|\noz\\
&\hs\hs\times\int_{x_i+B_{l_i}}|a_i(z)|
\frac{|x_i+B_{l_i}|^{\bz\frac{\ln\lambda_+}{\ln b}}}
{[\rho(y-x_i)]^{1+\bz\frac{\ln\lambda_+}{\ln b}}}\,dz\,dy\noz\\
&\ls \textrm{J}_{3,1}+\textrm{J}_{3,2},\noz
\end{align}
where the implicit positive constant depends on $m$,
$$\textrm{J}_{3,1}:=\frac{|x_i+B_{l_i}|^{\frac1{2}+\bz\frac{\ln\lambda_+}{\ln b}}}
{[\rho(x-x_i)]^{1+\bz\frac{\ln\lambda_+}{\ln b}}}\|a_i\|_{L^2({\rn})}
\int_{\rn}\lf|\Phi_k(x-y)\r|\,dy$$
and
\begin{align*}
\textrm{J}_{3,2}&:=|x_i+B_{l_i}|^{\frac1{2}+\bz\frac{\ln\lambda_+}
{\ln b}}\|a_i\|_{L^2({\rn})}\\
&\hs\hs\times\sum_{|\az|\le \lceil \bz\rceil-1}
\int_{\{y\in\rn:\ \rho(y-x_i)\ge b^{-2\omega}\rho(x-x_i)\}}
\frac{b^k}{[\rho(A^k(x-x_i))]^m}
\frac{|A^k(y-x_i)|^{|\az|}}{[\rho(y-x_i)]
^{1+\bz\frac{\ln\lambda_+}{\ln b}}}\,dy.
\end{align*}

For the term $\textrm{J}_{3,1}$, by the size condition of $a_i$
and the fact that $\Phi\in L^1(\rn)$, we easily know that,
for any $i\in\nn$ and $x\in (x_i+B_{l_i+4\omega})^\com$,
\begin{align}\label{8e17}
\textrm{J}_{3,1}
\ls\frac1{\|{\mathbf 1}_{x_i+B_{l_i}}\|_{\lv}}
\lf[M_{\rm HL}\lf({\mathbf 1}_{x_i+B_{l_i}}\r)(x)\r]
^{1+\bz\frac{\ln\lambda_-}{\ln b}}.
\end{align}
Moreover, from the size condition of $a_i$,
Definition \ref{2d2}(ii), and
an argument similar to that used in the estimation of
\eqref{8e14}, together with suitably choosing $m\in\zz_+$,
we deduce that,
for any $i\in\nn$ and $x\in (x_i+B_{l_i+4\omega})^\com$,
\begin{align*}
\textrm{J}_{3,2}
\ls\frac1{\|{\mathbf 1}_{x_i+B_{l_i}}\|_{\lv}}
\lf[M_{\rm HL}\lf({\mathbf 1}_{x_i+B_{l_i}}\r)(x)\r]
^{1+\bz\frac{\ln\lambda_-}{\ln b}},
\end{align*}
which, combined with \eqref{8e16} and \eqref{8e17},
implies that, for any $i\in\nn$ and $x\in (x_i+B_{l_i+4\omega})^\com$,
$$\textrm{J}_3
\ls\frac1{\|{\mathbf 1}_{x_i+B_{l_i}}\|_{\lv}}
\lf[M_{\rm HL}\lf({\mathbf 1}_{x_i+B_{l_i}}\r)(x)\r]
^{1+\bz\frac{\ln\lambda_-}{\ln b}}.$$
By this, \eqref{8e12}, \eqref{8e14} and \eqref{8e15},
we find that, for any $i\in\nn$ and $x\in(x_i+B_{l_i+4\omega})^\com$,
\begin{align}\label{8e18}
M_0(T(a_i))(x)
&=\sup_{k\in\zz}\lf|\Phi_k*T(a_i)(x)\r|
\ls\frac1{\|{\mathbf 1}_{x_i+B_{l_i}}\|_{\lv}}
\lf[M_{\rm HL}\lf({\mathbf 1}_{x_i+B_{l_i}}\r)(x)\r]
^{1+\bz\frac{\ln\lambda_-}{\ln b}}.
\end{align}
Then, by \eqref{8e18}, Lemma \ref{3l3}, the fact that
$p_->\frac{\ln b}{\ln b+\beta\ln \lz_-}$
and Lemma \ref{4l3}, we further conclude that
\begin{align*}
\Sigma_2
\ls \lf\|\sum_{i\in\mathbb{N}}\frac{|\lambda_{i}|}
{\|{\mathbf 1}_{x_i+B_{l_i}}\|_{\lv}}
\lf[M_{\rm HL}\lf({\mathbf 1}_{x_i+B_{l_i}}\r)\r]
^{1+\bz\frac{\ln\lambda_-}{\ln b}}\r\|_{\lv}
\ls \lf\|\lf\{\sum_{i\in\nn}
\lf[\frac{|\lz_{i}|{\mathbf 1}_{x_i+B_{l_i}}}
{\|{\mathbf 1}_{x_i+B_{l_i}}\|_{\lv}}\r]^
{\underline{p}}\r\}^{1/\underline{p}}\r\|_{\lv}
\ls 1,
\end{align*}
which, together with \eqref{8e19} and the fact that $\Sigma_1\ls 1$, implies that
\eqref{8e11} holds true and hence completes the proof of Theorem \ref{8t3}.
\end{proof}

\begin{remark}
We point out that the regularity condition \eqref{8e4} exactly reflects
the anisotropy of the homogeneous quasi-norm of the so-considered underlying
space $(\rn,\rho)$, which is quite natural (see also \cite[Remark 6.11]{hlyy}).
Actually, the difference between the
condition \eqref{8e4} and the classical one on the Euclidean space $(\rn,|\cdot|)$ is caused by
the relationship between the homogeneous quasi-norm and the Euclidean norm.
To be precise, in the first inequality of the proof of \eqref{8.1},
to transfer the Euclidean norm $|(z-x_i)^\az|$ therein
into the homogeneous quasi-norm, we have to use the following fact that,
for any multi-index $\az\in\zz_+^n$
with $|\az|=\lceil \bz\rceil-1$ and $z\in x_i+B_{l_i}$,
\begin{equation}\label{8.2}
\lf|(z-x_i)^\az\r|
\le \max\lf\{\lf[\rho(z-x_i)\r]^{(\lceil \bz\rceil-1) \frac{\ln \lz_+}{\ln b}},\,
\lf[\rho(z-x_i)\r]^{(\lceil \bz\rceil-1) \frac{\ln \lz_-}{\ln b}}\r\}
\end{equation}
(see Lemma \ref{3l8}). Therefore, in order to
cancel this quantity appearing on the right-hand side of \eqref{8.2},
we need that $\mathcal{K}$ has the regularity of the version as in \eqref{8e4}.
The same problem appears in the estimations of \eqref{8e16}.
\end{remark}

\noindent\textbf{Acknowledgements}\quad Long Huang would like to express
his deep thanks to Ziyi He for several useful conversations
on Lemma \ref{6l1'}. The authors would also like to thank
both referees for their carefully reading and
many motivating and useful comments which indeed improve the
quality of this article.

\bigskip

\noindent  Long Huang, Dachun Yang (Corresponding author) and Wen Yuan

\medskip

\noindent  Laboratory of Mathematics and Complex Systems
(Ministry of Education of China),
School of Mathematical Sciences, Beijing Normal University,
Beijing 100875, People's Republic of China

\smallskip

\noindent {\it E-mails}:
\texttt{longhuang@mail.bnu.edu.cn} (L. Huang)

\noindent\phantom{{\it E-mails:}}
\texttt{dcyang@bnu.edu.cn} (D. Yang)

\noindent\phantom{{\it E-mails:}}
\texttt{wenyuan@bnu.edu.cn} (W. Yuan)

\bigskip

\noindent Jun Liu

\medskip

\noindent  School of Mathematics,
China University of Mining and Technology,
Xuzhou 221116, Jiangsu, People's Republic of China

\smallskip

\noindent{\it E-mail:}
\texttt{junliu@cumt.edu.cn} (J. Liu)

\end{document}